\documentclass{amsart}
\usepackage{amsmath,amscd,xypic,amssymb,combelow,color,enumitem,graphicx,float}

\usepackage{tikz-cd}
\usepackage{mathtools}
\usepackage{setspace}

\usepackage{mathrsfs}
\usepackage[margin=1 in]{geometry}

\usepackage{xr-hyper}
\usepackage[
pdftex,
bookmarks=true,
colorlinks=true,
debug=true,
pdfnewwindow=true]{hyperref}
\hypersetup{bookmarksdepth = 3}

\newtheorem{theorem}{Theorem}[section]

\newtheorem*{theorem*}{Theorem}
\newtheorem{prop}[theorem]{Proposition}
\newtheorem{lemma}[theorem]{Lemma}
\newtheorem{cor}[theorem]{Corollary}

\newtheorem{definition}[theorem]{Definition}
\newtheorem{remark}[theorem]{Remark}

\theoremstyle{remark}

\numberwithin{equation}{section}

\usepackage{amsthm}
\makeatletter
\def\th@plain{%
  \thm@notefont{}
  \itshape 
}
\def\th@definition{%
  \thm@notefont{}
  \normalfont 
}
\makeatother
\newcommand{\bb}[1]{\mathbb{#1}}

\newcommand{\cal}[1]{\mathcal{#1}}
\renewcommand{\frak}[1]{\mathfrak{#1}}

\newcommand{\qbinom}{\genfrac{[}{]}{0pt}{}}
\usepackage{nicematrix}
\usepackage{accents}

\usepackage{etoolbox}
\usepackage{tikz}
\usepackage{latexsym}
\usepackage{verbatim}
\usepackage{cases}
\usepackage{theoremref}

\DeclareMathOperator{\supp}{supp}

\DeclareMathOperator{\hgt}{ht}

\newcommand{\fg}{\frak{g}}

\newcommand{\fn}{\frak{n}}
\newcommand{\fh}{\frak{h}}

\newcommand{\BZ}{\mathbb{Z}}
\newcommand{\BC}{\mathbb{C}}
\newcommand{\gl}{\mathfrak{gl}}
\newcommand{\ssl}{\mathfrak{sl}}
\newcommand{\sso}{\mathfrak{so}}
\newcommand{\ssp}{\mathfrak{sp}}
\newcommand{\uu}{U_{r,s}(\fg)}

\newcommand\iso{\,\vphantom{j^{X^2}}\smash{\overset{\sim}{\vphantom{\rule{0pt}{0.20em}}\smash{\longrightarrow}}}\,}

\newcommand{\ol}{\overline}
\newcommand{\wtd}{\widetilde}

\newcommand{\bk}{\mathsf{k}}
\newcommand{\bh}{\mathsf{h}}

\begin{document}

\title[Orthogonal Bases for Two-Parameter Quantum Groups]
      {\Large{\textbf{Orthogonal Bases for Two-Parameter Quantum Groups}}}

\author[Ian Martin and Alexander Tsymbaliuk]{Ian Martin and Alexander Tsymbaliuk}

\address{I.M.: Purdue University, Department of Mathematics, West Lafayette, IN, USA}
\email{mart2151@purdue.edu}

\address{A.T.: Purdue University, Department of Mathematics, West Lafayette, IN, USA}
\email{sashikts@gmail.com}

\begin{abstract}
In this note, we construct dual PBW bases of the positive and negative subalgebras of the two-parameter quantum 
groups $U_{r,s}(\mathfrak{g})$ in classical types, as used in~\cite{MT}. Following the ideas of Leclerc~\cite{L} and 
Clark-Hill-Wang~\cite{CHW}, we introduce the two-parameter shuffle algebra and relate it to the subalgebras above. 
We then use the combinatorics of dominant Lyndon words to establish the main results.
\end{abstract}

\maketitle
\tableofcontents


\section{Introduction}\label{sec:intro}


\subsection{Summary}\label{ssec:summary}
\

Let $\fg$ be a simple finite dimensional Lie algebra. Corresponding to any polarization $\Phi = \Phi^{+} \sqcup (-\Phi^{+})$
of the root system of $\fg$, there is a root space decomposition
\begin{equation}\label{eq:simpleLie}
  \fg=\fn^- \oplus \fh \oplus \fn^+ \qquad \mathrm{with} \qquad \fn^{\pm}=\bigoplus_{\gamma\in \Phi^+} \BC\cdot e_{\pm \gamma}.
\end{equation}
The elements $e_{\pm \gamma}$ are called \emph{root vectors}. This induces a decomposition
$U(\fg)=U(\fn^-) \otimes U(\fh) \otimes U(\fn^+)$, and the ordered products in $\{e_{\pm \gamma}\}_{\gamma\in \Phi^+}$
form a basis of $U(\fn^\pm)$ for any total order on $\Phi^+$. In fact, the root vectors can be normalized so that
\begin{equation}\label{eqn:Lie-root-vectors}
  [e_{\alpha},e_{\beta}] = e_{\alpha}e_{\beta}-e_{\beta}e_{\alpha} \in (\BZ \setminus \{0\}) \cdot e_{\alpha+\beta}
  \qquad \mathrm{for\ all} \quad \alpha, \beta \in \Phi^+ \ \ \mathrm{such\ that} \ \ \alpha+\beta \in \Phi^+.
\end{equation}

For each such $\fg$, Drinfeld and Jimbo simultaneously defined the quantum group $U_{q}(\fg)$, which is a quantization
of the universal enveloping algebra of $\fg$. As with $U(\fg)$, the quantum groups possess a triangular decomposition
$U_q(\fg) = U_q(\fn^-) \otimes U_q(\fh) \otimes U_q(\fn^+)$. Furthermore, $U_q(\fn^\pm)$ admit PBW-type bases
\begin{equation*}
  U_q(\fn^\pm) \ =
  \bigoplus_{\gamma_1 \geq \dots \geq \gamma_k \in \Phi^+} \BC(q) \cdot e_{\pm \gamma_1} \dots e_{\pm \gamma_k}
\end{equation*}
formed by the ordered products of \emph{$q$-deformed root vectors} $e_{\pm \gamma} \in U_q(\fn^\pm)$. The latter are classically 
defined via Lusztig's braid group action, which requires one to choose a reduced decomposition of the longest element $w_0$ in 
the Weyl group of $\fg$, and the order $\geq$ on $\Phi^+$ used above is induced by this reduced decomposition.

However, there is also a purely combinatorial approach to the construction of PBW-type bases of $U_q(\fn^\pm)$ that goes
back to the works of Kharchenko, Leclerc, Rosso. To this end, recall Lalonde-Ram's bijection~\cite{LR}:
\begin{equation}\label{eqn:LR-bij-intro}
  \ell \colon \Phi^+ \iso \Big\{\text{standard Lyndon words in}\ I\Big\}.
\end{equation}
Here, the notion of standard Lyndon words intrinsically depends on a fixed total order of the indexing set $I$
of simple roots of $\fg$. Furthermore, this bijection $\ell$ gives rise to a total \emph{lexicographical} order on $\Phi^+$ via:
\begin{equation}\label{eq:lex_order_intro}
  \alpha < \beta \quad \Longleftrightarrow \quad \ell(\alpha) < \ell(\beta) \text{ lexicographically}.
\end{equation}
According to~\cite[Proposition 26]{L}, this order is \emph{convex} and thus corresponds to a reduced decomposition
of $w_0$. This allows one to define $e_{\pm \gamma}$ as iterated $q$-commutators, eliminating Lusztig's braid group action.
Explicitly, one defines the root vectors inductively via (cf.~\eqref{eqn:Lie-root-vectors})
\begin{equation}\label{eq:root_vectors_intro}
  e_\gamma=e_\alpha e_\beta - q^{(\alpha,\beta)}e_\beta e_\alpha \qquad \mathrm{and} \qquad
  e_{-\gamma}=e_{-\beta} e_{-\alpha} - q^{-(\alpha,\beta)} e_{-\alpha} e_{-\beta},
\end{equation}
where $\alpha,\beta$ are determined combinatorially by choosing the longest prefix in the decomposition (cf.~\eqref{eqn:LR-bij-intro})
\begin{equation}\label{eq:cost-factorization_intro}
  \ell(\gamma)=\ell(\alpha)\ell(\beta) \qquad \mathrm{with} \quad \alpha,\beta\in \Phi^+,\ \gamma=\alpha+\beta.
\end{equation}

Although the theory of multiparameter quantum groups goes back to the early 1990s (see e.g.~\cite{AST,R,T}), much of the
current interest in the subject stems from the papers~\cite{BW1, BW2, BW3}, which study the two-parameter quantum group
$U_{r,s}(\gl_n)$ and subsequently give an application to pointed finite dimensional Hopf algebras. In~\cite{BW2}, they developed
the theory of finite dimensional representations in a complete analogy with the one-parameter case, computed the two-parameter
$R$-matrix for the first fundamental $U_{r,s}(\gl_n)$-representation, and used it to establish the Schur-Weyl duality between
$U_{r,s}(\gl_n)$ and a two-parameter Hecke algebra.
These works of Benkart and Witherspoon stimulated an increased interest in the theory of two-parameter quantum groups.
In particular, the definitions of $\uu$ for other classical simple Lie algebras $\fg$ were first given in~\cite{BGH1,BGH2},
where basic results on the structure and representation theory of $\uu$ were also established. Subsequently, these algebras
have been treated case~by~case in multiple papers. For a more uniform treatment and complete references, we refer the reader to~\cite{HP}.

In the companion paper~\cite{MT}, we evaluated the finite and affine $R$-matrices for two-parameter quantum groups
of classical types, associated with the first fundamental representation and the corresponding evaluation modules.
We further presented a factorization of the finite $R$-matrices into ``local factors'', parametrized by $\Phi^+$.
The latter relied on the construction of dual PBW-type bases of the positive and negative subalgebras,
which was announced in~\cite[Theorem 5.12]{MT} without a proof. The main objective of the present note is to provide a proof of this result.
Due to the absence of Lusztig's braid group action on $\uu$ (noted first in~\cite{BGH1}), we use the aforementioned
combinatorial construction of orthogonal bases of the positive and the negative subalgebras $U_{r,s}^{\pm}(\fg) \subset \uu$ through
dominant\footnote{We use the terminology of~\cite{CHW}; they are also called \emph{good} in~\cite{L}, and
coincide with \emph{standard} from~\cite{LR}.} Lyndon words, which goes back to~\cite{L,Ro} in the one-parameter setup,
to~\cite{CHW} in the super case, and to~\cite{BKL} in the two-parameter $A$-type case.
To this end, we follow~\eqref{eq:root_vectors_intro} and define the root vectors $\{e_{\pm \gamma}\}_{\gamma\in \Phi^+}$ iteratively via
(cf.~\eqref{eq:q-root-vectors})
\begin{equation}\label{eq:root_vectors_2param_intro}
  e_\gamma=e_\alpha e_\beta - (\omega'_\beta,\omega_\alpha) e_\beta e_\alpha \qquad \mathrm{and} \qquad
  e_{-\gamma}=e_{-\beta} e_{-\alpha} - (\omega'_\alpha,\omega_\beta)^{-1} e_{-\alpha} e_{-\beta},
\end{equation}
where $\alpha,\beta$ are as in~\eqref{eq:cost-factorization_intro} and the pairing $(\omega'_\lambda,\omega_\mu)$
is as in~\eqref{eq:generators-parity-2}.

Our first main result is Theorem~\ref{thm:main-1}, a key part of which is the following:

\begin{theorem*} 
The products
\begin{equation}\label{eq:ordered_products}
  \left\{\overset{\longleftarrow}{\underset{\gamma \in \Phi^{+}}{\prod}} e_{\pm \gamma}^{m_{\gamma}} \,\Big|\, m_\gamma\in \BZ_{\geq 0}\right\}
     \qquad \mathrm{ordered\ w.r.t.\ lexicographical\ order}\  \eqref{eq:lex_order_intro}
\end{equation}
form bases for $U_{r,s}^{\pm}(\fg)$, which are orthogonal with respect to the Hopf pairing $(\cdot,\cdot)_{H}$ of Proposition~\ref{prop:pairing_2param}.
\end{theorem*}

Moreover, Theorem~\ref{thm:main-1}(2) reduces the evaluation of $(\cdot,\cdot)_{H}$ on the basis elements~\eqref{eq:ordered_products} to just that of 
$(e_{-\gamma},e_{\gamma})_{H}$ for each $\gamma \in \Phi^{+}$. An explicit evaluation of the latter pairing for a particular order on the indexing set $I$ 
is carried out in Theorem~\ref{thm:main-2}, which is our second main result. While the construction of such bases for type $A$ was already presented
(through a slightly different, but equivalent, perspective of Gr\"{o}bner bases) in~\cite{BKL}, the results on their pairing already seem to be new in $A$-type. 
As mentioned in the previous paragraph, these results are integral to our factorization of $R$ in~\cite[\S5]{MT}, since the \emph{reduced $R$-matrix} $\Theta$ 
can be written as an ordered product of deformed exponents of $\frac{e_{-\gamma}\otimes e_{\gamma}}{(e_{-\gamma},e_{\gamma})_H}$, cf.~\cite[Remark 5.14]{MT}. 
As such, together with the computations in~\cite[\S5.4]{MT}, this provides a conceptual proof of~\cite[Theorems 4.4--4.6]{MT}.

Let us briefly outline the strategy of our proof. First, we translate the problem solely into the construction of dual bases of $U^+_{r,s}(\fg)$, 
endowed with a twisted coproduct and a twisted product on its tensor powers, with respect to a different symmetric form $(\cdot,\cdot)$. We then 
introduce a two-parameter shuffle algebra $\cal{F}$ and embed $U^+_{r,s}(\fg)$ into $\cal{F}$. Interpreting the coproduct and the pairing on the 
shuffle side allows us to construct dual bases, which correspond to the ordered products~\eqref{eq:ordered_products}. The explicit computation 
of the nonzero pairing constants for a special order of simple roots is then accomplished by brute force. Finally, pulling back these results 
to the original setup yields the aforementioned Theorems~\ref{thm:main-1}--\ref{thm:main-2}.

We conclude the Summary by noting that two-parameter quantum groups naturally give rise to the quantum knot invariants via the Reshetikhin-Turaev 
type construction, see~\cite[\S4]{FMX} for more details. However, we expect that these invariants are intrinsically related to those attached to the usual 
one-parameter quantum groups, in analogy with~\cite[\S4]{C} who matched $\mathfrak{osp}(1|2n)$ and $\mathfrak{so}(2n+1)$ quantum knot invariants.


\subsection{Outline}\label{ssec:outline}
\

\noindent
The structure of the present paper is as follows: 
\begin{itemize}[leftmargin=0.5cm]

\item[$\bullet$]
In Section~\ref{sec:notation}, we recall the two-parameter quantum groups $U_{r,s}(\fg)$ for simple finite dimensional 
Lie algebras~$\fg$, see Definition~\ref{def:general_2param}, as well as the Hopf pairing for those,  
see Proposition~\ref{prop:pairing_2param}. We also construct the analogues of the Cartan involution 
and the bar involution, see Proposition~\ref{prop:extra_structures}. 
Finally, we compute the graded dimensions and establish the non-degeneracy of the Hopf pairing, 
see Propositions~\ref{prop:graded-dimension} and~\ref{prop:nondegeneracy-Hopf}.

\item[$\bullet$]
In Section~\ref{sec:bilinear_forms}, we study several other pairings and ultimately relate them to the Hopf pairing 
of Proposition~\ref{prop:pairing_2param}, see Theorem~\ref{thm:pairing_comparison}. This is needed to replace 
the Hopf subalgebra $U^\geq_{r,s}(\fg)$ with $U^+_{r,s}(\fg)$ (eliminating Cartan elements) at the cost of changing 
the product structure on the tensor powers, and also modifying the coproduct and the pairing. The latter is essential 
for the combinatorial constructions in Sections~\ref{sec:shuffle}--\ref{sec:root_vectors}.

\item[$\bullet$]
In Section~\ref{sec:shuffle}, we introduce the two-parameter shuffle algebra $(\cal{F},*)$ and relate it to 
the positive subalgebra $U^+_{r,s}(\fg)$, see~\eqref{eq:shuffle_product_iterative}--\eqref{eq:e_ab} and 
Propositions~\ref{prop:shuffle_embedding},~\ref{prop:psi_equals_upsilon},~\ref{prop:image_description}. 
We also provide a shuffle interpretation of the bar involution and twisted coproduct, see 
Proposition~\ref{prop:tau_and_bar}, Corollary~\ref{cor:auto_matching}, and Proposition~\ref{prop:shuffle_coproduct}.

\item[$\bullet$]
In Section~\ref{sec:orthogonal_bases}, following~\cite[Sections 4-5]{CHW} (which in turn is largely based on~\cite{L}),  
we recall the notions of dominant and Lyndon words as well as introduce the bracketing of words. We use the latter to 
introduce the \emph{Lyndon basis} $\{R_w\}_{w\in \cal{W}^+}$ and its closely related versions $\{\wtd{R}_w\}_{w\in \cal{W}^+}$, 
$\{\bar{R}_w\}_{w\in \cal{W}^+}$, all parametrized by the set $\cal{W}^+$ of dominant words. Our main result is
Theorem~\ref{thm:dual_bases}, which establishes the orthogonality of the last two bases and expresses all nonzero
pairings through the pairings $\{(R_{\ell},\bar{R}_\ell)\}_{\ell\in \cal{L}^+}$, parametrized by the set $\cal{L}^+$ 
of dominant Lyndon words (which is in bijection~\eqref{eqn:LR-bij-intro} with the set $\Phi^+$ of positive roots). 

\item[$\bullet$]
In Section~\ref{sec:root_vectors}, we explicitly compute $R_\ell$ (which are shuffle incarnations of the quantum root vectors 
in $U^+_{r,s}(\fg)$) and the corresponding pairing $(R_{\ell},\bar{R}_{\ell})$ for each dominant Lyndon word $\ell\in \cal{L}^+$, 
for the special order $1 < \dots < n$ on the alphabet $I=\{1,\dots,n\}$. We treat each type separately, similarly to~\cite[\S6]{CHW}.

\item[$\bullet$]
In Section~\ref{sec:main}, we combine Theorem~\ref{thm:pairing_comparison}, Theorem~\ref{thm:dual_bases}, and the explicit 
calculations of Section~\ref{sec:root_vectors} to prove the main results of this paper: Theorems~\ref{thm:main-1} 
and~\ref{thm:main-2}. This establishes PBW-type bases of $U^+_{r,s}(\fg)$ and $U^-_{r,s}(\fg)$ dual with respect to the 
Hopf pairing of Proposition~\ref{prop:pairing_2param}, which was announced in our earlier work~\cite[Theorem 5.12]{MT} 
without a proof, and used there to factorize the corresponding finite $R$-matrices.

\item[$\bullet$]
In Appendix~\ref{sec:app_pairing_formulas}, we evaluate $(R_{\ell},\bar{R}_{\ell})$ for any order on the alphabet $I$ given that 
the first letter of $\ell$ occurs exactly once, see Theorem~\ref{thm:pairing_formula}, which is crucially based on an interesting 
combinatorial Lemma~\ref{lem:word_description}.

\end{itemize}


\subsection{Acknowledgement}\label{ssec:acknowl}
\

This note represents a part of the year-long REU project at Purdue University; we are grateful to Purdue University 
for support. A.T.\ is deeply indebted to Andrei Negu\c{t} for numerous stimulating discussions over the years and 
sharing the beautiful combinatorial features of quantum groups, to Sarah Witherspoon for a correspondence on 
two-parameter quantum groups, and to Weiqiang Wang for a correspondence on~\cite{CHW}. 
A.T.\ is extremely grateful to IHES (Bures-sur-Yvette, France) for the hospitality and wonderful working
conditions in the summer 2025, where the final version of this paper was prepared. 
We are very grateful to the referees for their useful suggestions that improved the exposition.

The work of both authors was partially supported by an NSF Grant DMS-$2302661$.


\section{Notations and Definitions}\label{sec:notation}

In this Section, we recall the notion of two-parameter quantum groups $U_{r,s}(\fg)$ for simple finite~dimensional 
Lie algebras~$\fg$, the Hopf algebra structure and the Hopf pairing on those, and finally construct several important 
(anti)automorphisms. We refer the interested reader to~\cite[\S1.1,~\S2]{MT} for a full list of references.


\subsection{Two-parameter quantum groups}
\
   
Let $E$ be a Euclidean space with a positive-definite symmetric bilinear form $(\cdot ,\cdot)$. Let $\Phi \subset E$ be an irreducible 
reduced root system with an ordered set of simple roots $\Pi = \{\alpha_{1},\ldots ,\alpha_{n}\}$, and let $\fg$ be the corresponding 
complex simple Lie algebra. We set $\fn^{\pm} = \bigoplus_{\alpha \in \Phi^{+}}\fg_{\pm \alpha}$, 
where $\fg_{\alpha}$ denotes the root space of $\fg$ corresponding to $\alpha \in \Phi$, and $\Phi^+$ denotes the set of positive
roots of $\Phi$, see~\eqref{eq:simpleLie}. Let $C = (a_{ij})_{i,j = 1}^{n}$ be the Cartan matrix of $\fg$, with entries given explicitly 
by $a_{ij} = \frac{2(\alpha_{i},\alpha_{j})}{(\alpha_{i},\alpha_{i})}$, and set $d_{i} = \frac{1}{2}(\alpha_{i},\alpha_{i})$, 
where $(\cdot,\cdot)$ is normalized so that the short roots have square length $2$. The root and weight lattices  
of $\fg$ will be denoted by $Q$ and $P$,~respectively:
\begin{equation*}
  \bigoplus_{i=1}^n \BZ \alpha_i=Q\subset P=\bigoplus_{i=1}^n \BZ \varpi_i
  \qquad \mathrm{with}\quad (\alpha_i,\varpi_j)=d_i\delta_{ij}.
\end{equation*}

Having fixed above the order on the set of simple roots $\Pi$, we consider the (modified) Ringel bilinear form 
$\langle \cdot ,\cdot \rangle$ on $Q$, such that (unless $\{i,j\}=\{n-1,n\}$ in type $D_n$) we have:
\begin{equation*}
  \langle \alpha_{i},\alpha_{j}\rangle =
  \begin{cases}
     d_{i}a_{ij} & \text{if}\ \ i < j \\
     d_{i} & \text{if}\ \ i = j \\
     0 & \text{if}\ \ i > j
  \end{cases} \,,
\end{equation*}
while in the remaining case of $D_n$-type system, we set (cf.~\eqref{eq:root_notation_D} and the paragraph preceding it):
\begin{equation*}
  \langle \alpha_{n-1},\alpha_{n} \rangle =
    \langle \varepsilon_{n-1}-\varepsilon_n,\varepsilon_{n-1}+\varepsilon_n\rangle = -1, \qquad
  \langle \alpha_{n},\alpha_{n-1}\rangle =
    \langle \varepsilon_{n-1}+\varepsilon_n,\varepsilon_{n-1}-\varepsilon_n \rangle = 1.
\end{equation*}
We note that $(\mu,\nu) = \langle \mu,\nu \rangle + \langle \nu,\mu \rangle$ for any $\mu,\nu \in Q$.

We also need the following two-parameter analogues of $q$-integers, $q$-factorials, and $q$-binomial coefficients:
\begin{equation*}
  [m]_{r,s} = \frac{r^{m} - s^{m}}{r - s} = r^{m-1} + r^{m-2}s + \ldots + rs^{m-2} + s^{m-1}
  \qquad \mathrm{for\ all} \quad  m\in \BZ_{\geq 0},
\end{equation*}
\begin{equation*}
  [m]_{r,s}! = [m]_{r,s} [m-1]_{r,s} \cdots [1]_{r,s}  \qquad \mathrm{for} \quad m>0,
  \qquad [0]_{r,s}!=1,
\end{equation*}
and
\begin{equation*}
  \qbinom{m}{k}_{r,s} = \frac{[m]_{r,s}!}{[m - k]_{r,s}![k]_{r,s}!}
  \qquad \mathrm{for\ all} \quad 0\leq k\leq m.
\end{equation*}
Finally, we also define
\begin{equation}\label{eq:rs_gamma}
\begin{split}
  & r_{\gamma} = r^{(\gamma,\gamma)/2},\ \qquad s_{\gamma} = s^{(\gamma,\gamma)/2}
    \qquad \mathrm{for\ all} \quad \gamma\in \Phi,\\
  & r_{i} = r_{\alpha_i} = r^{d_{i}},\qquad s_{i} = s_{\alpha_i} = s^{d_{i}}
    \qquad \mathrm{for\ all} \quad 1\leq i\leq n.
\end{split}
\end{equation}

We now recall the definition of the \textbf{two-parameter quantum group} of $\fg$:

\begin{definition}\label{def:general_2param}
The two-parameter quantum group $U_{r,s}(\fg)$ of a simple Lie algebra $\fg$ is the associative $\BC(r,s)$-algebra 
generated by $\{e_{i},f_{i},\omega_{i}^{\pm 1},(\omega_{i}')^{\pm 1}\}_{i=1}^{n}$ with the following defining relations 
(for all $1\leq i,j\leq n$):
\begin{equation}\label{eq:R1}
   [\omega_i,\omega_j]=[\omega_i,\omega'_j]=[\omega'_i,\omega'_j]=0, \qquad
   \omega_{i}^{\pm 1}\omega_{i}^{\mp 1} = 1 = (\omega_{i}')^{\pm 1}(\omega_{i}')^{\mp 1},
\end{equation}
\begin{equation}\label{eq:R2}
   \omega_{i}e_{j} = r^{\langle \alpha_{j},\alpha_{i}\rangle}s^{-\langle \alpha_{i},\alpha_{j}\rangle}e_{j}\omega_{i}, \qquad
   \omega_{i}f_{j} = r^{-\langle \alpha_{j},\alpha_{i}\rangle}s^{\langle \alpha_{i},\alpha_{j}\rangle}f_{j}\omega_{i},
\end{equation}
\begin{equation}\label{eq:R3}
   \omega_{i}'e_{j} = r^{-\langle \alpha_{i},\alpha_{j}\rangle}s^{\langle \alpha_{j},\alpha_{i}\rangle}e_{j}\omega_{i}', \qquad
   \omega_{i}'f_{j} = r^{\langle \alpha_{i},\alpha_{j}\rangle}s^{-\langle \alpha_{j},\alpha_{i}\rangle}f_{j}\omega_{i}',
\end{equation}
\begin{equation}\label{eq:R4}
    e_{i}f_{j} - f_{j}e_{i} = \delta_{ij}\frac{\omega_i-\omega'_i}{r_{i} - s_{i}},
\end{equation}
and quantum $(r,s)$-Serre relations
\begin{equation}\label{eq:R5}
\begin{split}
  & \sum_{k = 0}^{1 - a_{ij}} (-1)^k \qbinom{1 - a_{ij}}{k}_{r_{i},s_{i}} (r_{i}s_{i})^{\frac{1}{2}k(k-1)}
    (rs)^{k\langle \alpha_{j},\alpha_{i}\rangle}e_{i}^{1 - a_{ij} - k}e_{j}e_{i}^{k} = 0
    \qquad \mathrm{for}\ i\ne j, \\
  & \sum_{k = 0}^{1 - a_{ij}} (-1)^k \qbinom{1 - a_{ij}}{k}_{r_{i},s_{i}} (r_{i}s_{i})^{\frac{1}{2}k(k-1)}
    (rs)^{k\langle \alpha_{j},\alpha_{i}\rangle}f_{i}^{k}f_{j}f_{i}^{1-a_{ij} - k} = 0 
    \qquad \mathrm{for}\  i\ne j.
\end{split}
\end{equation}
\end{definition}

We note that the algebra $U_{r,s}(\fg)$ admits a $Q$-grading, defined on the generators via:
\begin{equation*}
  \deg(e_{i})=\alpha_i,\quad \deg(f_i)=-\alpha_i, \quad
  \deg(\omega_{i})=0, \quad \deg(\omega_{i}')=0 \qquad \mathrm{for\ all} \quad 1\leq i\leq n.
\end{equation*}
For $\mu\in Q$, let $U_{r,s}(\fg)_{\mu}$ (or simply $(U_{r,s})_{\mu}$) denote the degree $\mu$ component of 
$U_{r,s}(\fg)$ under this $Q$-grading. We shall also need several subalgebras of $U_{r,s}(\fg)$:
\begin{itemize}

\item 
the ``positive'' subalgebra $U_{r,s}^{+}=U_{r,s}^{+}(\fg)$, generated by $\{e_{i}\}_{i=1}^{n}$,

\item 
the ``negative'' subalgebra $U_{r,s}^{-}= U_{r,s}^{-}(\fg)$, generated by $\{f_{i}\}_{i=1}^{n}$,

\item 
the ``Cartan'' subalgebra $U_{r,s}^{0}= U_{r,s}^{0}(\fg)$, generated by $\{\omega_{i}^{\pm 1},(\omega'_{i})^{\pm 1}\}_{i=1}^{n}$,

\item 
the ``non-negative subalgebra'' $U_{r,s}^{\ge}= U_{r,s}^{\ge}(\fg)$, generated by $\{e_{i},\omega_{i}^{\pm 1}\}_{i=1}^{n}$,

\item 
the ``non-positive subalgebra'' $U_{r,s}^{\le}= U_{r,s}^{\le}(\fg)$, generated by $\{f_{i},(\omega'_{i})^{\pm 1}\}_{i=1}^{n}$.

\end{itemize}
Evoking~\eqref{eq:R1}, for any $\mu=\sum_{i=1}^{n} c_{i}\alpha_{i}\in Q$, we define 
$\omega_\mu,\omega'_\mu\in U_{r,s}^{0}(\fg)$ via:
\begin{equation*}
  \omega_{\mu} = \omega_{1}^{c_{1}}\omega_{2}^{c_{2}} \cdots \omega_{n}^{c_{n}}, \qquad
  \omega'_\mu = (\omega_{1}')^{c_{1}}(\omega_{2}')^{c_{2}} \cdots (\omega_{n}')^{c_{n}}.
\end{equation*}

Finally, the algebra $U_{r,s}(\fg)$ has a Hopf algebra structure, where the coproduct $\Delta$, 
counit $\epsilon$, and antipode $S$ are defined on the generators by the following formulas:
\begin{align*}
  &\Delta(\omega_{i}^{\pm 1}) = \omega_{i}^{\pm 1} \otimes \omega_{i}^{\pm 1}
  & &\epsilon(\omega_{i}^{\pm 1}) = 1
  & &S(\omega_{i}^{\pm 1}) = \omega_{i}^{\mp 1} \\
  &\Delta((\omega_{i}')^{\pm 1}) = (\omega_{i}')^{\pm 1} \otimes (\omega_{i}')^{\pm 1}
  & & \epsilon((\omega_{i}')^{\pm 1}) = 1
  & & S((\omega_{i}')^{\pm 1}) = (\omega_{i}')^{\mp 1} \\
  &\Delta(e_{i}) = e_{i} \otimes 1 + \omega_{i} \otimes e_{i}
  & &\epsilon(e_{i}) = 0
  & &S(e_{i}) = -\omega_{i}^{-1}e_{i} \\
  &\Delta(f_{i}) = 1 \otimes f_{i} + f_{i} \otimes \omega_{i}'
  & &\epsilon(f_{i}) = 0
  & &S(f_{i}) = -f_{i}(\omega_{i}')^{-1}
\end{align*}
We note that 
\begin{equation}\label{eq:coprod-1}
  \Delta(x) \in x\otimes 1 +  
    \bigoplus_{0<\nu<\mu} U_{r,s}^{+}(\fg)_{\mu - \nu}\omega_{\nu} \otimes U_{r,s}^{+}(\fg)_{\nu} + \omega_{\mu} \otimes x,
\end{equation}
\begin{equation}\label{eq:coprod-2}
  \Delta(y) \in y \otimes \omega_{\mu}' + 
    \bigoplus_{0<\nu<\mu} U_{r,s}^{-}(\fg)_{-\nu} \otimes U_{r,s}^{-}(\fg)_{-(\mu - \nu)}\omega_{\nu}' + 1 \otimes y,
\end{equation}
for any $x\in U_{r,s}^{+}(\fg)_{\mu}$ and $y \in U_{r,s}^{-}(\fg)_{-\mu}$. 
Here, we use the standard order $\leq$ on the root lattice $Q$:
\begin{equation*}
  \nu\leq \mu \, \Longleftrightarrow \,
  \mu-\nu \in Q^+ ,
\end{equation*}
where $Q^{+}=\bigoplus_{i=1}^n \BZ_{\geq 0} \alpha_i$ is the positive cone of the root lattice $Q$.

For any $\lambda=\sum_{i=1}^n c_i\alpha_i\in Q^+$, we define its height as $\hgt(\lambda)=\sum_{i=1}^n c_i$.


\subsection{Hopf pairing}
\

In this Subsection, we recall the Hopf pairing on $U_{r,s}(\fg)$, which allows one to realize $U_{r,s}(\fg)$ 
as a Drinfeld double of its Hopf subalgebras $U_{r,s}^{\le}(\fg)$, $U_{r,s}^{\ge}(\fg)$. This pairing is also 
crucial to the main results of this paper.

\begin{prop}\label{prop:pairing_2param}
There exists a unique non-degenerate bilinear pairing
\begin{equation}\label{eq:Hopf-parity}
  (\cdot,\cdot)_{H}\colon U_{r,s}^{\le}(\fg) \times U_{r,s}^{\ge}(\fg) \longrightarrow \BC(r,s)
\end{equation}
satisfying the following structural properties
\begin{equation*}
  (yy',x)_{H} = (y \otimes y',\Delta(x))_{H}, \qquad (y, xx')_{H} = (\Delta(y),x' \otimes x)_{H}
  \qquad \forall\, x,x' \in U_{r,s}^{\ge}(\fg),\ y,y' \in U_{r,s}^{\le}(\fg),
\end{equation*}
where $(x'\otimes x'',y'\otimes y'')_H=(x',y')_H (x'',y'')_H$, as well as being given on the generators by:
\begin{equation*}
  (f_{i},\omega_{j})_{H} = 0, \qquad (\omega_{i}', e_{i})_{H} = 0, \qquad
  (f_{i},e_{j})_{H} = \delta_{ij}\frac{1}{s_{i} - r_{i}} \qquad \mathrm{for\ all} \quad 1\leq i,j\leq n ,
\end{equation*}
\begin{equation}\label{eq:generators-parity-2}
  (\omega_{\lambda}',\omega_{\mu})_{H} = r^{\langle \lambda,\mu\rangle}s^{-\langle \mu,\lambda \rangle}
  \qquad \mathrm{for\ all} \quad \lambda,\mu\in Q.
\end{equation}
\end{prop}

The above pairing is clearly homogeneous of degree zero with respect to the above $Q$-grading:
\begin{equation}\label{eq:pairing-orthogonal}
  (y,x)_{H}=0 \quad \mathrm{for} \quad
  x\in U_{r,s}^{\ge}(\fg)_{\mu},\ y\in U_{r,s}^{\le}(\fg)_{-\nu}
  \quad \mathrm{with} \quad \mu\ne \nu \in Q^+.
\end{equation}

\begin{remark} 
We shall provide a careful proof of the non-degeneracy of $(\cdot,\cdot)_H$ in Proposition~\ref{prop:nondegeneracy-Hopf}.   
\end{remark}

Using~\eqref{eq:coprod-1}, we may define linear maps 
$p_{i},p_{i}'\colon (U_{r,s}^{+})_{\mu} \to (U_{r,s}^{+})_{\mu - \alpha_{i}}$ for any $\mu\in Q^+$ via 
\begin{align}\label{eq:p-maps}
  \Delta(x) &= x \otimes 1 + \sum_{i = 1}^{n}p_{i}(x)\omega_{i} \otimes e_{i} + \ldots + 
    \sum_{i = 1}^{n}e_{i}\omega_{\mu - \alpha_{i}} \otimes p_{i}'(x) + \omega_{\mu} \otimes x,
\end{align}
which satisfy (and are uniquely determined by) $p_{i}(1) = p_{i}'(1) = 0$, $p_{i}(e_{j}) = p_{i}'(e_{j}) = \delta_{ij}$, 
and the following analogues of the Leibniz rule:
\begin{equation}\label{eq:+derivation}
\begin{split}
  p_{i}(xx') &= xp_{i}(x') + (\omega_{\deg(x')}',\omega_{i})_H \cdot p_{i}(x)x',  \\
  p_{i}'(xx') &= p_{i}'(x)x' + (\omega_{i}',\omega_{\deg(x)})_H \cdot xp_{i}'(x'),
\end{split}
\end{equation}
for all homogeneous $x,x' \in U_{r,s}^{+}$. Combining~\eqref{eq:pairing-orthogonal} with 
Proposition~\ref{prop:pairing_2param}, we obtain the following result:

\begin{lemma}\label{prop:+p_adjoint} 
For any $x \in U_{r,s}^{+}$ and $y \in U_{r,s}^{-}$, we have 
\begin{equation*}
  (f_{i}y,x)_H = \frac{1}{s_{i} - r_{i}}(y,p_{i}'(x))_H \qquad \text{and} \qquad 
  (yf_{i},x)_H = \frac{1}{s_{i} - r_{i}}(y,p_{i}(x))_H.
\end{equation*}
\end{lemma}

Likewise, using~\eqref{eq:coprod-2}, we define linear maps 
$p_{i},p_{i}'\colon (U_{r,s}^{-})_{-\mu} \to (U_{r,s}^{-})_{-(\mu - \alpha_{i})}$ for any $\mu\in Q^+$ via 
\begin{align*}
  \Delta(y) &= y \otimes \omega_{\mu}' + \sum_{i = 1}^{n} p_{i}(y) \otimes f_{i}\omega_{\mu - \alpha_{i}}' + \ldots + 
    \sum_{i = 1}^{n} f_{i} \otimes p_{i}'(y)\omega_{i}' + 1 \otimes y,
\end{align*}
which satisfy (and are uniquely determined by) $p_{i}(1) = p_{i}'(1) = 0$, $p_{i}(f_{j}) = p_{i}'(f_{j}) = \delta_{ij}$, 
and the following analogues of the Leibniz rule:
\begin{equation}\label{eq:-derivation}
\begin{split}
  p_{i}(yy') &= p_{i}(y)y' + (\omega_{-\deg(y)}',\omega_{i})_H\cdot yp_{i}(y'), \\
  p_{i}'(yy') &= yp_{i}'(y') + (\omega_{i}',\omega_{-\deg(y')})_H\cdot p_{i}'(y)y',
\end{split}
\end{equation}
for all homogeneous $y,y' \in U_{r,s}^{-}$. As above, they are related to the Hopf pairing of Proposition~\ref{prop:pairing_2param} via:

\begin{lemma}\label{prop:-p_adjoint} 
For any $y \in U_{r,s}^{-}$ and $x \in U_{r,s}^{+}$, we have
\begin{equation*}
  (y,e_{i}x)_H = \frac{1}{s_{i} - r_{i}}(p_{i}(y),x)_H \qquad \text{and} \qquad 
  (y,xe_{i})_H = \frac{1}{s_{i} - r_{i}}(p_{i}'(y),x)_H.
\end{equation*}
\end{lemma}

Since we will frequently use the restriction of $(\cdot ,\cdot)_{H}$ to the Cartan subalgebra $U_{r,s}^{0}$ 
throughout the paper, we shall denote it simply by $(\cdot,\cdot)$ for brevity:
\begin{equation}\label{eq:parity-convention}
  (y,x) = (y,x)_H \qquad \mathrm{for\ any} \quad y,x\in  U_{r,s}^{0}. 
\end{equation}
Let us present explicit formulas for the latter in each of the classical types. To this end, we use the following 
standard embeddings of the classical-type root systems in Euclidean space:
\begin{itemize}

\item
\emph{$A_{n}$-type} (corresponding to $\fg\simeq \ssl_{n+1}$).

Let $\{\varepsilon_{i}\}_{i=1}^{n+1}$ be an orthonormal basis of $\bb{R}^{n+1}$. Then, we have
\begin{equation*}
\begin{split}
  & \Phi_{A_{n}} = \big\{\varepsilon_{i} - \varepsilon_{j} \,\big|\, 1\leq i\ne j\leq n+1 \big\} \subset \bb{R}^{n+1}, \\
  & \Pi_{A_{n}} = \big\{ \alpha_{i} = \varepsilon_{i} - \varepsilon_{i + 1} \big\}_{i=1}^{n}.
\end{split}
\end{equation*}
We shall use the following notation for the set of positive roots $\Phi^+$ in $\Phi_{A_{n}}$:
\begin{equation}\label{eq:root_notation_A}
  \gamma_{ij} = \alpha_{i} + \dots + \alpha_{j} \qquad \text{for} \quad 1 \le i \le j \le n.
\end{equation}

\item
\emph{$B_{n}$-type} (corresponding to $\fg\simeq \sso_{2n+1}$).

Let $\{\varepsilon_{i}\}_{i=1}^{n}$ be an orthogonal basis of $\bb{R}^{n}$ with $(\varepsilon_{i},\varepsilon_{i})=2$ 
for all $i$. Then, we have
\begin{equation*}
\begin{split}
  & \Phi_{B_{n}} = 
    \big\{ \pm \varepsilon_{i} \pm \varepsilon_{j} \,\big|\, 1 \le i < j \le n \big\} \cup
    \big\{ \pm \varepsilon_{i} \,\big|\, 1\leq i\leq n \big\} \subset \bb{R}^{n}, \\
  & \Pi_{B_{n}} = \big\{ \alpha_{i} = \varepsilon_{i} - \varepsilon_{i + 1} \big\}_{i=1}^{n-1} \cup 
    \big\{ \alpha_{n} = \varepsilon_{n} \big\}.
\end{split}
\end{equation*}
We shall use the following notation for the set of positive roots $\Phi^+$ in $\Phi_{B_{n}}$:
\begin{equation}\label{eq:root_notation_B}
\begin{split}
  & \gamma_{ij} = \alpha_{i} + \dots + \alpha_{j} \qquad \text{for} \quad 1 \le i \le j \le n, \\
  & \beta_{ij} = \alpha_{i} + \dots + \alpha_{j - 1} + 2\alpha_{j} + \dots + 2\alpha_{n} 
    \qquad \text{for} \quad 1 \le i < j \le n.
\end{split}
\end{equation}

\item
\emph{$C_{n}$-type} (corresponding to $\fg\simeq \ssp_{2n}$).

Let $\{\varepsilon_{i}\}_{i=1}^{n}$ be an orthonormal basis of $\bb{R}^{n}$. Then, we have
\begin{equation*}
\begin{split}
  & \Phi_{C_{n}} = \big\{ \pm \varepsilon_{i} \pm \varepsilon_{j} \,\big|\, 1 \le i < j \le n \big\} \cup 
    \big\{ \pm 2\varepsilon_{i} \,\big|\, 1\leq i\leq n \big\} \subset \bb{R}^{n}, \\
  & \Pi_{C_{n}} = \big\{ \alpha_{i} = \varepsilon_{i} - \varepsilon_{i + 1} \big\}_{i=1}^{n-1}\cup 
    \big\{ \alpha_{n} = 2\varepsilon_{n} \big\}.
\end{split}
\end{equation*}
We shall use the following notation for the set of positive roots $\Phi^+$ in $\Phi_{C_{n}}$:
\begin{equation}\label{eq:root_notation_C}
\begin{split}
  & \gamma_{ij} = \alpha_{i} + \dots + \alpha_{j} \qquad \text{for} \quad 1 \le i \le j \le n, \\
  & \beta_{ij} = \alpha_{i} + \dots + \alpha_{j - 1} + 2\alpha_{j} + \dots + 2\alpha_{n-1} + \alpha_{n} 
    \qquad \text{for} \quad 1 \le i \le j < n.
\end{split}
\end{equation}

\item
\emph{$D_{n}$-type} (corresponding to $\fg\simeq \sso_{2n}$).

Let $\{\varepsilon_{i}\}_{i=1}^{n}$ be an orthonormal basis of $\bb{R}^{n}$. Then, we have
\begin{equation*}
\begin{split}
  & \Phi_{D_{n}} = \big\{\pm \varepsilon_{i} \pm \varepsilon_{j} \,\big|\, 1 \le i < j \le n \big\} \subset \bb{R}^{n}, \\
  & \Pi_{D_{n}} = \big\{ \alpha_{i} = \varepsilon_{i} - \varepsilon_{i + 1} \big\}_{i=1}^{n-1} \cup 
    \big\{ \alpha_{n} = \varepsilon_{n-1} + \varepsilon_{n} \big\}.
\end{split}
\end{equation*}
We shall use the following notation for the set of positive roots $\Phi^+$ in $\Phi_{D_{n}}$:
\begin{equation}\label{eq:root_notation_D}
\begin{split}
  & \gamma_{ij} = \alpha_{i} + \dots + \alpha_{j} \qquad \text{for} \quad 1 \le i \le j < n, \\
  & \beta_{in} = \alpha_{i} + \dots + \alpha_{n-2} + \alpha_{n} \qquad \text{for} \quad 1 \le i < n, \\
  & \beta_{i,n-1} = \alpha_{i} + \dots + \alpha_{n} \qquad \text{for} \quad 1 \le i < n-1, \\
  & \beta_{ij} = \alpha_{i} + \dots + \alpha_{j-1} + 2\alpha_{j} + \dots + 2\alpha_{n-2} + \alpha_{n-1} + \alpha_{n} 
    \qquad \text{for} \quad 1 \le i < j < n-1.
\end{split}
\end{equation}

\end{itemize}
Then we have the following explicit formulas for the pairing of Cartan elements, where 
$\lambda = \sum_{i = 1}^{n} c_{i}\alpha_{i} \in Q$:
\begin{itemize}

\item
\emph{$A_{n}$-type}
\begin{align*}
\begin{split}
  (\omega_{\lambda}',\omega_{i}) &= r^{( \varepsilon_{i},\lambda )}s^{( \varepsilon_{i +1},\lambda)}, \\
  (\omega_{i}',\omega_{\lambda}) &= r^{-(\varepsilon_{i + 1},\lambda )}s^{-( \varepsilon_{i},\lambda)}.
\end{split}
\end{align*}

\item
\emph{$B_{n}$-type}
\begin{align*}
\begin{split}
  (\omega_{\lambda}',\omega_{i}) &=
  \begin{cases}
    r^{( \varepsilon_{i},\lambda )}s^{( \varepsilon_{i + 1},\lambda )} & \mathrm{if}\ \ 1\leq i<n  \\
    r^{( \varepsilon_{n},\lambda ) }(rs)^{-c_{n}} & \mathrm{if}\ \ i = n
  \end{cases} \,, \\
  (\omega_{i}',\omega_{\lambda}) &=
  \begin{cases}
    r^{-( \varepsilon_{i + 1},\lambda)}s^{-( \varepsilon_{i},\lambda)} & \mathrm{if}\ \ 1\leq i<n \\
    s^{-( \varepsilon_{n},\lambda ) }(rs)^{c_{n}} & \mathrm{if}\ \ i = n
  \end{cases} \,.
\end{split}
\end{align*}

\item
\emph{$C_{n}$-type}
\begin{align*}
\begin{split}
  (\omega_{\lambda}',\omega_{i}) &=
  \begin{cases}
    r^{( \varepsilon_{i},\lambda) }s^{( \varepsilon_{i + 1},\lambda )} &  \mathrm{if}\ \ 1\leq i<n \\
    r^{2( \varepsilon_{n},\lambda ) }(rs)^{-2c_{n}} & \mathrm{if}\ \ i=n
  \end{cases} \,, \\
  (\omega_{i}',\omega_{\lambda}) &=
  \begin{cases}
    r^{-( \varepsilon_{i + 1},\lambda )}s^{-( \varepsilon_{i},\lambda)} & \mathrm{if}\ \ 1\leq i<n \\
    s^{-2( \varepsilon_{n},\lambda )}(rs)^{2c_{n}} & \mathrm{if}\ \ i = n
  \end{cases} \,.
\end{split}
\end{align*}

\item
\emph{$D_{n}$-type}
\begin{align*}
\begin{split}
  (\omega_{\lambda}',\omega_{i}) &=
  \begin{cases}
    r^{( \varepsilon_{i},\lambda )}s^{( \varepsilon_{i + 1},\lambda )} & \mathrm{if}\ \ 1\leq i<n \\
    r^{( \varepsilon_{n-1},\lambda)}s^{-(\varepsilon_{n},\lambda)}(rs)^{-2c_{n - 1}} & \mathrm{if}\ \ i = n
  \end{cases} \,, \\
  (\omega_{i}',\omega_{\lambda}) &=
  \begin{cases}
    r^{-( \varepsilon_{i + 1},\lambda )}s^{-(\varepsilon_{i},\lambda )} & \mathrm{if}\ \ 1\leq i<n \\
    r^{( \varepsilon_{n},\lambda )}s^{-(\varepsilon_{n-1},\lambda )}(rs)^{2c_{n-1}} & \mathrm{if}\ \ i = n
\end{cases} \,.
\end{split}
\end{align*}

\end{itemize}


\subsection{(Anti)automorphisms}
\

Finally, we need to introduce several additional structures on $U_{r,s}(\fg)$ that will be used later.

\begin{prop}\label{prop:extra_structures}
(1) There is a unique $\BC(r,s)$-algebra anti-automorphism $\varphi \colon U_{r,s}(\fg) \to U_{r,s}(\fg)$ 
(called the \textbf{Cartan involution}) satisfying 
\begin{equation*}
  \varphi(e_{i}) = f_{i}, \qquad  \varphi(f_{i}) = e_{i},  \qquad 
  \varphi(\omega_{i}) = \omega_{i}, \qquad  \varphi(\omega_{i}') = \omega_{i}'  \qquad \text{for all} \quad 1 \le i \le n.
\end{equation*}

\noindent
(2) There is a unique $\BC$-algebra anti-automorphism $\tau\colon U_{r,s}(\fg) \to U_{r,s}(\fg)$ satisfying 
$\tau(r) = s^{-1}$, $\tau(s) = r^{-1}$, and
\begin{equation*}
    \tau(e_{i}) = e_{i}, \qquad  \tau(f_{i}) = f_{i}, \qquad  \tau(\omega_{i}) = (r_{i}s_{i})^{-1}\omega_{i}', \qquad 
    \tau(\omega_{i}') = (r_{i}s_{i})^{-1}\omega_{i} \qquad  \text{for all} \quad 1 \le i \le n.
\end{equation*}

\noindent
(3) There is a unique $\BC$-algebra automorphism $x \mapsto \bar{x}$ of $U_{r,s}(\fg)$ (called the \textbf{bar involution}) 
satisfying $\bar{r} = s$, $\bar{s} = r$, and 
\begin{equation*}
  \bar{e}_{i} = e_{i}, \qquad  \bar{f}_{i} = f_{i}, \qquad  \bar{\omega}_{i} = \omega_{i}', \qquad 
  \bar{\omega_{i}'} = \omega_{i} \qquad  \text{for all} \quad 1 \le i \le n.
\end{equation*}
\end{prop}

\begin{proof} 
For each of these, we need to check that defining relations~\eqref{eq:R1}--\eqref{eq:R5} are preserved. 
For parts (1) and (3), this is straightforward, and we leave details to the reader. For part (2), it is 
easy to check that~\eqref{eq:R1}--\eqref{eq:R3} and~\eqref{eq:R4} with $i\ne j$ are preserved under $\tau$. 
For~\eqref{eq:R4} with $i=j$, we have: 
\[
  \frac{\tau(\omega_{i}) - \tau(\omega_{i}')}{\tau(r_{i}) - \tau(s_{i})} = 
  \frac{(r_{i}s_{i})^{-1}(\omega_{i}' - \omega_{i})}{s_{i}^{-1} - r_{i}^{-1}} = 
  \frac{\omega_{i}' - \omega_{i}}{r_{i} - s_{i}} = 
  f_{i}e_{i} - e_{i}f_{i} = \tau(f_{i})\tau(e_{i}) - \tau(e_{i})\tau(f_{i}).
\]
For the quantum Serre relations~\eqref{eq:R5}, we shall only carry out the verification for the $e_{i}$'s, 
since the calculations for the $f_{i}$'s are analogous. First, we note that 
\[
  \tau([m]_{r,s}) = \frac{s^{-m} - r^{-m}}{s^{-1} - r^{-1}} = (rs)^{1 - m}[m]_{r,s},
\]
and therefore 
\begin{align*}
  \tau\left( \qbinom{m}{k}_{r,s} \right) 
  &= \frac{\tau([m]_{r,s}!)}{\tau([k]_{r,s}!)\tau([m - k]_{r,s}!)} 
   = \frac{(rs)^{-\frac{1}{2}m(m-1)}}{(rs)^{-\frac{1}{2}k(k-1)}(rs)^{-\frac{1}{2}(m-k)(m-k-1)}} 
     \frac{[m]_{r,s}!}{[k]_{r,s}![m-k]_{r,s}!} \\
  &= (rs)^{k(k-m)} \qbinom{m}{k}_{r,s}.
\end{align*}
Hence, we have:
\begin{align*}
  & \sum_{k = 0}^{1 - a_{ij}} (-1)^{k}\tau\left( \qbinom{1 - a_{ij}}{k}_{r_{i},s_{i}} (r_{i}s_{i})^{\frac{1}{2}k(k-1)} 
     (rs)^{k\langle \alpha_{j},\alpha_{i}\rangle} \right) \tau(e_{i})^{k}\tau(e_{j})\tau(e_{i})^{1 - a_{ij} - k} = \\
  & \sum_{k = 0}^{1 - a_{ij}} (-1)^{k} (r_{i}s_{i})^{k(k - 1 + a_{ij})} \qbinom{1 - a_{ij}}{k}_{r_{i},s_{i}} 
     (r_{i}s_{i})^{-\frac{1}{2}k(k-1)} (rs)^{-k\langle \alpha_{j},\alpha_{i}\rangle} e_{i}^{k}e_{j}e_{i}^{1- a_{ij} - k} = \\
  & \sum_{k = 0}^{1 - a_{ij}} (-1)^{1 - a_{ij} - k} (r_{i}s_{i})^{-(1 - a_{ij} - k)k} \qbinom{1 - a_{ij}}{k}_{r_{i},s_{i}} 
     (r_{i}s_{i})^{-\frac{1}{2}(1 - a_{ij} -k)(-a_{ij} - k)} (rs)^{-(1 - a_{ij} - k)\langle \alpha_{j},\alpha_{i}\rangle} 
     e_{i}^{1-a_{ij} - k}e_{j}e_{i}^{k}.
\end{align*}
Since 
\[
  -(1 - a_{ij} - k)k - \tfrac{1}{2}(1 - a_{ij} - k)(-a_{ij} - k) = 
  \tfrac{1}{2}k(k-1) - \tfrac{1}{2}a_{ij}(a_{ij} - 1),
\]
we thus find that applying $\tau$ to the first relation in~\eqref{eq:R5}, we get:
\begin{align*}
  & 0= \sum_{k = 0}^{1 - a_{ij}} (-1)^{k} \tau\left( \qbinom{1 - a_{ij}}{k}_{r_{i},s_{i}} (r_{i}s_{i})^{\frac{1}{2}k(k-1)}
       (rs)^{k\langle \alpha_{j},\alpha_{i}\rangle} \right) \tau(e_{i})^{k}\tau(e_{j})\tau(e_{i})^{1 - a_{ij} - k} = \\
  & \ \ (-1)^{1 - a_{ij}} (r_{i}s_{i})^{-\frac{1}{2}a_{ij}(a_{ij} - 1)} (rs)^{-(1 - a_{ij})\langle \alpha_{j},\alpha_{i}\rangle}
    \sum_{k = 0}^{1 - a_{ij}} (-1)^{k} \qbinom{1 - a_{ij}}{k}_{r_{i},s_{i}} (r_{i}s_{i})^{\frac{1}{2}k(k-1)}
       (rs)^{k\langle \alpha_{j},\alpha_{i}\rangle} e_{i}^{1 - a_{ij} - k}e_{j}e_{i}^{k}.
\end{align*}
The above shows that the first relation in~\eqref{eq:R5} is indeed preserved by $\tau$. This completes the proof.
\end{proof}


\subsection{Non-degeneracy of pairing and weight space dimensions}
\

Let $\lambda = \sum_{i = 1}^{n}l_{i}\varpi_{i} \in P \cap Q$ with all $l_i\geq 0$ be a dominant weight 
(we only consider such weights that lie in the root lattice just to avoid extending the base field $\BC(r,s)$).
Recall that a vector $v$ in a $U_{r,s}(\fg)$-module $V$ is said to have \textbf{weight} $\lambda$ if 
\[
  \omega_{i}\cdot v = (\omega_{\lambda}',\omega_{i})v \qquad \text{and} \qquad  
  \omega_{i}'\cdot v = (\omega_{i}',\omega_{\lambda})^{-1}v \quad \text{for all}\quad 1 \le i \le n.
\]
We denote the subspace of all vectors of weight $\lambda$ in $V$ by $V_{\lambda}$.

Let $M(\lambda) = U_{r,s}(\fg) \otimes_{U_{r,s}^{\ge}(\fg)} \BC(r,s)$ be the $U_{r,s}(\fg)$-\textbf{Verma module} 
with highest weight $\lambda$, where the action of $U_{r,s}^{\ge}(\fg)$ on $\BC(r,s)$ is defined by 
\begin{equation*}
  e_{i} \cdot 1 = 0 \quad \text{and}\quad \omega_{i}\cdot 1 = (\omega_{\lambda}',\omega_{i}) \quad \mathrm{for\ all} \quad 1\leq i\leq n.
\end{equation*}  
Let $L(\lambda)$ be the unique irreducible quotient of $M(\lambda)$. 
If $v_{\lambda} \in M(\lambda)$ is a nonzero highest weight vector, set
\begin{equation*}
  \wtd{L}(\lambda) = M(\lambda) \Big/ \sum_{i = 1}^{n}U_{r,s}^{-}f_{i}^{l_{i} +1}v_{\lambda}.
\end{equation*}
For classical $\fg$, the $U_{r,s}(\fg)$-module $\wtd{L}(\lambda)$ is known to be finite~dimensional of highest weight 
$\lambda$ (see~\cite[proof of Lemma 2.12]{BW2} for $A$-type and~\cite[Proposition 2.16]{BGH2} for $BCD$-types). Since 
$\wtd{L}(\lambda)$ surjects onto $L(\lambda)$, the module $L(\lambda)$ is also finite~dimensional. The argument 
of~\cite[Proposition~3.8]{MT} can be carried out for both $\wtd{L}(\lambda)$ and $L(\lambda)$, as they are both 
finite~dimensional highest weight $U_{r,s}(\fg)$-modules of weight $\lambda$. This shows that $\wtd{L}(\lambda)$ and 
$L(\lambda)$ have the same dimension, and since $L(\lambda)$ is a quotient of $\wtd{L}(\lambda)$, we conclude 
that $L(\lambda) \simeq \wtd{L}(\lambda)$. This observation allows us to prove the following result:

\begin{prop}\label{prop:big_representation} 
Let $\lambda = \sum_{i = 1}^{n} l_{i}\varpi_{i} \in P \cap Q$ be a dominant weight (i.e.\ $l_i\geq 0$), and let $v_{\lambda} \in L(\lambda)$ 
be a nonzero vector of weight $\lambda$. For any $\mu = \sum_{i = 1}^{n} m_{i}\alpha_{i} \in Q^{+}$ satisfying 
$m_{i} \le l_{i}$ for all $1\leq i\leq n$, the map $(U_{r,s}^{-})_{-\mu} \to L(\lambda)_{\lambda-\mu}$ defined by 
$y \mapsto yv_{\lambda}$ is bijective. 
\end{prop}

\begin{proof} 
Let $J^{-} \subset U_{r,s}^{-}$ be the left ideal generated by the elements $\{f_{i}^{l_{i} + 1}\}_{i=1}^n$. Then 
the map $U_{r,s}^{-} \to M(\lambda)$ defined by $y \mapsto yv_{\lambda}$ induces a $U_{r,s}^{-}$-module isomorphism 
$U_{r,s}^{-}/J^{-} \iso \wtd{L}(\lambda) \simeq L(\lambda)$. But since $m_{i} \le l_{i}$ for all $i$, we have 
$J^{-} \cap (U_{r,s}^{-})_{-\mu} = 0$, so that the restriction of $U_{r,s}^{-} \twoheadrightarrow U_{r,s}^{-}/J^{-} \iso L(\lambda)$ 
to $(U_{r,s}^{-})_{-\mu}$ gives rise to the claimed isomorphism $(U_{r,s}^{-})_{-\mu} \iso L(\lambda)_{\lambda-\mu}$.
\end{proof}

As our first application of the proposition above, let us prove the following result on the dimensions of 
the weight spaces $(U_{r,s}^{\pm})_{\mu}$, which will be needed later for Theorem~\ref{thm:LR-bijection}:

\begin{prop}\label{prop:graded-dimension} 
For all $\mu \in Q^{+}$, we have 
\[
  \dim_{\BC(r,s)}(U_{r,s}^{+})_{\mu} = \dim_{\BC(r,s)}(U_{r,s}^{-})_{-\mu} = \dim_{\BC}U(\fn^{\pm})_{\pm\mu}.
\] 
\end{prop}

\begin{proof} 
Let $\mu = \sum_{i = 1}^{n} m_{i}\alpha_{i} \in Q^{+}$, and let $\lambda = \sum_{i = 1}^{n}l_{i}\varpi_{i} \in P$ be 
a dominant weight with $l_{i} \ge m_{i}$. Because $P$ is contained in the $\bb{Q}$-span of $\alpha_{1},\ldots ,\alpha_{n}$, 
we may replace $\lambda$ by a suitable positive integer multiple to ensure that $\lambda \in P \cap Q$. Then 
$\dim_{\BC(r,s)}(U_{r,s}^{-})_{-\mu} = \dim_{\BC(r,s)}L(\lambda)_{\lambda - \mu}$, due to Proposition~\ref{prop:big_representation}. 
On the other hand, by~\cite[Proposition 3.8]{MT} and the choice of $\lambda$, we have 
$\dim_{\BC(r,s)}L(\lambda)_{\lambda - \mu} = \dim_{\BC}U(\fn^{-})_{-\mu}$, which proves the claim for $U_{r,s}^{-}$. 
The result for $U_{r,s}^{+}$ follows by applying the anti-automorphism $\varphi$ of Proposition~\ref{prop:extra_structures}(1).
\end{proof}

We shall now prove the non-degeneracy of the Hopf pairing $(\cdot,\cdot)_{H}$ introduced in Proposition~\ref{prop:pairing_2param}. 
For this, we require the following result, which is another application of Proposition~\ref{prop:big_representation}:

\begin{prop}\label{prop:commutator_zero} 
Let $\mu \in Q^{+}\setminus \{0\}$. If $y \in (U_{r,s}^{-})_{-\mu}$ satisfies $e_{i}y - ye_{i} = 0$ for all $1\leq i\leq n$, 
then $y = 0$. Similarly, if $x \in (U_{r,s}^{+})_{\mu}$ satisfies $f_{i}x - xf_{i} = 0$ for all $1\leq i\leq n$, then $x = 0$.
\end{prop}

\begin{proof} 
Given $\mu$, choose $\lambda$ as in Proposition~\ref{prop:big_representation}. Then, if $y \neq 0$, we have $yv_{\lambda} \neq 0$, 
where $v_{\lambda} \in L(\lambda)$ is a nonzero vector of weight $\lambda$. The assumption that $e_{i}y - ye_{i} = 0$ for all $i$ 
then implies that $e_{i}yv_{\lambda} = ye_{i}v_{\lambda} = 0$ for all $i$, and hence $yv_{\lambda}\in L(\lambda)$ 
is a highest weight vector of weight $\lambda - \mu < \lambda$. This contradicts the fact that $L(\lambda)$ is irreducible, so 
we must have $y = 0$. The result for $U_{r,s}^{+}$ is obtained by applying $\varphi$ of Proposition~\ref{prop:extra_structures}(1).
\end{proof}

We also need the following lemma:

\begin{lemma}\label{lem:commutator_formulas} 
(1) For all homogeneous $y \in U_{r,s}^{-}$ and all $1\leq i\leq n$, we have 
\begin{equation}\label{eq:commutator+}
  e_{i}y - ye_{i} = \frac{1}{r_{i} - s_{i}} \big( \omega_{i}p_{i}(y) - p_{i}'(y)\omega_{i}' \big).
\end{equation}

\noindent
(2) For all homogeneous $x \in U_{r,s}^{+}$ and all $1\leq i\leq n$, we have 
\[
  xf_{i} - f_{i}x = \frac{1}{r_{i} - s_{i}} \big( p_{i}(x)\omega_{i} - \omega_{i}'p_{i}'(x) \big).
\]
\end{lemma}

\begin{proof} 
For part (1), the equality is clear when $y = 1$, and for $y = f_{j}$ it follows from~\eqref{eq:R4}. Since both sides 
of~\eqref{eq:commutator+} are linear in $y$, it is enough to show that if~\eqref{eq:commutator+} holds for $y'$ and $y''$, 
then it also holds for $y = y'y''$. Using the identities $y'\omega_{i} = (\omega_{-\deg(y')}',\omega_{i})\omega_{i}y'$ 
and $\omega_{i}'y'' = (\omega_{i}',\omega_{-\deg(y'')})y''\omega_{i}'$, we obtain:
\begin{align*}
  e_{i}(y'y'') - (y'y'')e_{i} 
  &= (e_{i}y' - y'e_{i})y'' + y'(e_{i}y'' - y''e_{i}) \\
  &= \frac{1}{r_{i} - s_{i}} 
      \left( (\omega_{i}p_{i}(y') - p_{i}'(y')\omega_{i}')y'' + y'(\omega_{i}p_{i}(y'') - p_{i}'(y'')\omega_{i}') \right) \\
  &= \frac{1}{r_{i} - s_{i}} \left( \omega_{i}(p_{i}(y')y'' + (\omega_{-\deg(y')}',\omega_{i})y'p_{i}(y'')) - 
      (y'p_{i}'(y'') + (\omega_{i}',\omega_{-\deg(y'')})p_{i}'(y')y'')\omega'_i \right) \\
  &= \frac{1}{r_{i} - s_{i}}(\omega_{i}p_{i}(y'y'') - p_{i}'(y'y'')\omega_{i}'),
\end{align*}
where the last equality follows from~\eqref{eq:-derivation}.

As for part (2), one can either use a similar argument, or rather note that it follows by applying $\varphi$ of Proposition~\ref{prop:extra_structures}(1) 
to~\eqref{eq:commutator+} since $\varphi \circ p_{i} = p_{i} \circ \varphi$ and $\varphi \circ p'_{i} = p'_{i} \circ \varphi$ (the latter can be established 
by comparing~\eqref{eq:+derivation} and~\eqref{eq:-derivation}).
\end{proof}

\begin{prop}\label{prop:nondegeneracy-Hopf} 
The restriction of the Hopf pairing $(\cdot,\cdot)_{H}$ of Proposition~\ref{prop:pairing_2param} to 
$(U_{r,s}^{-})_{-\mu} \times (U_{r,s}^{+})_{\mu}$ is non-degenerate for all $\mu \in Q^{+}$.
\end{prop}

\begin{proof} 
The claim is obvious for $\mu = 0$. Now, suppose that the claim holds for all $\nu \in Q^{+}$ with $0 \le \nu < \mu$ 
and let $y \in (U_{r,s}^{-})_{-\mu}$ be an element such that $(y,x)_{H} = 0$ for all $x \in (U_{r,s}^{+})_{\mu}$. 
Then for all $1\leq i\leq n$ and $x \in (U_{r,s}^{+})_{\mu - \alpha_{i}}$, we have $0 = (y,e_{i}x)_{H} = (y,xe_{i})_{H}$, 
which implies $(p_{i}(y),x)_{H} = (p_{i}'(y),x)_{H} = 0$, due to Lemma~\ref{prop:-p_adjoint}. Since 
$p_{i}(y),p_{i}'(y) \in (U_{r,s}^{-})_{-(\mu - \alpha_{i})}$, we must have $p_{i}(y) = p_{i}'(y) = 0$ for all $1\leq i\leq n$ 
by the induction hypothesis. Then $e_{i}y - ye_{i} = 0$ for all $1\leq i\leq n$ by Lemma~\ref{lem:commutator_formulas}(1), 
and therefore Proposition~\ref{prop:commutator_zero} implies that $y = 0$, as claimed. 
If $x \in (U_{r,s}^{+})_{\mu}$ satisfies $(y,x)_{H} = 0$ for all $y \in (U_{r,s}^{-})_{-\mu}$, then 
applying similar arguments one shows that $x = 0$. This completes the proof.
\end{proof}


\section{Bilinear Forms}\label{sec:bilinear_forms}

In this Section, we discuss several other pairings and their relation to the Hopf pairing $(\cdot,\cdot)_H$ 
of~\eqref{eq:Hopf-parity}. This allows us to translate the problem solely into the construction of dual bases 
of $U^+_{r,s}$, endowed with a twisted coproduct and a twisted product on its tensor powers, with respect to 
a different symmetric form $(\cdot,\cdot)$. This technical part is essential to the rest of the paper, as it 
eliminates Cartan elements from consideration.


\subsection{Twisted product and compatible pairing}
\

Let $\cal{F}$ be the free associative $\BC(r,s)$-algebra generated by the finite alphabet $I = \{1,2,\ldots ,n\}$. 
Let $\cal{W}$ be the set of words in $I$, i.e.\ the monoid generated by $I$. We shall often use the notation 
$[i_{1} \dots i_{d}] = i_{1}i_{2} \dots i_{d}$ for the elements in $\cal{W}$, where $i_{1},\ldots ,i_{d} \in I$. 
The algebra $\cal{F}$ has a natural grading by the positive cone $Q^{+}$ of the root lattice $Q$, defined by $|i| = \alpha_{i}$; 
we write $|x|$ for the degree of a homogeneous $x$. 
For any $a,b \in \BC(r,s)$, we define the twisted product $\odot_{a,b}$ on $\cal{F}^{\otimes n}$ via
\begin{equation}\label{eq:twisted_product}
  (x_{1} \otimes \dots \otimes x_{n}) \odot_{a,b} (x_{1}' \otimes x_{2}' \otimes \dots \otimes x_{n}') = 
  a^{-\sum_{1 \le i < j \le n}\langle |x_{j}|,|x_{i}'|\rangle} 
  b^{\sum_{1 \le i < j \le n}\langle |x_{i}'|,|x_{j}|\rangle}
  x_{1}x_{1}' \otimes \dots \otimes x_{n}x_{n}'
\end{equation}
for all homogeneous $x_{i},x_{i}' \in \cal{F}$. In particular, evoking~\eqref{eq:generators-parity-2}, 
for all homogeneous $x,x',y,y' \in \cal{F}$, we have: 
\begin{align*}
  (x \otimes y) \odot_{r,s} (x' \otimes y') &= (\omega_{|y|}',\omega_{|x'|})^{-1}xx' \otimes yy', \\
  (x \otimes y) \odot_{s^{-1},r^{-1}}(x' \otimes y') &= (\omega_{|x'|}',\omega_{|y|})^{-1}xx' \otimes yy',
\end{align*}  
cf.~\eqref{eq:parity-convention}. For a fixed choice of $a,b$ as above, we define the algebra homomorphism 
\begin{equation}\label{eq:twisted_coproduct}
  \Delta_{a,b}\colon \cal{F} \longrightarrow (\cal{F} \otimes \cal{F},\odot_{a,b}) \qquad \mathrm{via} \qquad 
  \Delta(i) = i \otimes \emptyset + \emptyset \otimes i.
\end{equation}
For any element $x \in \cal{F}$, we shall use the notation
\[
  \Delta_{a,b}(x) = \sum_{(x)} x_{1;a,b} \otimes x_{2;a,b}.
\]
If $x$ is homogeneous, then we have $|x_{1;a,b}| + |x_{2;a,b}| = |x|$, by the definition of $\Delta_{a,b}$. 
If the values of $a$ and $b$ are clear from context, we will omit the subscripts $a,b$, and write simply  
$\Delta_{a,b}(x) = \sum_{(x)} x_{1} \otimes x_{2}$ instead. The following result shows 
that~\eqref{eq:twisted_coproduct} is in fact coassociative:

\begin{prop}\label{prop:coassociative} 
For any $a,b \in \BC(r,s)$, we have 
  $(\Delta_{a,b} \otimes 1) \circ \Delta_{a,b} = (1 \otimes \Delta_{a,b}) \circ \Delta_{a,b}$.
\end{prop}

\begin{proof} 
This is clearly true on the generators, so it suffices to show that $\Delta_{a,b} \otimes 1$ 
and $1 \otimes \Delta_{a,b}$ are both algebra homomorphisms 
$(\cal{F}^{\otimes 2},\odot_{a,b}) \to (\cal{F}^{\otimes 3},\odot_{a,b})$.

To this end, let $x,x',y,y' \in \cal{F}$ be any homogeneous elements. Then, we have:
\begin{align*}
  (\Delta_{a,b} \otimes 1)((x \otimes y) \odot_{a,b} (x' \otimes y')) 
  &= a^{-\langle |y|, |x'| \rangle}b^{\langle |x'|,|y|\rangle}(\Delta_{a,b}(x) \odot_{a,b} \Delta_{a,b}(x')) \otimes yy' \\
  &= a^{-\langle |y|, |x'|\rangle}b^{\langle |x'|, |y|\rangle }
     \sum_{(x)(x')} a^{-\langle |x_{2}|, |x_{1}'|\rangle} b^{\langle|x_{1}'|, |x_{2}|\rangle}
       x_{1}x_{1}' \otimes x_{2}x_{2}' \otimes yy'
\end{align*}
as well as 
\begin{align*}
  & (\Delta_{a,b}(x) \otimes y) \odot_{a,b}(\Delta_{a,b}(x') \otimes y') = 
    \sum_{(x)(x')} (x_{1} \otimes x_{2} \otimes y) \odot_{a,b}(x_{1}' \otimes x_{2}' \otimes y') \\
  &\qquad= \sum_{(x)(x')} a^{-\langle |x_{2}|, |x_{1}'|\rangle - \langle |y|, |x_{1}'| + |x_{2}'|\rangle} 
     b^{\langle |x_{1}'|, |x_{2}|\rangle + \langle |x_{1}'| + |x_{2}'|, |y|\rangle} 
     x_{1}x_{1}' \otimes x_{2}x_{2}' \otimes yy' \\
  &\qquad= a^{-\langle |y|, |x'|\rangle}b^{\langle |x'|, |y|\rangle } 
     \sum_{(x)(x')} a^{-\langle |x_{2}|, |x_{1}'|\rangle} b^{\langle|x_{1}'|, |x_{2}|\rangle} 
     x_{1}x_{1}' \otimes x_{2}x_{2}' \otimes yy',
\end{align*}
where we used $|x_{1}'| + |x_{2}'| = |x'|$ in the last equality. Comparison of the above two formulas 
shows that $\Delta_{a,b}\otimes 1$ is an algebra homomorphism. The proof that $1 \otimes \Delta_{a,b}$ 
is an algebra homomorphism is similar.
\end{proof}

Consequently, we can define a product, also denoted by $\odot_{a,b}$, on $\cal{F}^{*}$ via 
\begin{equation*}
  (f \odot_{a,b} g)(x) = (f \otimes g)(\Delta_{a,b}(x)). 
\end{equation*}
The coassociativity property of Proposition~\ref{prop:coassociative} guarantees that this product is associative. 
The identity element of $\cal{F}^{*}$ is the map $\epsilon\colon \cal{F} \to \BC(r,s)$ given by 
$\epsilon(\emptyset) = 1$ and $\epsilon(i_1 \ldots i_d) = 0$ for any $i_1,\ldots,i_d\in I$ and $d>0$. 
We can now use this to prove the following theorem:

\begin{theorem}\label{thm:pairing} 
There is a unique bilinear form $\{\cdot,\cdot\}\colon \cal{F} \times \cal{F} \to \BC(r,s)$ satisfying 
\begin{itemize}

\item[(1)] $\{1,1\} = 1$, 

\item[(2)] $\{i,j\} = \delta_{ij}$ for all $i,j \in I$,

\item[(3)] $\{xy,z\} = \{x \otimes y, \Delta_{s^{-1},r^{-1}}(z)\}$ for all $x,y,z \in \cal{F}$, 

\item[(4)] $\{x,yz\} = \{\Delta_{r,s}(x),y \otimes z\}$ for all $x,y,z \in \cal{F}$,
\end{itemize}
where $\{x'\otimes x'',y'\otimes y''\}=\{x',y'\} \{x'',y''\}$ for any $x',x'',y',y''\in \cal{F}$.
\end{theorem}

\begin{proof} 
For each $i \in I$, define a linear map $i^{*}\colon \cal{F} \to \BC(r,s)$ given by $i^{*}(i) = 1$ 
and $i^{*}(x) = 0$ if $x$ is any word other than $i$. We now define an algebra homomorphism 
$\psi\colon \cal{F} \to (\cal{F}^{*},\odot_{s^{-1},r^{-1}})$ via $\psi(i) = i^{*}$ . 
Then we may define a bilinear form on $\cal{F}$ by 
\[
  \{x,y\} = \psi(x)(y) \qquad \text{for all}\quad x,y \in \cal{F}.
\] 
First, it is easy to show that if $|x| = \mu$, then $\psi(x) \in \cal{F}^{*}_{\mu}$, i.e.\ $\psi(x)(y) = 0$ unless 
$y \in \cal{F}_{\mu}$. This translates to $\{x,y\} = 0$ unless $|x| = |y|$. Moreover, this bilinear form clearly 
satisfies (1) and (2), while property (3) follows from the fact that $\psi$ is an algebra homomorphism. 
Indeed, for any $x,y,z \in \cal{F}$, we have
\[
  \{xy,z\} = \psi(xy)(z) = (\psi(x) \odot_{s^{-1},r^{-1}}\psi(y))(z) = 
  (\psi(x) \otimes \psi(y))(\Delta_{s^{-1},r^{-1}}(z)) = \{x \otimes y,\Delta_{s^{-1},r^{-1}}(z)\}.
\]
It is sufficient to prove (4) for any word $x \in \cal{W}$, and we shall do so by induction on $\hgt(|x|)$. 
First, we note that it is clearly true whenever $|x|$ has height $0$ or $1$, for any $y,z \in \cal{F}$. 
If $\hgt(|x|) > 1$, we can write $x = x'x''$ for $\hgt(|x'|),\hgt(|x''|) < \hgt(|x|)$. 
Then, by the induction assumption, we obtain:
\begin{align*}
  &\{x,yz\} \\
  &= \{x' \otimes x'',\Delta_{s^{-1},r^{-1}}(yz)\} \\
  &= \sum_{(y)(z)} (\omega_{|z_{1;s^{-1},r^{-1}}|}',\omega_{|y_{2;s^{-1},r^{-1}}|})^{-1} 
     \{x',y_{1;s^{-1},r^{-1}}z_{1;s^{-1},r^{-1}}\} \{x'',y_{2;s^{-1},r^{-1}} z_{2;s^{-1},r^{-1}}\} \\
  &= \sum_{(y)(z)} (\omega_{|z_{1;s^{-1},r^{-1}}|}',\omega_{|y_{2;s^{-1},r^{-1}}|})^{-1}
     \{\Delta_{r,s}(x'),y_{1;s^{-1},r^{-1}} \otimes z_{1;s^{-1},r^{-1}}\} 
     \{\Delta_{r,s}(x''),y_{2;s^{-1},r^{-1}} \otimes z_{2;s^{-1},r^{-1}}\} \\
  &= \sum_{\substack{(y)(z)\\(x')(x'')}} (\omega_{|z_{1;s^{-1},r^{-1}}|}',\omega_{|y_{2;s^{-1},r^{-1}}|})^{-1}
     \{x'_{1;r,s},y_{1;s^{-1},r^{-1}}\} \{x'_{2;r,s},z_{1;s^{-1},r^{-1}}\} 
     \{x''_{1;r,s},y_{2;s^{-1},r^{-1}}\} \{x''_{2;r,s},z_{2;s^{-1},r^{-1}}\}
\end{align*}
while
\begin{align*}
  &\{\Delta_{r,s}(x),y \otimes z\} = \{\Delta_{r,s}(x'x''),y \otimes z\} \\
  &\quad= \sum_{(x')(x'')} (\omega_{|x'_{2;r,s}|}',\omega_{|x''_{1;r,s}|})^{-1}
     \{x'_{1;r,s}x''_{1;r,s},y\} \{x'_{2,r,s}x''_{2;r,s},z\} \\
  &\quad= \sum_{\substack{(y)(z)\\(x')(x'')}} (\omega_{|x'_{2;r,s}|}',\omega_{|x''_{1;r,s}|})^{-1}
     \{x'_{1;r,s},y_{1;s^{-1},r^{-1}}\} \{x'_{2;r,s},z_{1;s^{-1},r^{-1}}\} 
     \{x''_{1;r,s},y_{2;s^{-1},r^{-1}}\} \{x''_{2;r,s},z_{2;s^{-1},r^{-1}}\}.
\end{align*}
But since $\{u,v\} = 0$ unless $|u| = |v|$, it follows that 
  $(\omega_{|x'_{2;r,s}|}',\omega_{|x''_{1;r,s}|})^{-1} = (\omega_{|z_{1;s^{-1},r^{-1}}|}',\omega_{|y_{2;s^{-1},r^{-1}}|})^{-1}$ 
for all nonzero terms in the last sum above. Thus, we have $\{x,yz\} = \{\Delta_{r,s}(x),y \otimes z\}$, as desired.

The uniqueness of $\{\cdot,\cdot\}$ satisfying (1)--(4) is clear.
\end{proof}

Let us now introduce two linear maps $\partial_i,\partial_{i}'$ that we will use frequently throughout the remainder 
of this Section. For any $x\in \cal{F}$, we have 
\begin{equation}\label{eq:coproduct_structure}
  \Delta_{r,s}(x) = 
  x \otimes 1 + \sum_{i = 1}^{n} \partial_{i}(x) \otimes i + \ldots + \sum_{i = 1}^{n} i \otimes \partial_{i}'(x) + 1 \otimes x.
\end{equation}
The resulting linear maps $\partial_i,\partial_{i}'\colon \cal{F} \to \cal{F}$ satisfy (and are uniquely determined by) 
$\partial_{i}(1) = \partial_{i}'(1) = 0$, $\partial_{i}(j) = \partial_{i}'(j) = \delta_{ij}$,  and the following 
analogues of the Leibniz rule: 
\begin{equation}\label{eq:Leibniz_partial}
\begin{split}
  & \partial_{i}(xx') = x\partial_{i}(x') + (\omega_{i}',\omega_{|x'|})^{-1}\partial_{i}(x)x', \\
  & \partial_{i}'(xx') = \partial_{i}'(x)x' + (\omega_{|x|}',\omega_{i})^{-1}x\partial_{i}'(x'),
\end{split}
\end{equation}
for all homogeneous elements $x,x' \in \cal{F}$.

\begin{lemma}\label{lem:derivative_adjoint}
The bilinear form of Theorem \ref{thm:pairing} has the following properties:
\begin{equation*}
  \{x,iy\} = \{\partial_{i}'(x),y\} \qquad \mathrm{and} \qquad \{x,yi\} = \{\partial_{i}(x),y\}
  \qquad \mathrm{for\ any} \ x,y \in \cal{F},\   i \in I.
\end{equation*}
\end{lemma}

\begin{proof}
This follows immediately by combining~\eqref{eq:coproduct_structure} with Theorem~\ref{thm:pairing}(2,4).    
\end{proof}

Likewise, we introduce linear maps $\tilde{\partial}_{i},\tilde{\partial}_{i}'$ on $\cal{F}$ via:
\begin{equation}\label{eq:coproduct_structure_2}
  \Delta_{s^{-1},r^{-1}}(x) = x \otimes 1 + \sum_{i = 1}^{n} \tilde{\partial}_{i}(x) \otimes i + \ldots + 
  \sum_{i = 1}^{n}i \otimes \tilde{\partial}_{i}'(x) + 1 \otimes x.
\end{equation}
Similarly to above, the resulting linear maps satisfy (and are uniquely determined by) 
  $\tilde{\partial}_{i}(1) = \tilde{\partial}_{i}'(1) = 0$, 
  $\tilde{\partial}_{i}(j) = \tilde{\partial}_{i}'(j) = \delta_{ij}$,   
and the following analogues of the Leibniz rule: 
\begin{align*}
  & \tilde{\partial}_{i}(xx') = x\tilde{\partial}_{i}(x') + (\omega_{|x'|}',\omega_{i})^{-1}\tilde{\partial}_{i}(x)x', \\
  & \tilde{\partial}_{i}'(xx') = \tilde{\partial}_{i}'(x)x' + (\omega_{i}',\omega_{|x|})^{-1}x\tilde{\partial}_{i}'(x'),
\end{align*}
for all homogeneous elements $x,x' \in \cal{F}$. The following is proved similarly to Lemma~\ref{lem:derivative_adjoint}:

\begin{lemma}\label{lem:derivative_adjoint_2}
The bilinear form of Theorem \ref{thm:pairing} has the following properties:
\begin{equation*}
  \{ix,y\} = \{x,\tilde{\partial}_{i}'(y)\} \qquad \mathrm{and} \qquad \{xi,y\} = \{x,\tilde{\partial}_{i}(y)\}
  \qquad \mathrm{for\ any} \ x,y \in \cal{F},\   i\in I.
\end{equation*}
\end{lemma}


\subsection{Reduction modulo radicals}
\

Let $\cal{I}$ be the left radical of the form $\{\cdot,\cdot\}$ introduced in Theorem~\ref{thm:pairing}:
\begin{equation}\label{eq:left-radical}
  \cal{I}=\big\{x \in \cal{F} \,\big|\, \{x,y\} = 0 \ \mathrm{for\ any}\ y \in \cal{F}\}.    
\end{equation}
We define $\bar{\cal{F}} = \cal{F}/\cal{I}$, and denote the image of $i$ in $\cal{F}/\cal{I}$ by $e_{i}$ 
(we shall see later in~\eqref{eq:F-vs-U} that $\bar{\cal{F}} \simeq U_{r,s}^{+}$, justifying this notation). 
It is easy to see that $\cal{I}$ is a homogeneous ideal, so that $\bar{\cal{F}}$ inherits the $Q^+$-grading from $\cal{F}$. 
Moreover, $\cal{I}$ is actually a two-sided ideal of $\cal{F}$, according to Theorem~\ref{thm:pairing}(3), 
so that $\bar{\cal{F}}$ is an algebra. Let us now define the divided powers
\begin{equation*}
  e_{i}^{(k)} = \frac{e_{i}^{k}}{[k]_{r_{i},s_{i}}!} \qquad \mathrm{for\ all} \quad k \in \BZ_{\geq 0},\ i \in I.
\end{equation*}
Our next goal is to prove the following theorem:

\begin{theorem}\label{thm:serre} 
For any $i \neq j$, the following relation holds in the algebra $\bar{\cal{F}}$:
\[
  \sum_{k = 0}^{1 - a_{ij}} (-1)^{k} (r_{i}s_{i})^{\frac{1}{2}k(k-1)} (rs)^{k\langle \alpha_{j},\alpha_{i}\rangle}
  e_{i}^{(1 - a_{ij} - k)}e_{j}e_{i}^{(k)} = 0.
\]
\end{theorem}

Before proceeding to the proof of this result, we observe the following easy consequence of Lemma~\ref{lem:derivative_adjoint}:

\begin{lemma}
The ideal $\cal{I}$ of~\eqref{eq:left-radical} is stable under the maps $\partial_i,\partial_{i}'$ 
of~\eqref{eq:coproduct_structure} for all $i \in I$.    
\end{lemma}

Therefore, we obtain the same-named linear maps $\partial_i,\partial_{i}'\colon \bar{\cal{F}}\to \bar{\cal{F}}$ 
on the quotient algebra $\bar{\cal{F}}$. Moreover, we have the following result:

\begin{prop}\label{prop:zero_criteria+} 
(1) If $x\in \bar{\cal{F}}$ satisfies $\partial_{i}(x) = 0$ for all $i \in I$, then $x = 0$.

\medskip
\noindent
(2) If $x \in \bar{\cal{F}}$ satisfies $\partial_{i}'(x) = 0$ for all $i \in I$, then $x = 0$.
\end{prop}

\begin{proof} 
It is enough to prove (1) and (2) for a homogeneous $x \in \bar{\cal{F}}$. If $x\ne 0$ and $\hat{x}\in \cal{F}$ is its 
lift of the same degree, then by the very definition of $\bar{\cal{F}}$, there must be some $\hat{y} \in \cal{F}$ such 
that $\{\hat{x},\hat{y}\} \neq 0$, and moreover we must have $|\hat{y}| = |\hat{x}|$. Furthermore, since $\cal{F}_{\mu}$ 
is spanned by all $[i_{1}\ldots i_{k}]$ for which $\alpha_{i_{1}} + \dots + \alpha_{i_{k}} = \mu$, it follows that 
in fact there is some sequence $(i_{1},\ldots ,i_{k})$ such that $\{\hat{x},[i_{1} \dots i_{k}]\} \neq 0$. 
But then we have
\[
  \{\partial_{i_{1}}'(\hat{x}),[i_{2} \dots i_{k}]\} = \{\hat{x},[i_{1}i_{2} \dots i_{k}]\} \neq 0
\]
and 
\[
  \{\partial_{i_{k}}(\hat{x}),[i_{1} \dots i_{k-1}]\} = \{\hat{x},[i_{1} \dots i_{k-1}i_{k}]\} \neq 0,
\]
so it follows that $\partial_{i_{1}}'(x) \neq 0$ and $\partial_{i_{k}}(x) \neq 0$ in $\bar{\cal{F}}$. 
This implies (2) and (1) via contradiction. 
\end{proof}

We shall now use this proposition to prove Theorem~\ref{thm:serre}.

\begin{proof}[Proof of Theorem~\ref{thm:serre}] 
For $i \neq j$, set 
\begin{equation}\label{eq:Sij}
  S_{ij} = \sum_{k = 0}^{1 - a_{ij}} (-1)^{k} (r_{i}s_{i})^{\frac{1}{2}k(k-1)} 
  (rs)^{k\langle \alpha_{j},\alpha_{i}\rangle} e_{i}^{(1 - a_{ij} - k)}e_{j}e_{i}^{(k)}.
\end{equation}
According to Proposition~\ref{prop:zero_criteria+}, it suffices to verify $\partial_{p}'(S_{ij}) = 0$ for all $p\in I$. 
This vanishing is clear for $p \neq i,j$, so it remains to verify $\partial_{i}'(S_{ij}) = \partial_{j}'(S_{ij}) = 0$.

First, a simple inductive argument shows that
\[
  \partial_{i}'(e_{i}^{(k)}) = r_{i}^{1- k}e_{i}^{(k - 1)},
\]
so that
\[
  \partial_{i}'(e_{i}^{(k)}e_{j}e_{i}^{(l)}) = 
  r_{i}^{1 - k} e_{i}^{(k - 1)}e_{j}e_{i}^{(l)} + 
  (\omega_{j}',\omega_{i})^{-1} (r_{i}^{-1}s_{i})^{k} r_{i}^{1 - l} e_{i}^{(k)}e_{j}e_{i}^{(l - 1)}
\]
for all $k,l \ge 0$, with the convention that $e_{i}^{(m)} = 0$ for $m < 0$. Thus, we obtain: 
\begin{align*}
  \partial_{i}'(S_{ij}) 
  &= \sum_{k = 0}^{1 - a_{ij}} (-1)^{k} (r_{i}s_{i})^{\frac{1}{2}k(k - 1)} r_{i}^{k + a_{ij}} 
     (rs)^{k\langle \alpha_{j},\alpha_{i}\rangle} e_{i}^{(-a_{ij} - k)}e_{j}e_{i}^{(k)} \\
  &\quad + \sum_{k = 0}^{1 - a_{ij}} (-1)^{k}(r_{i}s_{i})^{\frac{1}{2}k(k-1)} 
     (rs)^{k\langle \alpha_{j},\alpha_{i}\rangle + \langle \alpha_{i},\alpha_{j}\rangle} s_{i}^{1 - a_{ij} - k} 
     e_{i}^{(1 - a_{ij} - k)}e_{j}e_{i}^{(k - 1)}.
\end{align*}
Combining the $k$-th term of the first sum with the $(k + 1)$-st term of the second sum yields
\begin{align*}
  & \partial_{i}'(S_{ij}) = \\
  & \sum_{k = 0}^{-a_{ij}} (-1)^{k} \left((r_{i}s_{i})^{\frac{1}{2}k(k - 1)} r_{i}^{k + a_{ij}} 
    (rs)^{k\langle \alpha_{j},\alpha_{i}\rangle} 
    - (r_{i}s_{i})^{\frac{1}{2}(k + 1)k} (rs)^{(k + 1)\langle \alpha_{j},\alpha_{i}\rangle + \langle \alpha_{i},\alpha_{j}\rangle}
       s_{i}^{-a_{ij} - k} \right) e_{i}^{(-a_{ij} - k)}e_{j}e_{i}^{(k)},
\end{align*}
and since 
\begin{align*}
  (r_{i}s_{i})^{\frac{1}{2}(k + 1)k} (rs)^{(k + 1)\langle \alpha_{j},\alpha_{i}\rangle + \langle \alpha_{i},\alpha_{j}\rangle} 
  s_{i}^{-a_{ij} - k} 
  &= (r_{i}s_{i})^{\frac{1}{2}k(k + 1)} (rs)^{k\langle \alpha_{j},\alpha_{i}\rangle} (r_{i}s_{i})^{a_{ij}}s_{i}^{-a_{ij} - k} \\
  &= (r_{i}s_{i})^{\frac{1}{2}k(k-1)} r_{i}^{k + a_{ij}} (rs)^{k\langle \alpha_{j},\alpha_{i}\rangle},
\end{align*}
we finally obtain $\partial_{i}'(S_{ij}) = 0$, as desired.

For $\partial_{j}'$, we first note that the one-parameter identity (cf.~\cite[\S0.2(4)]{J})
\[
  \sum_{t = 0}^{a} (-1)^{t}q^{t(1 - a)}\qbinom{a}{t}_{q} = 0
\]
implies, by setting $q = (r_{i}s_{i}^{-1})^{1/2}$, the identity
\begin{equation}\label{eq:rs-identity}
  \sum_{t = 0}^{a} (-1)^{t} (r_{i}s_{i})^{\frac{1}{2}t(t-1)} r_{i}^{t(1 - a)} \qbinom{a}{t}_{r_{i},s_{i}} =\ 0.
\end{equation}
According to~\eqref{eq:Leibniz_partial}, we have  
\[
  \partial_{j}'(e_{i}^{(k)}e_{j}e_{i}^{(l)}) = 
  \partial_{j}'(e_{i}^{(k)}e_{j})e_{i}^{(l)} = 
  (\omega_{i}',\omega_{j})^{-k} e_{i}^{(k)}e_{i}^{(l)}
\]
for any $k,l \ge 0$. We thus obtain: 
\begin{align*}
  \partial_{j}'(S_{ij}) 
  &= \sum_{k = 0}^{1 - a_{ij}} (-1)^{k} (r_{i}s_{i})^{\frac{1}{2}k(k - 1)} (rs)^{k\langle \alpha_{j},\alpha_{i}\rangle}
     (\omega_{i}',\omega_{j})^{-(1 - a_{ij} - k)} e_{i}^{(1 - a_{ij} - k)}e_{i}^{(k)} \\
  &= \sum_{k = 0}^{1 - a_{ij}} (-1)^{k} (r_{i}s_{i})^{\frac{1}{2}k(k - 1)}
      r^{k\langle \alpha_{j},\alpha_{i}\rangle - (1 - a_{ij} - k)\langle \alpha_{i},\alpha_{j}\rangle}
      s^{(1 - a_{ij}) \langle \alpha_{j},\alpha_{i}\rangle}
      \qbinom{1 - a_{ij}}{k}_{r_{i},s_{i}} e_{i}^{(1 - a_{ij})} \\
  &= r^{(a_{ij}-1) \langle \alpha_{i},\alpha_{j} \rangle} s^{(1 - a_{ij}) \langle \alpha_{j},\alpha_{i} \rangle} 
     \left( \sum_{k = 0}^{1 - a_{ij}} (-1)^{k} (r_{i}s_{i})^{\frac{1}{2}k(k - 1)} r_{i}^{ka_{ij}} 
           \qbinom{1 - a_{ij}}{k}_{r_{i},s_{i}} \right) e_{i}^{(1 - a_{ij})}=0,
\end{align*}
where we used~\eqref{eq:rs-identity} in the last equality. This completes our proof of the theorem.
\end{proof}

Likewise, let $\cal{I}'$ be the right radical of the bilinear form $\{\cdot,\cdot\}$ from Theorem~\ref{thm:pairing}:
\begin{equation}\label{eq:right_radical}
  \cal{I}' = \big\{ y \in \cal{F} \,\big|\, \{x,y\} = 0\ \text{for any}\ x \in \cal{F} \big\}.
\end{equation}
We set $\bar{\cal{F}}' = \cal{F}/\cal{I}'$, and denote the images of $i$ in $\bar{\cal{F}}'$ by $f_{i}$ for all $i \in I$ 
(this notation is justified by~\eqref{eq:F-vs-U}, where we show that $\bar{\cal{F}}' \simeq U_{r,s}^{-}$). 
We note that $\cal{I}'$ is actually a two-sided ideal of $\cal{F}$, due to Theorem~\ref{thm:pairing}(4), 
so that $\bar{\cal{F}}'$ is an algebra.

As above, we have the following easy consequence of Lemma~\ref{lem:derivative_adjoint_2}:

\begin{lemma} 
The ideal $\cal{I}'$ of~\eqref{eq:right_radical} is stable under the maps $\tilde{\partial}_{i},\tilde{\partial}_{i}'$ 
of~\eqref{eq:coproduct_structure_2} for all $i \in I$.
\end{lemma}

We thus obtain the same-named linear maps $\tilde{\partial}_{i},\tilde{\partial}_{i}'\colon \bar{\cal{F}}' \to \bar{\cal{F}}'$, 
satisfying the analogue of Proposition~\ref{prop:zero_criteria+}:

\begin{prop}\label{prop:zero_criteria-} 
(1) If $x \in \bar{\cal{F}}'$ satisfies $\tilde{\partial}_{i}(x) = 0$ for all $i \in I$, then $x = 0$.

\medskip
\noindent
(2) If $x \in \bar{\cal{F}}'$ satisfies $\tilde{\partial}_{i}'(x) = 0$ for all $i \in I$, then $x = 0$.
\end{prop}

Set $f_{i}^{(k)} = f_{i}^{k}/[k]_{r_{i},s_{i}}!$ for $k \in \BZ_{\geq 0}, i \in I$.
Using the result above, we obtain a counterpart of Theorem~\ref{thm:serre}:

\begin{theorem}\label{thm:-serre} 
For any $i \neq j$, the following relation holds in the algebra  $\bar{\cal{F}}'$:
\[
  \sum_{k = 0}^{1 - a_{ij}} (-1)^{k} (r_{i}s_{i})^{\frac{1}{2}k(k - 1)} (rs)^{k\langle \alpha_{j},\alpha_{i}\rangle}
  f_{i}^{(k)}f_{j}f_{i}^{(1 - a_{ij} - k)} = 0.
\]
\end{theorem}

We now introduce a $\BC(r,s)$-algebra anti-isomorphism $\varphi\colon \bar{\cal{F}} \to \bar{\cal{F}}'$, which will 
ultimately be matched with the corresponding map on $U_{r,s}^{\pm}$ introduced in Proposition~\ref{prop:extra_structures}(1).  
To do so, we first consider the $\BC(r,s)$-algebra anti-involution $\varphi\colon \cal{F} \to \cal{F}$ defined by 
$\varphi(i) = i$ for all $i \in I$. We then have:

\begin{lemma} 
For all $x \in \cal{F}$, we have $\partial_{i}(\varphi(x)) = \varphi(\tilde{\partial}_{i}'(x))$ and 
$\varphi(\partial_{i}(x)) = \tilde{\partial}_{i}'(\varphi(x))$.
\end{lemma}

\begin{proof} 
For the first equality, it is enough to consider $x \in \cal{W}$, for which we proceed by induction on $\hgt(|x|)$. 
If $\hgt(|x|) \leq 1$, the claim is obvious. If $\hgt(|x|) = m > 1$, then we may write $x = x'x''$ for some 
$x',x'' \in \cal{W}$ with $\hgt(|x'|), \hgt(|x''|) < m$. Then by the induction assumption, we have:
\begin{equation*}
\begin{split}
  \partial_{i}(\varphi(x)) 
  &= \partial_{i}(\varphi(x'')\varphi(x')) 
   = \varphi(x'')\partial_{i}(\varphi(x')) + 
     (\omega_{i}',\omega_{|x'|})^{-1}\partial_{i}(\varphi(x''))\varphi(x') \\
  &= \varphi(x'')\varphi(\tilde{\partial}_{i}'(x')) 
     + (\omega_{i}',\omega_{|x'|})^{-1}\varphi(\tilde{\partial}_{i}(x''))\varphi(x') 
   = \varphi\left (\tilde{\partial}_{i}'(x')x'' 
     + (\omega_{i}',\omega_{|x'|})^{-1}x'\tilde{\partial}_{i}'(x'')\right ) \\
  &= \varphi(\tilde{\partial}_{i}'(x'x'')) = \varphi(\tilde{\partial}_{i}'(x)).
\end{split}
\end{equation*}
This proves that $\partial_{i}\circ \varphi = \varphi \circ \tilde{\partial}_{i}'$ for any $i\in I$. 
Since $\varphi$ is an anti-involution, we also get $\varphi \circ \partial_{i} = \tilde{\partial}_{i}'\circ \varphi$.
\end{proof}

We can now easily derive the desired result:

\begin{prop}\label{prop:phi-reduced-F} 
There is a unique $\BC(r,s)$-algebra anti-isomorphism $\varphi\colon \bar{\cal{F}} \to \bar{\cal{F}}'$ such that 
\begin{equation*}
  \varphi(e_{i}) = f_{i} \qquad \mathrm{for\ all} \quad  i \in I.
\end{equation*}
\end{prop}

\begin{proof} 
To prove this, it suffices to show that $\varphi(\cal{I}) \subseteq \cal{I}'$ and $\varphi(\cal{I}') \subseteq \cal{I}$. 
To this end, suppose that $x \in \cal{F}$ satisfies $\varphi(x) \notin \cal{I}'$. By Proposition~\ref{prop:zero_criteria-}, 
we then have $0 \neq \tilde{\partial}_{i}'(\varphi(x)) = \varphi(\partial_{i}(x))$ for some $i \in I$. The latter implies 
$\partial_{i}(x) \neq 0$, and so $x \notin \cal{I}$. This proves the first 
inclusion $\varphi(\cal{I}) \subseteq \cal{I}'$. Similarly, if $y \in \cal{F}$ satisfies $\varphi(y) \notin \cal{I}$, 
then by Proposition~\ref{prop:zero_criteria+}, there is some $j \in I$ such that 
$0 \neq \partial_{j}(\varphi(y)) = \varphi(\tilde{\partial}_{j}'(y))$. This means that $\tilde{\partial}_{j}'(y) \neq 0$ 
and so $y \notin \cal{I}'$, establishing the second inclusion  
$\varphi(\cal{I}') \subseteq \cal{I}$.
\end{proof}

We note that $\cal{I}$ is a Hopf ideal with respect to $\Delta_{r,s}$, i.e.\ 
$\Delta_{r,s}(\cal{I})\subseteq \cal{F} \otimes \cal{I} + \cal{I} \otimes \cal{F}$, due to Theorem~\ref{thm:pairing}(4). 
This implies that $\Delta_{r,s}$ descends from $\cal{F}$ to the same-named algebra homomorphism 
$\Delta_{r,s}\colon \bar{\cal{F}}\to \bar{\cal{F}} \otimes \bar{\cal{F}}$. Likewise, we have 
$\Delta_{s^{-1},r^{-1}}(\cal{I}')\subseteq \cal{F} \otimes \cal{I}' + \cal{I}' \otimes \cal{F}$ by 
Theorem~\ref{thm:pairing}(3), so that $\Delta_{s^{-1},r^{-1}}$ descends from $\cal{F}$ to the same-named 
algebra homomorphism $\Delta_{s^{-1},r^{-1}}\colon \bar{\cal{F}}' \to \bar{\cal{F}}' \otimes \bar{\cal{F}}'$.  
Finally, we define the linear map 
\begin{equation*}
  T\colon \bar{\cal{F}}' \otimes \bar{\cal{F}}' \longrightarrow \bar{\cal{F}}' \otimes \bar{\cal{F}}' 
  \qquad {via} \qquad T(x \otimes y) = y \otimes x. 
\end{equation*}
Then, we have the following result:

\begin{prop}\label{prop:Delta-varphi}
For all $z \in \bar{\cal{F}}$, we have
\[
  (T\circ (\varphi \otimes \varphi) \circ \Delta_{r,s})(z) = (\Delta_{s^{-1},r^{-1}} \circ \varphi)(z).
\]
\end{prop}

\begin{proof}  
As usual, it is enough to prove the claim for elements of the form $z = e_{i_{1}}e_{i_{2}} \dots e_{i_{m}}$. 
We proceed by induction on $\hgt(|z|)$. The claim is obvious for $\hgt(|z|) \leq 1$, so let us suppose that 
$\hgt(|z|) = m > 1$ and the claim holds for all $z'$ with $\hgt(|z'|) < m$. By assumption, there is some $i$ 
such that $z = z'e_{i}$, where $\hgt(|z'|) < \hgt(|z|)$. Thus, by the induction assumption, we have: 
\begin{align*}
  \Delta_{s^{-1},r^{-1}}(\varphi(z)) 
  &= \Delta_{s^{-1},r^{-1}}(f_{i}\varphi(z')) = (f_{i} \otimes 1 + 1 \otimes f_{i}) \odot_{s^{-1},r^{-1}} 
     \sum_{(z')} \varphi(z_{2;r,s}') \otimes \varphi(z_{1;r,s}') \\
  &= \sum_{(z')} f_{i}\varphi(z_{2;r,s}') \otimes \varphi(z_{1;r,s}') + 
     \sum_{(z')} (\omega_{|z_{2;r,s}'|}',\omega_{i})^{-1} \varphi(z_{2;r,s}') \otimes f_{i}\varphi(z_{1;r,s}').
\end{align*}
On the other hand, we also have: 
\begin{align*}
  T(\varphi \otimes \varphi)\Delta_{r,s}(z'e_{i}) 
  &= T(\varphi \otimes \varphi) 
     \left( \left( \sum_{(z')} z_{1;r,s}' \otimes z_{2;r,s}' \right) \odot_{r,s} (e_{i} \otimes 1 + 1 \otimes e_{i}) \right) \\
  &= T(\varphi \otimes \varphi) 
     \left( \sum_{(z')} (\omega_{|z_{2;r,s}'|}',\omega_{i})^{-1} z_{1;r,s}'e_{i} \otimes z_{2;r,s}' + 
     \sum_{(z')} z_{1;r,s}' \otimes z_{2;r,s}'e_{i} \right) \\ 
  &= \sum_{(z')} (\omega_{|z_{2;r,s}'|}',\omega_{i})^{-1} \varphi(z_{2;r,s}') \otimes f_{i}\varphi(z_{1;r,s}') + 
     \sum_{(z')} f_{i}\varphi(z_{2;r,s}') \otimes \varphi(z_{1;r,s}').
\end{align*}
This completes the proof, since the right-hand sides of the above two equations coincide.
\end{proof}


\subsection{Symmetric version of pairing}
\

The pairing $\{\cdot,\cdot\} \colon \cal{F} \times \cal{F} \to \BC(r,s)$ of Theorem~\ref{thm:pairing} induces 
the same-named non-degenerate pairing $\{\cdot,\cdot\}$ on $\bar{\cal{F}} \times \bar{\cal{F}}'$, by the 
definition of~(\ref{eq:left-radical},~\ref{eq:right_radical}). The latter can be turned into the pairing 
$(\cdot,\cdot)\colon \bar{\cal{F}} \times \bar{\cal{F}} \to \BC(r,s)$ through the use of the anti-isomorphism 
$\varphi$ from Proposition~\ref{prop:phi-reduced-F}. More specifically, we set 
\begin{equation}\label{eq:pairing_positive}
  (x,x') = \{x,\varphi(x')\} \qquad \mathrm{for\ any} \quad x,x' \in \bar{\cal{F}}. 
\end{equation}

\begin{lemma}\label{lem:symm_pairing_properties}
The resulting bilinear form $(\cdot,\cdot)$ satisfies the following five properties:
\begin{itemize}
\item[(1)] $(1,1) = 1$,

\item[(2)] $(e_{i},e_{j}) = \delta_{ij}$,

\item[(3)] $(x,y) = 0$ if $|x| \neq |y|$, 

\item[(4)] $(x,yz) = (\Delta_{r,s}(x),z \otimes y)$,

\item[(5)] $(xy,z) = (y \otimes x, \Delta_{r,s}(z))$,
\end{itemize}
where $(x'\otimes x'',y'\otimes y'')=(x',y') (x'',y'')$ for any $x', x'',y',y''\in \bar{\cal{F}}$.
\end{lemma}

\begin{proof}
Properties (1)--(3) are immediate from the definition and the corresponding properties of $\{\cdot,\cdot\}$. 
Property (4) follows from 
  $(x,yz) =\{x,\varphi(yz)\} = \{x,\varphi(z)\varphi(y)\} = 
   \{\Delta_{r,s}(x),\varphi(z) \otimes \varphi(y)\} = (\Delta_{r,s}(x),z \otimes y)$. 
Property (5) follows along the same line by evoking Proposition~\ref{prop:Delta-varphi}:
\[
  (xy,z) = \{x \otimes y, \Delta_{s^{-1},r^{-1}}(\varphi(z))\} = 
  \{x \otimes y, T(\varphi \otimes \varphi)(\Delta_{r,s}(z))\} = 
  \{y \otimes x, \varphi \otimes \varphi(\Delta_{r,s}(z))\} = (y \otimes x,\Delta_{r,s}(z)).
\]
This completes the proof of the lemma.
\end{proof}

We also note the following important property of $(\cdot, \cdot)$:

\begin{lemma}
The pairing $(\cdot, \cdot)$ of~\eqref{eq:pairing_positive} is symmetric. 
\end{lemma}

\begin{proof}
The pairing $(\cdot,\cdot)$ not only satisfies the above properties (1)--(5), but is also uniquely determined by them. 
However, the pairing $(\cdot,\cdot)^\circ$ defined via $(x,y)^\circ = (y,x)$ clearly satisfies the same properties. 
Therefore, $(\cdot,\cdot)^\circ = (\cdot,\cdot)$, which shows that $(\cdot,\cdot)$ is indeed symmetric.  
\end{proof}

Finally, as an immediate consequence of Lemma~\ref{lem:derivative_adjoint}, we obtain:

\begin{lemma}\label{lem:symm-der-adjoint}
For any $x,x'\in \bar{\cal{F}}$ and $i\in I$, we have the following two equalities:
\begin{equation*}
  (x,e_{i}x') = (\partial_{i}(x),x') \qquad \mathrm{and} \qquad (x,x'e_{i}) = (\partial_{i}'(x),x'). 
\end{equation*}
\end{lemma}


\subsection{Relation to the Hopf pairing}
\

Recall the Hopf pairing $(\cdot,\cdot)_{H}\colon U_{r,s}^{\le} \times U_{r,s}^{\ge} \to \BC(r,s)$ from~\eqref{eq:Hopf-parity}. 
We conclude this Section with an explicit relation between $(\cdot,\cdot)_{H}$ and $(\cdot,\cdot)$ of~\eqref{eq:pairing_positive}, 
which will allow us to show that $U_{r,s}^{+} \simeq \bar{\cal{F}}$, $U_{r,s}^{-} \simeq \bar{\cal{F}}'$, see~\eqref{eq:F-vs-U}. 
This relationship is also instrumental to the proof of our main Theorems~\ref{thm:main-1} and~\ref{thm:main-2}.

Before proceeding, we note first that by Theorems~\ref{thm:serre} and~\ref{thm:-serre}, there are natural $\BC(r,s)$-algebra 
homomorphisms $\psi^{+}\colon U_{r,s}^{+} \to \bar{\cal{F}}$ and $\psi^{-}\colon U_{r,s}^{-} \to \bar{\cal{F}}'$, determined 
by $\psi^+(e_i)=e_i$ and $\psi^-(f_i)=f_i$ for all $i\in I$. We also recall the bar involution $x \mapsto \bar{x}$ on 
$U_{r,s}^{\pm}$ from Proposition~\ref{prop:extra_structures}(3). We can now relate $(\cdot,\cdot)_{H}$ to $(\cdot,\cdot)$:

\begin{theorem}\label{thm:pairing_comparison} 
For all $y \in (U_{r,s}^{-})_{-\mu}$ and $x \in (U_{r,s}^{+})_{\mu}$, 
where $\mu = \sum_{i = 1}^{n}c_{i}\alpha_{i} \in Q^{+}$, we have\footnote{Our convention~\eqref{eq:parity-convention} 
is not in contradiction with~\eqref{eq:pairing_positive}, as the latter is not defined on the Cartan subalgebra.}
\[
  \ol{(\bar{y},\bar{x})}_{H} = 
  \left( \prod_{i = 1}^{n} \frac{1}{(r_{i} - s_{i})^{c_{i}}} \right) \big( \psi^{+}(\varphi(y)),\psi^{+}(x) \big).
\] 
\end{theorem}

\begin{proof}
Evoking the linear maps $p_{i}\colon U_{r,s}^{+} \to U_{r,s}^{+}$ of~\eqref{eq:p-maps}, let us define 
a $\BC(r,s)$-linear map
\begin{equation*}
  \bar{p}_{i}\colon U_{r,s}^{+} \to U_{r,s}^{+} \qquad \mathrm{via} \qquad \bar{p}_{i}(x) = \ol{p_{i}(\bar{x})}.
\end{equation*}
Then, $\bar{p}_{i}(1) = 0$, $\bar{p}_{i}(e_{j}) = \delta_{ij}$, and we claim that they satisfy the following 
analogue of the Leibniz rule: 
\begin{equation}\label{eq:bar-p-Leibniz}
  \bar{p}_{i}(xx') = x\bar{p}_{i}(x') + (\omega_{i}',\omega_{\deg(x')})^{-1}\bar{p}_{i}(x)x' 
\end{equation}
for all homogeneous $x,x' \in U_{r,s}^{+}$. Indeed, this follows from~\eqref{eq:+derivation} and 
$\ol{(\omega'_\mu,\omega_\nu)}=(\omega'_\nu,\omega_\mu)^{-1}$: 
\begin{equation*}
  \bar{p}_{i}(xx') = \ol{p_i(\bar{x}\bar{x}')} = \ol{\bar{x} p_i(\bar{x}') + (\omega'_{\deg(x')},\omega_i) p_i(\bar{x}) \bar{x}'} = 
  x\bar{p}_i(x') + (\omega'_i, \omega_{\deg(x')})^{-1} \bar{p}_i(x) x'.
\end{equation*}

Let us now record the relation between these $\bar{p}_i$ and the linear maps $\partial_i$ of~\eqref{eq:coproduct_structure}:
\begin{equation}\label{eq:p_vs_partial}
  \partial_{i}(\psi^{+}(x)) = \psi^{+}(\bar{p}_{i}(x)) \qquad \text{for all}\quad x \in U_{r,s}^{+}.
\end{equation}
This equality is clear when $x \in \BC(r,s)$ or $x = e_{j}$ with $j\in I$. Thus, it remains to show that 
if~\eqref{eq:p_vs_partial} holds for $x'$ and $x''$, then it also holds for $x'x''$. To this end, we have: 
\begin{equation*}
\begin{split}
  \partial_{i}(\psi^{+}(x'x'')) 
  &= \partial_{i}(\psi^{+}(x')\psi^{+}(x'')) 
   \overset{\eqref{eq:Leibniz_partial}}{=} 
   \psi^{+}(x')\partial_{i}(\psi^{+}(x'')) + (\omega_{i}',\omega_{\deg(x'')})^{-1}\partial_{i}(\psi^{+}(x'))\psi^{+}(x'') \\
  &= \psi^{+}(x')\psi^{+}(\bar{p}_{i}(x'')) + (\omega_{i}',\omega_{\deg(x'')})^{-1}\psi^{+}(\bar{p}_{i}(x'))\psi^{+}(x'') \\
  &= \psi^{+}(x' \bar{p}_{i}(x'')  + (\omega_{i}',\omega_{\deg(x'')})^{-1} \bar{p}_{i}(x') x'')
   \overset{\eqref{eq:bar-p-Leibniz}}{=} \psi^{+}(\bar{p}_{i}(x'x'')).
\end{split}
\end{equation*}

To prove the theorem, we proceed by induction on $\hgt(\mu)$, with the base case $\hgt(\mu) = 0$ being obvious. 
Assuming $\hgt(\mu)> 0$, it is enough to consider $y = y'f_{i}$ for some $i$. Then by 
Lemma~\ref{prop:+p_adjoint} and Lemma~\ref{lem:symm-der-adjoint}, we have:
\begin{align*}
  \ol{(\bar{y},\bar{x})}_{H} 
  &= \ol{(\bar{y}'f_{i},\bar{x})}_{H} 
   = \frac{1}{r_{i} - s_{i}} \ol{(\bar{y}',p_{i}(\bar{x}))}_{H} = \frac{1}{r_{i} - s_{i}}\ol{(\bar{y}',\ol{\bar{p}_{i}(x)})}_{H} \\
  &= \frac{1}{r_{i} - s_{i}} 
     \left( \frac{1}{(r_{i} - s_{i})^{c_{i} - 1}} \cdot \prod_{j \neq i}\frac{1}{(r_{j} - s_{j})^{c_{j}}} \right) 
     \big( \psi^{+}(\varphi(y')),\psi^{+}(\bar{p}_{i}(x)) \big) \\
  &\overset{\eqref{eq:p_vs_partial}}{=} \left (\prod_{i = 1}^{n} \frac{1}{(r_{i} - s_{i})^{c_{i}}} \right)
     \big( \psi^{+}(\varphi(y')),\partial_{i}(\psi^{+}(x)) \big) 
   = \left( \prod_{i = 1}^{n} \frac{1}{(r_{i} - s_{i})^{c_{i}}} \right)  \big( e_{i}\psi^{+}(\varphi(y')),\psi^{+}(x) \big) \\
  &= \left( \prod_{i = 1}^{n} \frac{1}{(r_{i} - s_{i})^{c_{i}}} \right) \big( \psi^{+}(\varphi(y'f_{i})),\psi^{+}(x) \big)
   = \left( \prod_{i = 1}^{n} \frac{1}{(r_{i} - s_{i})^{c_{i}}} \right) \big( \psi^{+}(\varphi(y)),\psi^{+}(x) \big),
\end{align*}
where we used the induction hypothesis in the second line. This completes the proof of the theorem.
\end{proof}

\begin{cor}
The above algebra homomorphisms $\psi^\pm$ are actually algebra isomorphisms:
\begin{equation}\label{eq:F-vs-U}
  \psi^+\colon U^+_{r,s} \iso \bar{\cal{F}}   \qquad \mathrm{and} \qquad 
  \psi^-\colon U^-_{r,s} \iso \bar{\cal{F}}'.
\end{equation}
\end{cor}

\begin{proof} 
Suppose that $\psi^{+}(x) = 0 \in \bar{\cal{F}}$ for some $x \in (U_{r,s}^{+})_{\mu}$. Then in particular, 
we have $(\psi^{+}(\varphi(y)),\psi^{+}(x)) = 0$ for all $y \in (U_{r,s}^{-})_{-\mu}$, and therefore 
Theorem~\ref{thm:pairing_comparison} implies that $(\bar{y},\bar{x})_{H} = 0$ for all $y \in (U_{r,s}^{-})_{-\mu}$. 
However, since $(\cdot,\cdot)_{H}$ is non-degenerate (see Proposition~\ref{prop:nondegeneracy-Hopf}) and 
$y \mapsto \bar{y}$ is an algebra automorphism, we thus get $x=0$.

Now suppose that $\psi^{-}(y) = 0\in \bar{\cal{F}}'$ for some $y \in (U_{r,s}^{-})_{-\mu}$. As 
$\varphi^{-1} \circ \psi^{-} = \psi^{+} \circ \varphi^{-1}$, we have $\psi^{+}(\varphi^{-1}(y)) = 0$, so that 
$\varphi^{-1}(y) = 0$ by above. Thus $y = 0$ as claimed, since $\varphi\colon U_{r,s}^{+} \to U_{r,s}^{-}$ 
is an anti-isomorphism.
\end{proof}

Thus, the algebra homomorphism $\Delta_{r,s}\colon \bar{\cal{F}}\to \bar{\cal{F}} \otimes \bar{\cal{F}}$ induces the same-named 
homomorphism $\Delta_{r,s}\colon U^+_{r,s}\to U^+_{r,s} \otimes U^+_{r,s}$. Combining Proposition~\ref{prop:nondegeneracy-Hopf} 
with Theorem~\ref{thm:pairing_comparison} and using~\eqref{eq:F-vs-U}, we also get:

\begin{cor}\label{cor:nondegeneracy-CHW}
The pairing $(\cdot,\cdot)\colon U^+_{r,s}\times U^+_{r,s}\to \BC(r,s)$ is non-degenerate.    
\end{cor}


\section{Shuffle Algebras}\label{sec:shuffle}

In this Section, we introduce the two-parameter shuffle algebra $(\cal{F},*)$, relate it to the positive subalgebra 
$U^+_{r,s}$, and provide a shuffle interpretation of some of the structures on the latter. Our exposition closely 
follows that of the one-parameter setup from~\cite[Section 2]{L} and~\cite[Section 3]{CHW}.


\subsection{Two-parameter shuffle algebra}
\

Recall that $\cal{F}$ is the free associative $\BC(r,s)$-algebra generated by the finite alphabet 
$I = \{1,2, \ldots, n\}$, and $\cal{W}$ is the set of words in $I$. Recall also the notation 
$[i_{1} \dots i_{d}] = i_{1}i_{2} \dots i_{d}$ for the elements in $\cal{W}$, where $i_{1},\ldots ,i_{d} \in I$. 
As before, $\cal{F}$ has a natural $Q^{+}$-grading induced by declaring the degree of $[i]$ equal to $\alpha_{i}$. 
For a homogeneous element $x \in \cal{F}$, we write $|x|$ for the degree of $x$. For any $a,b \in \BC(r,s)$, we 
now define the quantum shuffle product $*_{a,b}\colon \cal{F}\times \cal{F}\to \cal{F}$ inductively via 
\begin{equation}\label{eq:shuffle_product_iterative}
  (xi) *_{a,b} (yj) = 
  (x *_{a,b} (yj))i + a^{-\langle |xi|,\alpha_{j}\rangle}b^{\langle\alpha_{j},|xi|\rangle} ((xi) *_{a,b} y)j ,\qquad 
  \emptyset *_{a,b} x = x *_{a,b} \emptyset = x,
\end{equation}
for all $i,j \in I$ and all homogeneous $x,y \in \cal{F}$. By iterating this definition, we find that 
\begin{equation}\label{eq:shuffle_product_general}
  [i_{1}\ldots i_{m}] *_{a,b} [i_{m+1}\ldots i_{m + d}] = 
  \sum_{\sigma} e_{a,b}(\sigma) [i_{\sigma^{-1}(1)}\ldots i_{\sigma^{-1}(m + d)}],
\end{equation}
where 
\begin{equation}\label{eq:e_ab}
  e_{a,b}(\sigma) = \prod_{\substack{k \le m < l \\ \sigma(k) < \sigma(l)}}   
  a^{-\langle \alpha_{i_{k}},\alpha_{i_{l}}\rangle} b^{\langle \alpha_{i_{l}},\alpha_{i_{k}}\rangle} 
\end{equation}
and the sum runs over all $(m,d)$-shuffles of $\{1,2,\ldots ,m + d\}$, i.e.\ the permutations $\sigma \in S_{m + d}$ 
such that $\sigma(1) < \sigma(2) < \dots <\sigma(m)$ and $\sigma(m + 1) < \dots < \sigma(m + d)$. There are four choices 
of $a,b$ that are of interest to us; in these cases, the inductive formula~\eqref{eq:shuffle_product_iterative} takes 
the following form:
\begin{align*}
  & (xi) *_{r,s} (yj) = (x *_{r,s} (yj))i + (\omega_{|xi|}',\omega_{j})^{-1} ((xi) *_{r,s} y)j, \\
  & (xi) *_{s,r} (yj) = (x *_{s,r} (yj))i + (\omega_{j}',\omega_{|xi|}) ((xi) *_{s,r} y)j, \\
  & (xi) *_{s^{-1},r^{-1}} (yj) = (x *_{s^{-1},r^{-1}} (yj))i + (\omega_{j}',\omega_{|xi|})^{-1} ((xi) *_{s^{-1},r^{-1}} y)j, \\
  & (xi) *_{r^{-1},s^{-1}} (yj) = (x *_{r^{-1},s^{-1}} (yj))i + (\omega_{|xi|}',\omega_{j}) ((xi) *_{r^{-1},s^{-1}} y)j,
\end{align*}
cf.~\eqref{eq:parity-convention}, and the corresponding expressions for $e_{a,b}(\sigma)$ of~\eqref{eq:e_ab} are: 
\begin{align*}
     e_{r,s}(\sigma) &= 
      \prod_{\substack{k \le m < l \\ \sigma(k) < \sigma(l)}} (\omega_{i_{k}}',\omega_{i_{l}})^{-1}, 
   & e_{s,r}(\sigma) &= 
      \prod_{\substack{k \le m < l \\ \sigma(k) < \sigma(l)}} (\omega_{i_{l}}',\omega_{i_{k}}), \\
     e_{s^{-1},r^{-1}}(\sigma) &= 
      \prod_{\substack{k \le m < l \\ \sigma(k) < \sigma(l)}} (\omega_{i_{l}}',\omega_{i_{k}})^{-1}, 
   & e_{r^{-1},s^{-1}}(\sigma) &= 
      \prod_{\substack{k \le m < l \\ \sigma(k) < \sigma(l)}} (\omega_{i_{k}}',\omega_{i_{l}}).
\end{align*}
Since the product structure $*_{r,s}$ will be used most frequently, we shall often omit the subscript for this operation, 
and just use the notation $*$ instead.

We have the following basic result, the proof of which is a direct computation using~\eqref{eq:shuffle_product_iterative}:

\begin{prop} 
The bilinear map $*_{a,b}\colon \cal{F} \times \cal{F} \to \cal{F}$ is associative.
\end{prop}

For the later use, we note that $*_{r,s}$ and $*_{s,r}$ are related via the following result:

\begin{prop}\label{prop:rs_to_sr}
For all homogeneous $x,y \in \cal{F}$, we have
\[
  x *_{r,s} y = (\omega_{|x|}',\omega_{|y|})^{-1}y *_{s,r} x.
\]
\end{prop}

\begin{proof} 
We proceed by induction on $m = \hgt(|x|) + \hgt(|y|)$. If $m = 0$ or $1$, then the result is obvious, because in these 
cases one of $x$ and $y$ must be the empty word. Thus, we may assume that $\hgt(|x|) \ge 1$ and $\hgt(|y|) \ge 1$, and 
that the result holds for all homogeneous elements $\tilde{x},\tilde{y}$ with $\hgt(|\tilde{x}|) + \hgt(|\tilde{y}|)  < m$. 
We may further assume that 
$x,y\in \cal{W}$. Then we may write $x = x'i$ and $y = y'j$ for some $i,j \in I$, and by induction we have:
\begin{align*}
  x *_{r,s} y 
  &= (x'i) *_{r,s} (y'j) = (x' *_{r,s} (y'j))i + (\omega_{|x|}',\omega_{j})^{-1} ((x'i) *_{r,s} y')j \\
  &= (\omega_{|x'|}',\omega_{|y|})^{-1} ((y'j) *_{s,r} x')i + 
     (\omega_{|x|}',\omega_{j})^{-1} (\omega_{|x|}',\omega_{|y'|})^{-1} (y' *_{s,r} (x'i))j \\
  &= (\omega_{|x|}',\omega_{|y|})^{-1} \left( (\omega_{i}',\omega_{|y|}) ((y'j) *_{s,r} x')i + (y' *_{s,r} (x'i))j \right) \\
  &= (\omega_{|x|}',\omega_{|y|})^{-1} (y'j) *_{s,r} (x'i) = (\omega_{|x|}',\omega_{|y|})^{-1} y *_{s,r} x.
\end{align*}
This completes the proof.
\end{proof}

Let $\pi_{+}\colon \cal{F} \to U_{r,s}^{+}$ and $\pi_{-}\colon \cal{F} \to U_{r,s}^{-}$ be the canonical algebra homomorphisms  
determined by $\pi_+(i)=e_i$ and $\pi_-(i)=f_i$ for $i\in I$. We note that by the definition of $\Delta_{r,s}$ on 
$\bar{\cal{F}}\simeq U_{r,s}^{+}$, cf.~\eqref{eq:F-vs-U}, we have 
  $$\Delta_{r,s}\pi_{+} = (\pi_{+} \otimes \pi_{+})\Delta_{r,s}.$$
For $w = [i_{1} \dots i_{d}]$ and any $P = \{k_{1} < \dots < k_{m}\} \subseteq \{1,2,\ldots ,d\}$, define 
$w_{P} = [i_{k_{1}} \dots i_{k_{m}}]$. We then have 
\[
  \Delta_{r,s}(w) = \sum_{P \subseteq \{1,2,\ldots ,d\}} z(P),
\]
where $z(P) = z_{1} \odot_{r,s} \dots \odot_{r,s} z_{d}$ with $z_{k} = i_{k} \otimes \emptyset$ when $k \in P$ and 
$z_{k} = \emptyset \otimes i_{k}$ when $k \in P^{c} = \{1,2,\ldots ,d\} \setminus P$. If $\sigma_{P}$ denotes the 
$(d - m,m)$-shuffle determined by $\sigma_{P}(d - m + i) = k_{i}$, then we have 
\[
  z(P) = e_{r,s}(\sigma_{P}) w_{P} \otimes w_{P^{c}}, 
\]
where $e_{r,s}(\sigma_{P})$ is the coefficient in the $\sigma_{P}$-th summand of $w_{P^{c}} * w_{P}$, cf.~\eqref{eq:shuffle_product_general}. 
This follows immediately from formula~\eqref{eq:twisted_product} for $\odot_{r,s}$ and the definition~\eqref{eq:e_ab} 
of $e_{r,s}(\sigma)$. Therefore, we have 
\[
  \Delta_{r,s}(w) = \sum_{P \subseteq \{1,2,\ldots ,d\}}e_{r,s}(\sigma_{P})w_{P} \otimes w_{P^{c}}.
\]
Let $\cal{F}^{*}$ be the graded dual of $\cal{F}$, and for each word $w \in \cal{W}$ we define $w^* \in \cal{F}^{*}$ by
\begin{equation}\label{eq:dual_vector}
  w^{*}(v) = \delta_{w,v}\quad \text{for all}\quad  v \in \cal{W}.
\end{equation}
Consider the product on $\cal{F}^{*}$ defined by:
\begin{equation*}
  (fg)(x) = (g \otimes f)(\Delta_{r,s}(x)).
\end{equation*}

\begin{lemma} 
The linear map $\phi\colon \cal{F}^{*} \to (\cal{F},*)$ defined by $w^{*} \mapsto w$ is a $\BC(r,s)$-algebra isomorphism.
\end{lemma}

\begin{proof} 
The map $\phi$ is clearly a vector space isomorphism, so we only need to show that $\phi(fg) = \phi(f) * \phi(g)$. 
For this, let $u = [i_{1} \dots i_{d}] \in \cal{W}$, $v = [i_{d + 1} \dots i_{d + m}] \in \cal{W}$, and let $w$ be 
any word of weight $|u| + |v|$. Then 
\[
  (u^{*}v^{*})(w) = 
  (v^{*} \otimes u^{*}) \left( \sum_{P \subseteq \{1,\ldots ,d + m\}} e_{r,s}(\sigma_{P})w_{P} \otimes w_{P^{c}} \right).
\]
If $\lambda_{u,v}^{w} = \sum e_{r,s}(\sigma_{P})$ with the sum over all $P \subseteq \{1,2,\ldots ,d + m\}$ 
satisfying $w_{P} = v, w_{P^{c}} = u$, then we get:
\begin{equation}\label{eq:product_fu}
  u^{*}v^{*} = \sum \lambda_{u,v}^{w}w^{*}.
\end{equation}
On the other hand, we have 
\[
  u * v = \sum_{\sigma}e_{r,s}(\sigma)[i_{\sigma^{-1}(1)}\ldots i_{\sigma^{-1}(d + m)}],
\]
so that the coefficient of any word $w\in \cal{W}$ of weight $|u| + |v|$ in the expansion of $u * v$ is precisely 
equal to $\sum e_{r,s}(\sigma)$, where the sum ranges over all $(d,m)$-shuffles $\sigma$ such that, if 
$P = \{\sigma(d + 1),\ldots ,\sigma(d +m)\}$, then $w_{P} = v$ and $w_{P^{c}} = u$. 
We thus get $u * v = \sum \lambda_{u,v}^{w}w$, which together with~\eqref{eq:product_fu} completes the proof.
\end{proof}

\begin{prop}\label{prop:shuffle_embedding} 
There is a unique $\BC(r,s)$-algebra homomorphism $\Psi\colon U_{r,s}^{+} \to (\cal{F},*)$ such that $\Psi(e_{i}) = i$. 
Moreover, $\Psi$ is injective. 
\end{prop}

\begin{proof} 
The quotient map $\pi_{+}\colon \cal{F} \to U_{r,s}^{+}$ induces an embedding of graded duals 
$\pi_{+}^{*}\colon (U_{r,s}^{+})^{*} \to \cal{F}^{*}$, where multiplication is defined by 
$(fg)(x) = (g \otimes f)(\Delta_{r,s}(x))$ in both cases. As the pairing $(\cdot,\cdot)$ on $\bar{\cal{F}}\simeq U_{r,s}^{+}$  
is non-degenerate by Corollary~\ref{cor:nondegeneracy-CHW}, we have a vector space isomorphism 
$\psi\colon U_{r,s}^{+} \to (U_{r,s}^{+})^{*}$ given by $\psi(x)(y) = (x,y)$ for all $x,y \in U_{r,s}^{+}$. 
Evoking that $(fg)(x) = (g \otimes f)(\Delta_{r,s}(x))$ for $f,g \in (U_{r,s}^{+})^{*}$, we thus obtain:
\begin{equation*}
  \psi(xx')(y) =  (xx',y) = (x' \otimes x, \Delta_{r,s}(y)) = (\psi(x') \otimes \psi(x))(\Delta_{r,s}(y)) = (\psi(x)\psi(x'))(y).
\end{equation*}  
This shows that the map $\psi$ is actually an algebra isomorphism. Now, define $\Psi = \phi \circ \pi_{+}^{*} \circ \psi$. 
Then $\Psi$ is an algebra embedding, and since 
  $((\pi_{+}^{*} \circ \psi)(e_{i}))(j_1 \ldots j_d) = \psi(e_{i})(e_{j_1}\cdots e_{j_d}) = (e_{i},e_{j_1}\cdots e_{j_d}) = 
   \delta_{[i], [j_1\ldots j_d]}$, 
it also follows that $\Psi(e_{i}) = i$.
\end{proof}

We shall now give an alternative description of the above map $\Psi$, making use of the operators $\partial_{i}'$ 
introduced in~\eqref{eq:coproduct_structure}. For each word $w = [i_{1}\ldots i_{d}] \in \cal{W}$, we define 
\[
  \partial_{w}' = \partial_{i_{1}}'\partial_{i_{2}}'\ldots \partial_{i_{d}}'.
\]
We then define a $\BC(r,s)$-linear map $\Upsilon\colon U_{r,s}^{+} \to \cal{F}$ by 
\begin{equation}\label{eq:Upsilon}
  \Upsilon(u) = \sum_{w \in \cal{W}_{\mu}}\partial_{w}'(u)w \qquad \text{for}\quad u \in (U_{r,s}^{+})_{\mu}.
\end{equation}
We will show in Proposition~\ref{prop:psi_equals_upsilon} below that this map coincides with the map $\Psi$ of 
Proposition~\ref{prop:shuffle_embedding}. To do so, we need to introduce analogues of the operators $\partial_{i}'$ 
for $(\cal{F},*)$, which is the content of the following lemma:

\begin{lemma}\label{lem:varepsilon_i}
For each $i = 1,2,\ldots ,n$, define the $\BC(r,s)$-linear map $\epsilon_{i}'\colon \cal{F} \to \cal{F}$ by 
\[
  \epsilon_{i}'([i_{1} \dots i_{d}]) = \delta_{i,i_{d}} [i_{1} \dots i_{d-1}] ,\qquad 
  \epsilon_{i}'(\emptyset) = 0.
\]
Then $\epsilon_{i}'(j) = \delta_{ij}$ and we have
\[
  \epsilon_{i}'(x * y) = \epsilon_{i}'(x) * y + (\omega_{|x|}',\omega_{i})^{-1} (x * \epsilon_{i}'(y))
\]
for all homogeneous $x,y \in \cal{F}$.
\end{lemma}

\begin{proof} 
The equality $\epsilon_{i}'(j) = \delta_{ij}$ is obvious. For the latter identity, it suffices to assume that 
$x,y\in \cal{W}$. If one of $x,y$ has length zero, then the formula is obvious. 
Otherwise, we may write $x = x'j, y = y'k$. Then: 
\begin{align*}
  \epsilon_{i}'(x * y) 
  &= \epsilon_{i}'((x' * (y'k))j + (\omega_{|x|}',\omega_{k})^{-1} ((x'j) * y')k) \\
  &= \delta_{ij}(x' * (y'k)) + \delta_{ik}(\omega_{|x|}',\omega_{k})^{-1} ((x'j) * y') 
   = \epsilon_{i}'(x) * y + (\omega_{|x|}',\omega_{i})^{-1}(x * \epsilon_{i}'(y)),
\end{align*}
as desired.
\end{proof}

For a word $w = [i_{1} \dots i_{d}]$, we also define $\epsilon_{w}'\colon \cal{F}\to \cal{F}$ via
\begin{equation}\label{eq:varepsilon_word}
  \epsilon_{w}' = \epsilon_{i_{1}}'\epsilon_{i_{2}}' \dots \epsilon_{i_{d}}'.
\end{equation}
Then for any word $v \in \cal{W}$, we have $\epsilon_{w}'(v) = \delta_{w,v}$.

\begin{prop}\label{prop:psi_equals_upsilon} 
The map $\Upsilon\colon U_{r,s}^{+} \to (\cal{F},*)$ of~\eqref{eq:Upsilon} is an injective algebra homomorphism 
satisfying $\Upsilon(e_{i}) = i$ for all $i \in I$, and hence it coincides with the map $\Psi$ of 
Proposition~\ref{prop:shuffle_embedding}.
\end{prop}

\begin{proof} 
First, we note that $\Psi \circ \partial_{i}' = \epsilon_{i}' \circ \Psi$, by~\eqref{eq:Leibniz_partial} and 
Lemma~\ref{lem:varepsilon_i}. Therefore, if $u \in (U_{r,s}^{+})_{\mu}$, $w \in \cal{W}_{\mu}$, and $\gamma_{w}(u)$ 
is the coefficient of $w$ in $\Psi(u)$, then
\[
  \gamma_{w}(u) = \epsilon_{w}'(\Psi(u)) = \Psi(\partial_{w}'(u)) = \partial_{w}'(u)\Psi(1) = \partial_{w}'(u).
\]
This shows that $\Psi = \Upsilon$ and completes the proof.
\end{proof}

Let $\cal{U}$ denote the image of $\Psi$, that is $\cal{U} = \Psi(U_{r,s}^{+})$, which is the subalgebra of $(\cal{F},*_{r,s})$ generated by $I$.

\begin{prop}\label{prop:image_description} 
The element $x = \sum_{w \in \cal{W}}\gamma(w)w \in \cal{F}$ lies in $\cal{U}$ if and only if
\begin{equation}\label{eq:image_description}
  \sum_{k = 0}^{1 - a_{ij}} (-1)^{k} \qbinom{1 - a_{ij}}{k}_{r_{i},s_{i}} (r_{i}s_{i})^{\frac{1}{2}k(k - 1)}
    (rs)^{k\langle \alpha_{j},\alpha_{i}\rangle} \gamma(z[i]^{k}[j][i]^{1 - a_{ij} - k}t) = 0 
\end{equation}
for all $i \neq j$ and $z,t \in \cal{W}$. 
\end{prop}

\begin{proof} 
Let $K$ be the $\BC(r,s)$-subspace of $\cal{F}$ spanned by the set of elements $\sum_{w \in \cal{W}}\gamma(w)w$ 
satisfying~\eqref{eq:image_description}. For any $u \in (U_{r,s}^{+})_{\mu}$, consider 
\[
  x = \Psi(u) = \sum_{|w| = \mu} \gamma_w(u)w.
\]
Then, for any word $w = [i_{1}\ldots i_{d}]$ with $|w| = \mu$, Proposition~\ref{prop:psi_equals_upsilon} and 
Lemma~\ref{lem:symm-der-adjoint} imply that
\[
  \gamma_w(u) = \partial_{i_{1}}'\ldots \partial_{i_{d}}'(u) = (e_{i_{1}}\ldots e_{i_{d}},u).
\]
Therefore $x \in K$ according to Theorem~\ref{thm:serre}, so that $\cal{U} \subseteq K$.

To prove the other inclusion, consider the linear map $L\colon \cal{F} \to \cal{F}^{*}$ defined by $w \mapsto w^{*}$ for 
$w \in \cal{W}$, where $w^{*}$ was defined in~\eqref{eq:dual_vector}. Then, $f \in K$ if and only if we have $L(x)(f) = 0$ 
for all $x \in \ker(\pi_{+})$, since $\ker(\pi_{+})$ is generated by $\{S_{ij}\}_{i,j\in I}$ of~\eqref{eq:Sij} due 
to~\eqref{eq:F-vs-U}. Thus, it follows that $\dim(K_{\mu}) = \dim(U_{r,s}^{+})_{\mu}$ for any $\mu \in Q^{+}$. But since 
$\Psi\colon U_{r,s}^+ \to \cal{U}$ is an isomorphism, we also have $\dim(\cal{U}_{\mu}) = \dim(U_{r,s}^{+})_{\mu}$ 
for all $\mu \in Q^{+}$.

Therefore, we must actually have the equality $K = \cal{U}$.
\end{proof}


\subsection{Additional structures}
\

\begin{prop}\label{prop:tau_and_bar}
(1) Let $\tau\colon \cal{F} \to \cal{F}$ be the $\BC$-linear map defined by $\tau(r)=s^{-1}, \tau(s)=r^{-1}$, and 
\[
  \tau([i_{1} \dots i_{d}]) = [i_{d} \dots i_{1}].
\]
Then, $\tau(x *_{r,s} y) = \tau(y) *_{r,s} \tau(x)$ for all $x,y \in \cal{F}$.

\medskip
\noindent
(2) Let $x \mapsto \bar{x}$ be the $\BC$-linear map $\cal{F} \to \cal{F}$ defined by $\bar{r} = s$, $\bar{s} = r$, and 
\[
  \ol{[i_{1}\ldots i_{d}]} = \left( \prod_{k < l} (\omega_{i_{l}}',\omega_{i_{k}})^{-1} \right) [i_{d}\ldots i_{1}].
\]
Then, $\ol{x *_{r,s} y} = \bar{x} *_{r,s} \bar{y}$ for all $x,y \in \cal{F}$.
\end{prop}

\begin{proof}
To prove part (1), let $x = [i_{1} \dots i_{m}]$ and $y = [i_{m + 1} \dots i_{m + d}]$. Then 
\begin{equation}\label{eq:tau-rhs-1}
\begin{split}
  \tau([i_{1} \dots i_{m}] *_{r,s} [i_{m + 1} \dots i_{m + d}])
  &= \tau \left (\sum_{\sigma}e_{r,s}(\sigma)[i_{\sigma^{-1}(1)}\ldots i_{\sigma^{-1}(m + d)}] \right ) \\ 
  &= \sum_{\sigma} \prod_{\substack{k \le m < l \\ \sigma(k) < \sigma(l)}} \tau\Big((\omega'_{i_k},\omega_{i_l})^{-1}\Big) 
     \, [i_{\sigma^{-1}(m + d)}\ldots i_{\sigma^{-1}(1)}] \\
  &= \sum_{\sigma'} \prod_{\substack{k \le m < l \\ \sigma'(k) > \sigma'(l)}} 
     (\omega_{i_{l}}',\omega_{i_{k}})^{-1} [i_{(\sigma')^{-1}(1)}\ldots i_{(\sigma')^{-1}(m + d)}],
\end{split}
\end{equation}
where the first two sums are taken over all $(m,d)$-shuffles $\sigma\in S_{m+d}$, the last sum is taken over all 
$\sigma' \in S_{m + d}$ such that $\sigma'(m + d) < \sigma'(m + d-1) < \dots < \sigma'(m + 1)$ and 
$\sigma'(m) < \sigma'(m-1) < \dots < \sigma'(1)$, and the final equality is obtained by matching these 
$\{\sigma\}$ and $\{\sigma'\}$ via $\sigma \mapsto \sigma'=w_{0}\sigma$ with the longest element $w_{0}\in S_{m+d}$.

On the other hand, we have:
\begin{equation}\label{eq:tau-rhs-2}
\begin{split}
  [i_{m + d}\ldots i_{m + 1}] *_{r,s}[i_{m} \ldots i_{1}]  
  &= \sum_{\tilde{\sigma}} \prod_{\substack{\tilde{k} \le d < \tilde{l} \\ \tilde{\sigma}(\tilde{k}) < \tilde{\sigma}(\tilde{l})}} 
     (\omega_{i_{m+d+1-\tilde{k}}}',\omega_{i_{m+d+1-\tilde{l}}})^{-1} 
     [i_{(\tilde{\sigma}w_0)^{-1}(1)} \ldots i_{(\tilde{\sigma}w_0)^{-1}(m+d)}] \\
  &= \sum_{\tilde{\sigma}} \prod_{\substack{k \le m < l \\ \tilde{\sigma}w_0(l) < \tilde{\sigma}w_0(k)}} 
     (\omega_{i_{l}}',\omega_{i_{k}})^{-1} [i_{(\tilde{\sigma}w_0)^{-1}(1)} \ldots i_{(\tilde{\sigma}w_0)^{-1}(m+d)}],
\end{split}
\end{equation}
where both sums are taken over all $(d,m)$-shuffles $\tilde{\sigma}\in S_{m+d}$, and the indexes $k,l$ in the latter sum 
are related to $\tilde{k},\tilde{l}$ in the former via $k=m+d+1-\tilde{l}=w_0(\tilde{l})$, $l=m+d+1-\tilde{k}=w_0(\tilde{k})$. 
The right-hand sides of~(\ref{eq:tau-rhs-1}) and~(\ref{eq:tau-rhs-2}) match up under a natural bijection between the corresponding 
$\{\tilde{\sigma}\}$ and $\{\sigma'\}$ given by $\sigma'=\tilde{\sigma}w_0$. This completes the proof of part (1).

For part (2), we first note that for any homogeneous $x',y' \in \cal{F}$ and $i,j \in I$, we have by part~(1):
\begin{equation}\label{eq:shuffle_opp}
\begin{split}
  (ix') *_{r,s} (jy') 
  &= \tau((\tau(y')j) *_{r,s} (\tau(x')i)) \\
  &= \tau\left( (\tau(y') *_{r,s}(\tau(x')i))j + (\omega_{|y'j|}',\omega_{i})^{-1}((\tau(y')j) *_{r,s} \tau(x'))i \right) \\
  &= j((ix') *_{r,s} y') + (\omega_{i}',\omega_{|jy'|})^{-1}i(x' *_{r,s} (jy')).
\end{split}
\end{equation}
For $x,y \in \cal{W}$, we proceed by induction on $m = \hgt(|x|) + \hgt(|y|)$. Note first that if either $\hgt(|x|) = 0$ or 
$\hgt(|y|) = 0$, then the claim is trivial. In particular, the assertion holds for $m = 0$ and $m = 1$. For the step of induction, 
suppose that $\hgt(|x|),\hgt(|y|) \ge 1$, $m = \hgt(|x|) + \hgt(|y|)$, and the claim holds for all $x',y' \in \cal{W}$ with 
$\hgt(|x'|) + \hgt(|y'|) < m$. By assumption, there are $i,j \in I$ and $x',y' \in \cal{W}$ such that $x = ix'$ and $y = jy'$, 
and we have $\hgt(|x'|) + \hgt(|y'|) = m - 2$. Thus, combining~\eqref{eq:shuffle_opp} with the equality 
$\ol{(\omega'_\mu,\omega_\nu)^{-1}}=(\omega'_\nu,\omega_\mu)$ and the induction hypothesis, we obtain:
\begin{align*}
  \ol{(ix') *_{r,s} (jy')} 
  &= \ol{j((ix') *_{r,s} y')} + \ol{(\omega_{i}',\omega_{|jy'|})^{-1}i(x' *_{r,s} (jy'))} \\
  &= (\omega_{|ix'|}'\omega_{|y'|}',\omega_{j})^{-1} \left (\ol{(ix') *_{r,s} y'}\right )j + 
     (\omega_{|jy'|}',\omega_{i})(\omega_{|x'|}'\omega_{|jy'|}',\omega_{i})^{-1} \left(\ol{x' *_{r,s} (jy')} \right)i \\
  &= (\omega_{|ix'|}'\omega_{|y'|}',\omega_{j})^{-1} (\ol{ix'} *_{r,s} \bar{y}')j + 
     (\omega_{|x'|}',\omega_{i})^{-1} (\bar{x}' *_{r,s} \ol{jy'})i \\
  &= (\omega_{|ix'|}'\omega_{|y'|}',\omega_{j})^{-1} (\omega_{|x'|}',\omega_{i})^{-1}((\bar{x}'i) *_{r,s} \bar{y}')j + 
     (\omega_{|x'|}',\omega_{i})^{-1}(\omega_{|y'|}',\omega_{j})^{-1} (\bar{x}'*_{r,s}(\bar{y}'j))i \\
  &= (\omega_{|x'|}',\omega_{i})^{-1}(\omega_{|y'|}',\omega_{j})^{-1} \left((\bar{x}'i) *_{r,s} (\bar{y}'j) \right) 
   = \ol{ix'} *_{r,s} \ol{jy'},
\end{align*}
which completes the proof.
\end{proof}

Comparing the above result with Proposition~\ref{prop:extra_structures}, we obtain:

\begin{cor}\label{cor:auto_matching} 
(1) For all $u \in U_{r,s}^{+}$, we have $\tau\Psi(u) = \Psi\tau(u)$, where $\tau\colon U_{r,s}^{+} \to U_{r,s}^{+}$ 
is the $\BC$-algebra anti-automorphism defined in Proposition~\ref{prop:extra_structures}(2).

\medskip
\noindent
(2) For all $u \in U_{r,s}^{+}$, we have $\ol{\Psi(u)} = \Psi(\bar{u})$, where $\bar{~}\colon U_{r,s}^{+} \to U_{r,s}^{+}$ 
is the $\BC$-algebra automorphism defined in Proposition~\ref{prop:extra_structures}(3).
\end{cor}

Finally, we equip $\cal{F}$ with a coproduct:

\begin{prop}\label{prop:shuffle_coproduct} 
Let $\Delta\colon \cal{F} \to \cal{F} \otimes \cal{F}$ be the linear map defined by 
\[
  \Delta([i_{1} \dots i_{d}]) = \sum_{0 \le k \le d} [i_{k+1}\ldots i_{d}] \otimes [i_{1}\ldots i_{k}].
\]
Then, $\Delta(x * y) = \Delta(x) * \Delta(y)$, where we define the shuffle product $*$ on $\cal{F} \otimes \cal{F}$ via  
\begin{equation}\label{eq:twisted_shuffle_square}
  (w \otimes x) * (y \otimes z) = (\omega_{|x|}',\omega_{|y|})^{-1} (w * y) \otimes (x * z).
\end{equation}
Furthermore, we have $\Delta \Psi = (\Psi \otimes \Psi)\Delta_{r,s}$.
\end{prop}

\begin{proof} 
It is enough to prove the claim when $x,y \in \cal{W}$. In this case, we write 
$\Delta(x) = \sum_{(x)}x_{1} \otimes x_{2}$ and $\Delta(y) = \sum_{(y)}y_{1} \otimes y_{2}$.
We note that for any $i \in I$, we have: 
\[
  \Delta(xi) = \Delta(x)\cdot (i \otimes 1) + 1 \otimes xi = \sum_{(x)}x_{1}i \otimes x_{2} + 1 \otimes xi,
\]
where $\cdot$ denotes the term-wise multiplication on $\cal{F} \otimes \cal{F}$:
\begin{equation}\label{eq:dot-multiplication}
  (w \otimes x)\cdot (y \otimes z) = wy \otimes xz.
\end{equation}

We shall now proceed by induction on $\hgt(|x|) + \hgt(|y|) = m$. The cases $m = 0$, $m = 1$, as well as when one of $x,y$ 
has length zero are trivial. Suppose now that $x,y \in \cal{W}$ satisfy $\hgt(|x|),\hgt(|y|) \ge 1$ and $m = \hgt(|x|) + \hgt(|y|)$, 
and the claim holds for all $x',y' \in \cal{W}$ with $\hgt(|x'|) + \hgt(|y'|) < m$. By this assumption, we can write 
$x = x'i$ and $y = y'j$ for some $i,j \in I$ and $x',y' \in \cal{W}$. Then, we have:
\begin{align*}
  & \Delta((x'i) * (y'j)) 
   =\Delta((x' * (y'j))i) + \Delta((\omega_{|x'i|}',\omega_{j})^{-1} ((x'i) * y')j) \\
  &= \Delta(x' * (y'j))\cdot (i \otimes 1) + 1 \otimes (x' * (y'j))i 
     + (\omega_{|x'i|}',\omega_{j})^{-1} 
       \left( \Delta((x'i) * y')\cdot (j \otimes 1) + 1 \otimes ((x'i) * y')j \right), 
\end{align*}
so the induction hypothesis implies that
\begin{align*}
  \Delta((x'i) * (y'j)) 
  &= (\Delta(x') * \Delta(y'j))\cdot (i \otimes 1) + (\omega_{|x'i|}',\omega_{j})^{-1}(\Delta(x'i) * \Delta(y'))\cdot (j \otimes 1) 
     + 1 \otimes ((x'i) * (y'j)) \\
  &= \left( \left( \sum_{(x')} x_{1}' \otimes x_{2}' \right) * 
            \left( \sum_{(y')} (y_{1}'j) \otimes y_{2}' + 1 \otimes (y'j) \right) \right)
      \cdot (i \otimes 1) \\
  & \quad + (\omega_{|x'i|}',\omega_{j})^{-1}
    \left( \left( \sum_{(x')} (x_{1}'i) \otimes x_{2}' + 1 \otimes (x'i) \right) * 
           \left( \sum_{(y')} y_{1}' \otimes y_{2}' \right) \right) 
    \cdot (j \otimes 1) \\ 
  & \quad + 1 \otimes ((x'i) * (y'j)) \\
  &= \sum_{(x')(y')} (\omega_{|x_{2}'|}',\omega_{|y_{1}'j|})^{-1} (x_{1}' * (y_{1}'j))i \otimes (x_{2}' * y_{2}') 
   + \sum_{(x')} (x_{1}'i) \otimes (x_{2}' * (y'j)) \\
  &\quad + (\omega_{|x'i|}',\omega_{j})^{-1} \sum_{(x')(y')} (\omega_{|x_{2}'|}',\omega_{|y_{1}'|})^{-1}
    ((x_{1}'i) * y_{1}')j \otimes (x_{2}' * y_{2}') \\
  &\quad + (\omega_{|x'i|}',\omega_{j})^{-1} \sum_{(y')} (\omega_{|x'i|}',\omega_{|y_{1}'|})^{-1} (y_{1}'j) \otimes ((x'i) * y_{2}') 
    + 1 \otimes ((x'i) * (y'j)) \\
  &= \sum_{(x')(y')} (\omega_{|x_{2}'|}',\omega_{|y_{1}'j|})^{-1} ((x_{1}'i) * (y_{1}'j)) \otimes (x_{2}' * y_{2}') \\
  &\quad + \left( \sum_{(x')} (x_{1}'i) \otimes x_{2}' \right) * (1 \otimes (y'j)) 
         + (1 \otimes (x'i)) * \left( \sum_{(y')} (y_{1}'j) \otimes y_{2}' \right) 
         + 1 \otimes ((x'i) * (y'j)) \\
  &= \left( \sum_{(x')} (x_{1}'i) \otimes x_{2}' + 1 \otimes (x'i) \right) * 
     \left( \sum_{(y')} (y_{1}'j) \otimes y_{2}' + 1 \otimes (y'j) \right)  
   = \Delta(x'i) * \Delta(y'j).
\end{align*}
This completes the proof of $\Delta(x * y) = \Delta(x) * \Delta(y)$.

As per the equality $\Delta \Psi = (\Psi \otimes \Psi)\Delta_{r,s}\colon U^+_{r,s}\to \cal{F}\otimes \cal{F}$, 
it suffices to verify its validity on the generators, where it immediately follows from 
$\Delta(\Psi(e_i)) = [i]\otimes \emptyset + \emptyset \otimes [i]  = (\Psi \otimes \Psi)(\Delta_{r,s}(e_i))$ for each $i\in I$. 
\end{proof}

Identifying $\cal{U}$ with $U^+_{r,s}$, we obtain a non-degenerate pairing $(\cdot,\cdot)\colon \cal{U}\times \cal{U} \to \BC(r,s)$, 
cf.~Corollary~\ref{cor:nondegeneracy-CHW}.


\section{Orthogonal Bases}\label{sec:orthogonal_bases}

This Section closely follows~\cite[Sections 4--5]{CHW}, which in turn is largely based on~\cite{L}. 
So we shall only highlight the key changes in the present setup.


\subsection{Dominant and Lyndon words}
\

From now on, we fix an ordering $\leq$ on the alphabet $I$, which induces a lexicographical order on the monoid~$\cal{W}$.  
For a nonzero $x\in \cal{F}$, its \textbf{leading term} $\max(x)$, is a word $w\in \cal{W}$ such that 
$x=\sum_{u\leq w} t_u\cdot u$ with $t_u\in \BC(r,s)$ and $t_w\ne 0$. Following the terminology of~\cite[\S4.1]{CHW}, 
we call a word $w\in \cal{W}$ \textbf{dominant} if it appears as the leading term of some element from $\cal{U} = \Psi(U_{r,s}^{+})$. 
We use $\cal{W}^{+}$ to denote the subset of $\cal{W}$ consisting of all dominant words. Then we have the following 
basic result, proved exactly as in~\cite[Proposition~12]{L}, cf.~\cite[Proposition~4.1]{CHW}:

\begin{prop}\label{prop:homogeneous_basis}
(1) There is a unique basis of homogeneous vectors $\{m_{w} \,|\, w \in \cal{W}^{+}\}$ in $\cal{U}$ 
such that for all $w_{1},w_{2} \in \cal{W}^{+}$ with $|w_{1}| = |w_{2}|$, we have (cf.~\eqref{eq:varepsilon_word}):
\[
  \epsilon_{w_{1}}'(m_{w_{2}}) = \delta_{w_{1},w_{2}}.
\]

\noindent
(2) The set $\big\{e_{w} = e_{i_{1}}\ldots e_{i_{d}} \,\big|\, w = [i_{1} \dots i_{d}] \in \cal{W}^{+}\big\}$ 
is a basis of $U_{r,s}^{+}$.
\end{prop}

A word $w=[i_1 \dots i_d]$ is called \textbf{Lyndon} if it is smaller than any of its proper right factors:
\begin{equation*}
  w < [i_k \dots i_d] \qquad \forall\, 1<k\leq d.
\end{equation*}
We use $\cal{L}$ to denote the set of all Lyndon words. 
It is well~known that any word $w$ admits a unique factorization (called a \textbf{canonical factorization}) as a product of non-increasing Lyndon words:
\begin{equation}\label{eq:canonical_factorization}
  w = \ell_{1}\ell_{2}\ldots \ell_{k}, \quad \ell_1\geq \ell_2 \geq \dots \geq \ell_k, \quad \ell_1, \dots, \ell_k\in \cal{L}.
\end{equation}

The following Lemma is an extension of~\cite[Lemma~15]{L} (which covers the case $k = 1$):

\begin{lemma}\label{lem:shuffle_lw} 
For any $\ell \in \cal{L}$ and $w \in \cal{W}$ with $\ell \ge w$, we have $\max(\ell^{k} * w) = \ell^{k} w$ for all $k \ge 1$.
\end{lemma}

\begin{proof} 
We proceed by induction on $k$. The base case $k = 1$ is analogous to~\cite[Lemma~4.5]{CHW}. Suppose that $k > 1$, and the result holds for all smaller values of $k$. Since all terms in $\ell^{k} * w$ appear with coefficients in $\bb{Z}_{\ge 0}[r^{\pm 1},s^{\pm 1}]$, it follows that the coefficient of $\ell^{k}w$ in $\ell^{k} * w$ is nonzero. Thus, if $u$ is a word in $\ell^{k} * w$, it suffices to show that $u \le \ell^{k}w$. Given such a word $u$, there is some factorization $w = w_{1}w_{2}$ (with $w_{1}$ or $w_{2}$ possibly empty) such that $u$ occurs in $(\ell^{k - 1} * w_{1})(\ell * w_{2})$. Then since $\ell \ge w \ge w_{1}$, the induction hypothesis implies that $\max(\ell^{k - 1} * w_{1}) = \ell^{k - 1}w_{1}$, and therefore $u$ is less than or equal to all words appearing in $\ell^{k-1}w_{1}(\ell * w_{2})$. Then since every word that appears in $w_{1}(\ell * w_{2})$ also appears in $\ell * w$, and the case $k = 1$ implies that $\max(\ell * w) = \ell w$, all words that appear in $\ell^{k - 1}w_{1}(\ell * w_{2})$ are less than or equal to $\ell^{k}w$. This implies that $u \le \ell^{k}w$, which completes the proof.
\end{proof}

Let $\cal{L}^{+} = \cal{W}^{+} \cap \cal{L}$ be the set of all dominant Lyndon words. Then Lemma~\ref{lem:shuffle_lw} implies 
the following result, analogous to~\cite[Proposition 16]{L} (cf.~\cite[Proposition 4.7]{CHW}):

\begin{prop} 
If $\ell \in \cal{L}^{+}$ and $w \in \cal{W}^{+}$ satisfy $\ell \ge w$, then $\ell w \in \cal{W}^{+}$.
\end{prop}

Completely analogously to~\cite[Proposition 17]{L}, we also have:

\begin{prop}\label{prop:dominant-factorization} 
A word $w \in \cal{W}$ is dominant if and only if it has the form 
\[
  w = \ell_{1} \ell_{2} \dots \ell_{k}, \qquad \ell_{1} \ge \ell_{2} \ge \dots \ge \ell_{k},
\]
where $\ell_{1}, \dots, \ell_{k}$ are dominant Lyndon words 
(such a decomposition is unique and coincides with~\eqref{eq:canonical_factorization}).
\end{prop}

Finally, due to our earlier dimension count from Proposition~\ref{prop:graded-dimension}, we also have 
the following analogue of \cite[Proposition~18]{L} (cf.~\cite[Theorem~4.8]{CHW}):

\begin{theorem}\label{thm:LR-bijection} 
The map $\ell \mapsto |\ell|$ defines a bijection from $\cal{L}^{+}$ to $\Phi^{+}$.
\end{theorem}

We shall denote the inverse of this bijection by $\ell \colon\Phi^{+} \to \cal{L}^{+}$.


\subsection{Bracketing}
\

For homogeneous elements $x,y \in \cal{F}$, define their \textbf{$(r,s)$-bracketing} 
\[
  [x,y]_{r,s} = xy - (\omega_{|y|}',\omega_{|x|})yx.
\]
For a Lyndon word $\ell\in \cal{L}$, a decomposition $\ell=\ell_1\ell_2$ is called a \textbf{costandard factorization} 
if $\ell_1,\ell_2$ are nonempty, $\ell_1\in \cal{L}$, and the length of $\ell_1$ is the maximal possible. In this case, 
it is known that $\ell_2$ is also a Lyndon word. Following~\cite[\S4.1]{L} and~\cite[\S4.3]{CHW}, given a Lyndon word 
$\ell$ we define its bracketing $[\ell]\in \cal{F}$ inductively via:
\begin{itemize}

\item 
$[\ell] = \ell$ if $\ell \in I$,

\item 
$[\ell] = [[\ell_{1}],[\ell_{2}]]_{r,s}$ if $\ell = \ell_{1}\ell_{2}$ is the costandard factorization of $\ell$.

\end{itemize}
Evoking the canonical factorization of~\eqref{eq:canonical_factorization}, we define the bracketing of 
any word $w\in \cal{W}$ via:
\begin{equation*}
  [w] = [\ell_{1}][\ell_{2}] \dots [\ell_{k}]. 
\end{equation*}
Finally, we also define $\Xi\colon (\cal{F},\cdot) \to (\cal{F},*)$ as the algebra homomorphism given by 
$\Xi([i_{1} \dots i_{d}]) = i_{1} * \dots *i_{d}$. Then we have the following three results, whose proofs 
are exactly the same as those of Proposition~19, Proposition~20, and Lemma~21 in~\cite{L}:

\begin{prop} 
For any $\ell \in \cal{L}$, we have $[\ell] = \ell + x$, where $x$ is a linear combination of words $v>\ell$.
\end{prop}

\begin{prop} 
The set $\big\{[w] \,\big|\, w \in \cal{W}\big\}$ is a basis for $\cal{F}$.
\end{prop}

\begin{lemma} 
A word $w \in \cal{W}$ is dominant if and only if it cannot be expressed modulo $\ker(\Xi)$ as 
a linear combination of words $v > w$.
\end{lemma}

For any dominant word $w \in \cal{W}^{+}$, we define 
\begin{equation*}
  R_{w} = \Xi([w]).
\end{equation*}
For any $x,y \in \cal{F}$, let us introduce the following notation: 
\[
  x \circledast y = x *_{r,s} y - x *_{s,r} y.
\]
Recall from Proposition~\ref{prop:rs_to_sr} that $x *_{r,s} y = (\omega_{|x|}',\omega_{|y|})^{-1}y *_{s,r}x$, hence we have 
\begin{equation*}
  x \circledast y = x *_{r,s} y - (\omega_{|y|}',\omega_{|x|})y *_{r,s}x.
\end{equation*} 
This formula immediately implies:

\begin{lemma} 
Let $\ell \in \cal{L}^{+}$, and let $\ell = \ell_{1}\ell_{2}$ be its costandard factorization. Then 
\begin{equation*}
  R_{\ell} = R_{\ell_{1}} \circledast R_{\ell_{2}}.
\end{equation*}
\end{lemma}

For any word $w =[i_{1} \dots i_{d}]$, set 
\begin{equation*}
  \epsilon_{w} = i_{1} * i_{2} * \dots * i_{d}.
\end{equation*}
Then completely analogously to~\cite[Proposition~4.13]{CHW}, we obtain:

\begin{prop}\label{prop:triangularity}
For $w \in \cal{W}^{+}$, we have 
\[
  R_{w} = \epsilon_{w} \, + \sum_{v \in \cal{W}^{+}}^{v > w} \chi_{wv}\epsilon_{v}
\]
for some $\chi_{wv} \in \BC(r,s)$. In particular, the set $\{R_{w} \,|\, w \in \cal{W}^{+}\}$ is a basis for $\cal{U}$.
\end{prop}

We call $\{R_{w} \,|\, w \in \cal{W}^{+}\}$ the \textbf{Lyndon basis} of $\cal{U}$. 
Due to Proposition~\ref{prop:dominant-factorization}, we have (cf.~\cite[Theorem~23]{L}):

\begin{prop}\label{prop:lyndon_basis} 
The Lyndon basis has the form 
\[
  \big\{ R_{\ell_{1}} * \dots * R_{\ell_{k}} \,\big|\, 
         k \in \BZ_{\geq 0},\ \ell_{1}, \dots, \ell_{k} \in \cal{L}^{+},\ \ell_{1} \ge \dots \ge \ell_{k} \big\}.
\]
\end{prop}

Let us also recall \textbf{Leclerc's algorithm} for computing the set $\cal{L}^{+}$ of dominant Lyndon words. 
For each $\beta \in \Phi^{+}$, let 
\[
  C(\beta) = 
  \big\{ (\beta_{1},\beta_{2}) \in \Phi^{+} \times \Phi^{+} \,|\, \beta_{1} + \beta_{2} = \beta, \ell(\beta_{1}) < \ell(\beta_{2}) \big\},
\]
where $\ell\colon \cal{L}^{+} \to \Phi^{+}$ denotes the inverse of the bijection in Theorem~\ref{thm:LR-bijection}. 
Then we have:

\begin{prop}\label{prop:leclerc_algorithm} 
For any $\beta \in \Phi^{+}$, we have $\ell(\beta) = \max\{\ell(\beta_{1})\ell(\beta_{2}) \,|\, (\beta_{1},\beta_{2}) \in C(\beta)\}$.
\end{prop}

\begin{proof} 
Because there is a unique dominant Lyndon word of weight $\beta$ for each $\beta \in \Phi^{+}$, it will suffice to show that the set $\cal{L}^{+}$ of dominant Lyndon words coincides with the set $\cal{GL}$ of good Lyndon words considered in~\cite{L}; then the Proposition is just a restatement of~\cite[Proposition 25]{L}.

To avoid confusion with our notation, we will denote the shuffle product $*$ on $\cal{F}$ defined in~\cite{L} by $*_{q}$, and we will write $\cal{U}^{q}$ for the image of the embedding of $U_{q}^{+}$ into $\cal{F}$ constructed in~\cite{L}.

Now, let $\cal{A}_{r,s} = \bb{C}[r^{\pm 1},s^{\pm 1}]$ and let $\cal{A}_{q} = \bb{C}[q,q^{-1}]$. Consider the free $\cal{A}_{r,s}$-module $\cal{F}_{\cal{A}_{r,s}} = \bigoplus_{w \in \cal{W}}\cal{A}_{r,s}w$ and the free $\cal{A}_{q}$-module $\cal{F}_{\cal{A}_{q}} = \bigoplus_{w \in \cal{W}}\cal{A}_{q}w$. Note that $\cal{F}_{\cal{A}_{r,s}}$ and $\cal{F}_{\cal{A}_{q}}$ are both rings under the products $*_{r,s}$ and $*_{q}$, respectively. Set $\cal{U}_{\cal{A}{r,s}} = \cal{U} \cap \cal{F}_{\cal{A}_{r,s}}$ and $\cal{U}^{q}_{\cal{A}_{q}} = \cal{U}^{q} \cap \cal{F}_{\cal{A}_{q}}$. 

Let $\psi_{0}\colon \cal{A}_{r,s} \to \cal{A}_{q}$ be the ring homomorphism defined by $\psi_{0}(r) = q$ and $\psi_{0}(s) = q^{-1}$. If we then define $\psi\colon \cal{F}_{\cal{A}_{r,s}} \to \cal{F}_{\cal{A}_{q}}$ by $\psi(\sum c_{w}w) = \sum \psi_{0}(c_{w})w$, then it is immediate from the definitions of $*_{r,s}$ and $*_{q}$ that $\psi$ is a ring homomorphism. Moreover, if $u = \sum_{w \in \cal{W}}\gamma(w)w \in \cal{U}_{\cal{A}_{r,s}}$, then we know from Proposition~\ref{prop:image_description} that the coefficients $\gamma(w)$ satisfy the linear equations~\eqref{eq:image_description}. If we apply $\psi_{0}$ to the equations in~\eqref{eq:image_description}, then we precisely recover the linear equations in~\cite[(12)]{L}, and therefore~\cite[Theorem 5]{L} implies that $\psi(u) = \sum_{w \in \cal{W}}\psi_{0}(\gamma(w))w \in \cal{U}^{q}_{\cal{A}_{q}}$. Now, for any $\ell \in \cal{L}^{+}$, consider the homogeneous element $m_{\ell}$ of the basis $\{m_{w} \,|\, w \in \cal{W}^{+}\}$ constructed in Proposition~\ref{prop:homogeneous_basis}. By the defining properties of $\{m_{w} \,|\, w \in \cal{W}^{+}\}$, the only element of $\cal{W}^{+}$ that occurs as a summand of $m_{\ell}$ is $\ell$, and therefore $\ell = \max(m_{\ell})$. Since $\ell$ also occurs with coefficient $1$, we have $\max(\psi(m_{\ell})) = \ell$, which means that $\ell \in \cal{GL}$. This proves that $\cal{L}^{+} \subseteq \cal{GL}$, and since both $\cal{L}^{+}$ and $\cal{GL}$ are in bijection with the finite set $\Phi^{+}$, the inclusion must actually be an equality.
\end{proof}

Combining Proposition~\ref{prop:leclerc_algorithm} with the convexity result~\cite[Proposition~26]{L}, we obtain (cf.~\cite[Corollary~27]{L}):

\begin{cor}\label{cor:smallest_word} 
For $\beta \in \Phi^{+}$, the dominant Lyndon word $\ell(\beta)$ is the smallest dominant word of weight $\beta$.
\end{cor}

The following result is analogous to~\cite[Lemma~4.18]{CHW} 
(we choose to present the proof below in order to highlight the only essential calculation):

\begin{lemma}\label{lem:first_letter}
Let $\ell = [i_{1} \dots i_{d}] \in \cal{L}^{+}$. Then each word appearing in the expansion of~$R_{\ell}$ starts with $i_1$.
\end{lemma}

\begin{proof} 
We proceed by induction on $d$, with $d = 1$ being trivial. Let $\ell = \ell_{1}\ell_{2}$ be the costandard factorization. 
Then $R_{\ell} = R_{\ell_{1}} \circledast R_{\ell_{2}}$, and by the induction assumption every word appearing in the expansion 
of $R_{\ell_{1}}$ starts with $i_1$. Now, by~\cite[Lemma~14]{L}, we may write $\ell_{2} = \ell_{1}^{k}wi$, where $k \ge 0$, 
$w$ is a (possibly empty) left factor of $\ell_{1}$, and $i$ is a letter such that $wi > \ell_{1}$. If $k > 0$ or $w$ is nonempty, 
then $\ell_{2}$ starts with $i_{1}$, and thus every word appearing in the expansion of $R_{\ell_{2}}$ also starts with $i_{1}$. 
In this case, the definition of the shuffle product implies that every word in $R_{\ell_{1}} \circledast R_{\ell_{2}}$ starts 
with $i_1$. In the remaining case $\ell_{2} = i$, we have $R_{\ell_{2}} = i$, and for any word $\ell' = [i_{1} j_{2} \dots  j_{d}]$ 
in the expansion of $R_{\ell_{1}}$, we obtain:
\begin{equation*}
  \ell' \circledast i 
  = \ell' *_{r,s} i - (\omega_{i}',\omega_{|\ell'|}) i *_{r,s}\ell' 
  = i\ell' - (\omega_{i}',\omega_{|\ell'|})(\omega_{i}',\omega_{|\ell'|})^{-1} i\ell' + \sum c_{w}w 
  = \sum c_{w}w,
\end{equation*}
where in the last sum the words $w$ start with $i_{1}$ and $c_{w}\in \BC(r,s)$. This completes the proof. 
\end{proof}

Finally, the following important lemma is completely analogous to~\cite[Lemma~4.19]{CHW}:

\begin{lemma}\label{lem:root_vector_max} 
For $\ell \in \cal{L}^{+}$, we have $\max(R_{\ell}) = \ell$.
\end{lemma}

Combining this with Lemma~\ref{lem:shuffle_lw}, we get the following result:

\begin{cor}\label{cor:max_Rw} 
For $w \in \cal{W}^{+}$, we have $\max(R_{w}) = w$.
\end{cor}

\begin{proof} 
By Proposition~\ref{prop:dominant-factorization}, we can write $w = \ell_{1}\ldots \ell_{k}$, where $\ell_{1} \ge \ell_{2} \ge \ldots \ge \ell_{k}$, and each $\ell_{i} \in \cal{L}^{+}$. Then by definition, we have $R_{w} = R_{\ell_{1}} * \ldots * R_{\ell_{k}}$. We proceed by induction on $k$. If $k = 1$, the statement reduces to Lemma~\ref{lem:root_vector_max}. Now suppose $k > 1$, and the result holds for smaller $k$. Let $d$ be the largest integer such that $\ell_{1} = \ell_{2} = \ldots = \ell_{d}$. By induction, $\max(R_{\ell_{1}} * \ldots * R_{\ell_{d-1}}) = \ell_{1}\ldots \ell_{d-1} = \ell_{1}^{d - 1}$, and we know from Lemma~\ref{lem:root_vector_max} that $\max(R_{\ell_{d}}) = \ell_{d} = \ell_{1}$. Thus, if $u$ is any word in $R_{\ell_{1}^{d - 1}}$ and $v$ is any word in $R_{\ell_{1}}$, then every word appearing in $u * v$ is less than or equal to the corresponding shuffle in $\ell_{1}^{d- 1} * \ell_{1}$. Hence, $\max(R_{\ell_{1}} * \ldots * R_{\ell_{d}}) = \max(\ell_{1}^{d-1} * \ell_{1}) = \ell_{1}^{d}$, by Lemma~\ref{lem:shuffle_lw}.

If $d = k$, then we are done. If $d < k$, then by the induction assumption, we have $\max(R_{\ell_{d+1}} * \ldots * R_{\ell_{k}}) = \ell_{d+1}\ldots \ell_{k}$, and by the choice of $d$, $\ell_{1} > \ell_{d+1}$. We shall see in the next paragraph that $\ell_{1} > \ell_{d + 1}\ldots \ell_{k}$. If this inequality holds, then by using Lemma~\ref{lem:shuffle_lw} and arguments similar to those above, we conclude that $\max(R_{\ell_{1}} * \ldots * R_{\ell_{k}}) = \max(\ell_{1}^{d} * (\ell_{d+1}\ldots \ell_{k})) = \ell_{1}\ldots \ell_{k} = w$, which completes the proof.

To prove the above inequality, suppose instead that we have $\ell_{1} \le \ell_{d+1}\ldots \ell_{k}$. Then we can write $\ell_{1} = \ell_{d+1}\ldots \ell_{t}v$, where $d \le t < k$ (if $t = d$, then $\ell_{1} = v$) and $v$ is a (possibly empty) word such that $v \le \ell_{t + 1}$. But then since $\ell_{1}$ is Lyndon and $v$ is a right factor of $\ell_{1}$, we have $\ell_{1} \le v \le \ell_{t + 1}$, which yields a contradiction with $\ell_1>\ell_{t+1}$. 
\end{proof}


\subsection{Orthogonal PBW Bases}
\

The following result is an analogue of~\cite[Lemma~5.6]{CHW} (and the proof is very similar):

\begin{lemma}\label{lem:R_coprod} 
For any $\ell \in \cal{L}^{+}$, we have 
\[
  \Delta(R_{\ell}) = 
  \sum_{\ell_{1},\ell_{2} \in \cal{W}^{+}} \vartheta_{\ell_{1},\ell_{2}}^{\ell} R_{\ell_{2}} \otimes R_{\ell_{1}},
\] 
where $\vartheta_{\ell_{1},\ell_{2}}^{\ell} = 0$ unless $|\ell_{1}| + |\ell_{2}| = |\ell|$, $\ell_{1} \le \ell$, 
and $\ell \le \ell_{2}$ whenever $\ell_{2} \neq \emptyset$. 
\end{lemma}

\begin{proof} 
The condition $|\ell_{1}| + |\ell_{2}| = |\ell|$ holds because $\Delta$ preserves the grading. 
To prove that $\ell_{1} \le \ell$, note that Lemma~\ref{lem:root_vector_max} implies that we can write 
\[
  R_{\ell} = \sum_{w \le \ell}\phi_{\ell w}w
\]
for some $\phi_{\ell w} \in \BC(r,s)$. Then
\[
  \Delta(R_{\ell}) = \sum_{w_{1},w_{2},\, w_{1}w_{2} = w \le \ell}\phi_{\ell w}w_{2} \otimes w_{1},
\]
and since $w_{1} \le w \le \ell$, Lemma~\ref{lem:root_vector_max} implies that $\vartheta_{\ell_{1},\ell_{2}}^{\ell} = 0$ 
unless $\ell_{1} \le \ell$. To prove that $\ell \le \ell_{2}$ whenever $\ell_{2} \neq \emptyset$, we proceed by induction 
on the length of $\ell$. If $\ell$ is a letter, the claim is obvious. For the induction step, let $\ell = wv$ be 
the costandard factorization of $\ell$, so that $R_{\ell} = R_{w} \circledast R_{v} = R_{w} *_{r,s} R_{v} - R_{w} *_{s,r} R_{v}$. 
Since $\{R_{w} \,|\, w \in \cal{W}^{+}\}$ is a basis for $\cal{U}$, we can write 
$\Delta(R_{w} *_{r,s} R_{v}) = \sum_{\bh,\bk \in \cal{W}^{+}} z_{\bh,\bk} R_{\bh} \otimes R_{\bk}$ for some 
$z_{\bh,\bk} \in \BC(r,s)$. Moreover, interchanging $r$ and $s$ in Proposition~\ref{prop:shuffle_coproduct} shows that 
we have $\Delta(x *_{s,r} y) = \Delta(x) *_{s,r} \Delta(y)$, where
\[
  (w \otimes x) *_{s,r} (y \otimes z) = (\omega_{|y|}',\omega_{|x|})(w *_{s,r} y) \otimes (x *_{s,r} z).
\]
Thus, $\Delta(R_{w} *_{s,r} R_{v}) = \sum_{\bh,\bk \in \cal{W}^{+}} \bar{z}_{\bh,\bk} R_{\bh} \otimes R_{\bk}$, since 
$\bar{r} = s$ and $\bar{s} = r$. This implies that 
\[
  \Delta(R_{\ell}) = \Delta(R_{w} *_{r,s} R_{v} - R_{w} *_{s,r} R_{v}) = 
  \sum_{\bh,\bk \in \cal{W}^{+}} (z_{\bh,\bk} - \bar{z}_{\bh,\bk}) R_{\bh} \otimes R_{\bk}.
\]
On the other hand, we can also compute $\Delta(R_{\ell})$ using the formula 
$R_{\ell} = R_{w} *_{r,s} R_{v} - (\omega_{|v|}',\omega_{|w|})R_{v} *_{r,s} R_{w}$. 
By the induction hypothesis, we have 
\[
  \Delta(R_{w} *_{r,s} R_{v}) = 
  \sum \vartheta_{w_{1},w_{2}}^{w} \vartheta_{v_{1},v_{2}}^{v} 
       (R_{w_{2}} \otimes R_{w_{1}}) *_{r,s} (R_{v_{2}} \otimes R_{v_{1}}),
\]
where the sum is over all $w_{1},w_{2},v_{1},v_{2} \in \cal{W}^+$ satisfying $w_{1} \le w \le w_{2}$ unless 
$w_{2} = \emptyset$ and $v_{1} \le v \le v_{2}$ unless $v_{2} = \emptyset$. Now, by Proposition~\ref{prop:triangularity}, 
the transition matrix from the basis $\{R_{w} \,|\, w \in \cal{W}^{+}\}$ to the basis 
$\{\epsilon_{w} \,|\, w \in \cal{W}^{+}\}$ is triangular, and therefore we can rewrite the expression above as 
\[
  \Delta(R_{w} *_{r,s} R_{v}) = 
  \sum_{\substack{\bh \ge w_{2}v_{2} \\ \bk \ge w_{1}v_{1}}} \Theta_{\bh,\bk} R_{\bh} \otimes R_{\bk}
\]
for some $\Theta_{\bh,\bk} \in \BC(r,s)$. Similar arguments show that 
\[
  \Delta(R_{v} *_{r,s} R_{w}) = 
  \sum_{\substack{\bh \ge v_{2}w_{2} \\ \bk \ge v_{1}w_{1}}} \Theta_{\bh,\bk}' R_{\bh} \otimes R_{\bk},
\]
where $\Theta_{\bh,\bk}' \in \BC(r,s)$.

Since $z_{\bh,\bk} = \Theta_{\bh,\bk}$ by definition, it follows that $\Theta_{\bh,\bk} = 0$ if and only if 
$\Theta_{\bh,\bk}' = 0$. Indeed, if $\Theta_{\bh,\bk} = 0$, then $z_{\bh,\bk} = \ol{z}_{\bh,\bk} = 0$, so 
the coefficient of $R_{\bh} \otimes R_{\bk}$ in the expansion of $\Delta(R_{\ell})$ is zero. This implies that 
$0 = \Theta_{\bh,\bk} - (\omega_{|v|}',\omega_{|w|}) \Theta_{\bh,\bk}' = - (\omega_{|v|}',\omega_{|w|})\Theta_{\bh,\bk}'$, 
which forces $\Theta_{\bh,\bk}' = 0$. Conversely, if $\Theta_{\bh,\bk}' = 0$, then 
$\ol{z}_{\bh,\bk} = 0$, so we also have $\Theta_{\bh,\bk} = z_{\bh,\bk} = 0$.

Now, suppose that $z_{\bh,\bk} - \ol{z}_{\bh,\bk} \neq 0$, so there are dominant words $w_{2},v_{2}$, such that $\bh \ge w_{2}v_{2}$. 
If $w_{2} \neq \emptyset$, then unless $v_{2} = \emptyset$, we get $\bh \ge w_{2}v_{2} \ge wv \ge \ell$, as desired. 
If $v_{2} = \emptyset$ but $w_{2} > w$, then since $w_{2}$ has length at most that of $w$, we still have $\bh \ge w_{2} > wv = \ell$. 
Now suppose that $v_{2} = \emptyset$ and $w_{2} = w$. Then the term $R_{\bh} \otimes R_{\bk}$ must have come from 
$(R_{w} \otimes 1) *_{r,s} (1 \otimes R_{v})$, so that $\bh = w$ and $\bk = v$. But since $v$ is a right factor of 
the Lyndon word $\ell$, we have $\ell < v$, which contradicts the fact that $\vartheta_{\ell_{1},\ell_{2}}^{\ell} = 0$ 
unless $\ell_{1} \le \ell$. If $w_{2} = \emptyset$, then in the case that $v_{2} \neq \emptyset$, we have 
$\bh \ge v_{2} \ge v > \ell$ as $v$ is a right factor of $\ell$, while in the case that $v_{2} = \emptyset$, 
we have $\bh = \emptyset$. This completes the proof of the lemma.
\end{proof}

We shall also need the analogue of the above lemma for $\bar{R}_{l}$, with $\, \bar{} \,$ 
introduced in Proposition~\ref{prop:tau_and_bar}:

\begin{lemma}\label{lem:R_bar_coprod} 
For any $\ell \in \cal{L}^{+}$, we have 
\[
  \Delta(\bar{R}_{\ell}) = 
  \sum_{\ell_{1},\ell_{2} \in \cal{W}^{+}} 
    \hat{\vartheta}_{\ell_{1},\ell_{2}}^{\ell} \bar{R}_{\ell_{1}} \otimes \bar{R}_{\ell_{2}},
\] 
where $\hat{\vartheta}_{\ell_{1},\ell_{2}}^{\ell} = 0$ unless $|\ell_{1}| + |\ell_{2}| = |\ell|$, 
$\ell_{1} \le \ell$, and $\ell \le \ell_{2}$ whenever $\ell_{2} \neq \emptyset$.
\end{lemma}

\begin{proof} 
This follows from Lemma~\ref{lem:R_coprod} once we show that if $\Delta(x) = \sum_{i} y_{i} \otimes z_{i}$, for some $x,y_{i},z_{i} \in \cal{U}$,  
then $\Delta(\bar{x}) = \sum_{i} \varrho_{i} \bar{z}_{i} \otimes \bar{y}_{i}$ for some $\varrho_{i} \in \BC(r,s)$. 
First, for $x = [i_{1}\ldots i_{d}] \in \cal{W}$, we have 
\begin{align*}
  \Delta(\ol{x}) 
  &= \left( \prod_{1 \le l < m \le d} (\omega_{i_{m}}',\omega_{i_{l}})^{-1} \right) 
     \sum_{0 \le k \le d} [i_{k}\ldots i_{1}] \otimes [i_{d}\ldots i_{k+1}] \\
  &= \sum_{0 \le k \le d} \left( \prod_{1 \le l \le k < m \le d} (\omega_{i_{m}}',\omega_{i_{l}})^{-1} \right) 
     \ol{[i_{1}\ldots i_{k}]} \otimes \ol{[i_{k+1}\ldots i_{d}]}.
\end{align*}
Because the above coefficient depends only on $\alpha_{i_1}+\ldots+\alpha_{i_k}$ and $\alpha_{i_{k+1}}+\ldots+\alpha_{i_d}$, the general case follows.
\end{proof}

For any $w \in \cal{W}^{+}$, consider its canonical factorization into dominant Lyndon words  
(Proposition~\ref{prop:dominant-factorization})
\begin{equation}\label{eq:canonical_dominant_fact}
  w = w_{1}w_{2} \dots w_{d}, \quad w_{1} \ge w_{2} \ge \dots \ge w_{d}, \quad w_{1}, w_{2}, \dots, w_{d} \in \cal{L}^+.
\end{equation}  
Then, we define
\begin{equation}\label{eq:R-basis}
\begin{split}
  & \wtd{R}_{w} = R_{w_{d}} * R_{w_{d-1}} * \dots * R_{w_{1}} ,\\
  & \bar{R}_{w} = \bar{R}_{w_{1}} * \bar{R}_{w_{2}} * \dots * \bar{R}_{w_{d}}.
\end{split}
\end{equation}

The following is the key result of this Section (see~\eqref{eq:rs_gamma} for the notation $r_{|\ell|},s_{|\ell|}$):

\begin{theorem}\label{thm:dual_bases} 
Let $\ell,w \in \cal{W}^{+}$. Then $(\wtd{R}_{\ell},\bar{R}_{w}) = 0$ unless $\ell = w$. Moreover, if 
$\ell = \ell_{1}^{m_{1}}\ell_{2}^{m_{2}} \dots \ell_{h}^{m_{h}}$ with $\ell_{1} > \ell_{2} > \dots > \ell_{h}$ 
is the canonical factorization of $\ell$ into dominant Lyndon words, then we have 
\begin{equation}\label{eq:Rw-pairing}
  (\wtd{R}_{\ell},\bar{R}_{\ell}) = 
  \prod_{i = 1}^{h} \left( [m_{i}]_{r_{|\ell_{i}|},s_{|\ell_{i}|}}! \, r_{|\ell_{i}|}^{-\frac{1}{2}m_{i}(m_{i}-1)} 
    (R_{\ell_{i}},\bar{R}_{\ell_{i}})^{m_{i}} \right).
\end{equation}
\end{theorem}

The proof of this theorem is similar to that of~\cite[Theorem~5.7]{CHW}. However, we prefer to present full 
details because our PBW bases ($\{\wtd{R}_{\ell}\}_{\ell\in \cal{W}^+}$ and $\{\bar{R}_{\ell}\}_{\ell\in \cal{W}^+}$) 
are different from those used in~\cite{CHW} (see a single basis $\{\mathsf{E_i}\}_{\mathsf{i}\in \mathsf{W}^+}$ 
of \emph{loc.cit.}).

\begin{proof} 
We may assume that $|\ell| = |w|$, because otherwise the claim follows from the basic properties of $(\cdot,\cdot)$. 
We then proceed by induction on the length of $\ell$ (which is equal to the length of $w$ since $|\ell| = |w|$). 
The case when $\ell$ and $w$ have length $1$ is trivial. Suppose first that $\ell \in \cal{L}^{+}$, and 
$w \in \cal{W}^{+}$ with $|w| = |\ell|$ but $w \neq \ell$. If $w = w_{1} \dots w_{d}$ is the canonical factorization 
of~\eqref{eq:canonical_dominant_fact}, then $d > 1$, for otherwise $w$ would be a dominant Lyndon word of the same 
degree as $\ell\in \cal{L}^{+}$. Using Lemmas~\ref{lem:symm_pairing_properties},~\ref{lem:R_coprod} and 
Proposition~\ref{prop:shuffle_coproduct}, we get:
\begin{align*}
  (\wtd{R}_{\ell},\bar{R}_{w}) 
  &= (R_{\ell},\bar{R}_{w_{1}} * \dots * \bar{R}_{w_{d}}) 
   = (\Delta(R_{\ell}),\bar{R}_{w_{d}} \otimes (\bar{R}_{w_{1}} * \dots * \bar{R}_{w_{d-1}})) \\
  &= \sum_{\ell_{1},\ell_{2} \in \cal{W}^{+}} \vartheta_{\ell_{1},\ell_{2}}^{\ell} 
     (R_{\ell_{2}},\bar{R}_{w_{d}}) (R_{\ell_{1}},\bar{R}_{w_{1}} * \dots * \bar{R}_{w_{d-1}}).
\end{align*}
Since $d > 1$, we must have $|w_{d}| \neq |\ell|$, so the only nonzero terms in the sum above satisfy 
$\ell_{1} < \ell < \ell_{2}$, $|w_{d}| = |\ell_{2}|$, and $|\ell_{1}| = |w_{1}\ldots w_{d-1}|$. As the 
length of $w_{d}$ is smaller than the length of $w$, we may apply the induction assumption to conclude 
that all nonzero terms in the above sum satisfy $\ell_{2} = w_{d}$ and $\ell_{1} = w_{1}\ldots w_{d-1}$. 
Then 
\[
  w_{1} \le w_{1}\ldots w_{d-1} = \ell_{1} < \ell_{2} = w_{d} \le w_{1},
\]
which never holds. This finally implies that $(\wtd{R}_{\ell},\bar{R}_{w}) = 0$ for $\ell \in \cal{L}^{+}$ 
and $w \neq \ell$.

Now suppose that $w \in \cal{L}^{+}$, and $\ell \in \cal{W}^{+}$ satisfies $|\ell| = |w|$ but $\ell \neq w$. Let 
$\ell = \ell_{1} \dots \ell_{d}$ be the canonical factorization of $\ell$, where $\ell_{1} \ge \dots \ge \ell_{d}$ 
are dominant Lyndon words, cf.~\eqref{eq:canonical_dominant_fact}. Using Lemmas~\ref{lem:symm_pairing_properties} 
and~\ref{lem:R_coprod}, we likewise obtain:
\[
  (\wtd{R}_{\ell},\bar{R}_{w}) = (R_{\ell_{d}} * \dots * R_{\ell_{1}},\bar{R}_{w}) 
  \, = \sum_{w_{1},w_{2} \in \cal{W}^{+}} \hat{\vartheta}_{w_{1},w_{2}}^{w}
  (R_{\ell_{d-1}} * \dots * R_{\ell_{1}},\bar{R}_{w_{1}})(R_{\ell_{d}},\bar{R}_{w_{2}}).
\]
Repeating the above argument, we conclude that the only nonzero terms in the sum above satisfy 
\[
  \ell_{1} \le \ell_{1} \dots \ell_{d-1} = w_{1} < w_{2} = \ell_{d} \le \ell_{1},
\]
which is never possible. This again implies that $(\wtd{R}_{\ell},\bar{R}_{w}) = 0$ in the present setup.

Now, for arbitrary $\ell, w \in \cal{W}^{+}_{\nu}$, we proceed by induction on the height of $\nu$. The base case 
$\nu \in \Pi$ is trivial, so suppose that $\nu \notin \Pi$, and for any $\mu \in Q^{+}$ with $\hgt(\mu) < \hgt(\nu)$ 
and any $u,v \in \cal{W}^{+}_{\mu}$, we have $(\wtd{R}_{u},\bar{R}_{v}) = 0$ unless $u = v$. 
Let $\ell = \ell_{1}\ell_{2} \dots \ell_{d}$ and $w = w_{1}w_{2} \dots w_{e}$ be the canonical factorizations of 
$\ell$ and $w$ into products of non-increasing dominant Lyndon words. We may assume that $d,e > 1$, for otherwise 
$\ell \in \cal{L}^{+}$ or $w \in \cal{L}^{+}$, which are the two cases already treated above.

We shall first consider the case when $\ell_{1} \le w_{1}$. Then, we have: 
\begin{equation}\label{eq:PBW_pairing_expansion}
\begin{split}
  (\wtd{R}_{\ell},\bar{R}_{w}) 
  &= (R_{\ell_{d}} * \dots * R_{\ell_{1}},\bar{R}_{w_{1}} * \dots * \bar{R}_{w_{e}}) 
   = (\Delta(R_{\ell_{d}}) * \dots * \Delta(R_{\ell_{1}}), (\bar{R}_{w_{2}} * \dots * \bar{R}_{w_{e}}) \otimes \bar{R}_{w_{1}}) \\
  &= \sum \vartheta_{\ell_{1,1},\ell_{1,2},\ldots,\ell_{d,1},\ell_{d,2}}
     (R_{\ell_{d,2}} * \dots * R_{\ell_{1,2}}, \bar{R}_{w_{2}} * \dots * \bar{R}_{w_{e}})
     (R_{\ell_{d,1}} * \dots * R_{\ell_{1,1}}, \bar{R}_{w_{1}}),
\end{split}
\end{equation}
where every term in the sum above satisfies the following three properties (for all $1 \le t \le d$): 
\begin{equation*}
 \ell_{t,1},\ell_{t,2} \in \cal{W}^{+} ,\qquad \ell_{t,1} \le \ell_{t} ,\qquad 
 \ell_{t} \le \ell_{t,2} \quad  \mathrm{unless} \quad \ell_{t,2} = \emptyset,    
\end{equation*}
and $\vartheta_{\ell_{1,1},\ldots,\ell_{d,2}}\in \BC(r,s)$. 
For a particular term in~\eqref{eq:PBW_pairing_expansion}, we shall first show that 
$(R_{\ell_{d,1}} * \dots * R_{\ell_{1,1}},\bar{R}_{w_{1}}) = 0$ unless there is a unique $k$ such that 
$\ell_{k,1} = w_{1}$ and $\ell_{t,1} = \emptyset$ for all $t \neq k$, cf.~\cite[Claim(**)]{CHW}.

To this end, let $k$ be the maximal integer such that $\ell_{k,1} \neq \emptyset$. Then 
\begin{align*}
  (R_{\ell_{d,1}} * \dots * R_{\ell_{1,1}},\bar{R}_{w_{1}}) 
  &= (R_{\ell_{k,1}} * \dots * R_{\ell_{1,1}},\bar{R}_{w_{1}}) \\ 
  &= \sum_{w_{1,1},w_{1,2} \in \cal{W}^{+}} \hat{\vartheta}_{w_{1,1},w_{1,2}}^{w_{1}} 
     (R_{\ell_{k-1,1}} * \dots * R_{\ell_{1,1}},\bar{R}_{w_{1,1}}) (R_{\ell_{k,1}},\bar{R}_{w_{1,2}}).
\end{align*}
Suppose that a particular term
\[
  \hat{\vartheta}_{w_{1,1},w_{1,2}}^{w_{1}}
  (R_{\ell_{k-1,1}} * \dots * R_{\ell_{1,1}}, \bar{R}_{w_{1,1}}) (R_{\ell_{k,1}},\bar{R}_{w_{1,2}})
\]
is nonzero. Then $|\ell_{k,1}| \le |\ell_{k}| < |\ell|$ and $|w_{1,2}| \le |w_{1}| < |w|$, so by induction on the height, 
we have $(R_{\ell_{k,1}},\bar{R}_{w_{1,2}}) = 0$ unless $\ell_{k,1} = w_{1,2}$. In this case, $w_{1,2} \neq \emptyset$, 
so that $w_{1,2} \ge w_{1}$. Thus, we obtain: 
\[
  w_{1,2} = \ell_{k,1} \le \ell_{k} \le \ell_{1} \le w_{1} \le w_{1,2},
\]  
which shows that $\ell_{k,1} = w_{1,2} = w_{1}$, and so $w_{1,1} = \emptyset$. This also implies that 
$\ell_{t,1} = \emptyset$ for all $1 \le t < k$.

Thus, if $(R_{\ell_{d,1}} * \dots * R_{\ell_{1,1}},\bar{R}_{w_{1}}) \neq 0$, then there is a unique $k$ such that 
$\ell_{k,1} = w_{1}$ and $\ell_{t,1} = \emptyset$ for $t \neq k$, and since 
$w_{1} = \ell_{k,1} \le \ell_{k} \le \ell_{1} \le w_{1}$, we also have $\ell_{k,1} = \ell_{k} = \ell_{1} = w_{1}$. 
This means that for the corresponding words $\ell_{t,2}$, we have $\ell_{t,2} = \ell_{t}$ if $t \neq k$, and 
$\ell_{k,2} = \emptyset$.

Now, let $m_{1}$ be the largest integer such that $\ell_{1} = \ell_{2} = \dots = \ell_{m_{1}}$. Then combining 
what we proved above and using the obvious equalities 
$\vartheta_{\ell,\emptyset}^{\ell} = \vartheta_{\emptyset,\ell}^{\ell} = 1$,~\eqref{eq:PBW_pairing_expansion} 
reduces to 
\begin{align*}
  (\wtd{R}_{\ell},\bar{R}_{w}) 
  &= (R_{\ell_{d}} * \dots * R_{\ell_{1}},\bar{R}_{w_{1}} * \dots * \bar{R}_{w_{e}}) \\
  &= \sum_{k = 1}^{m_{1}} \left( (R_{\ell_{d}} \otimes 1) * \dots * (R_{\ell_{k+1}} \otimes 1) * 
     (1 \otimes R_{\ell_{1}}) * (R_{\ell_{1}}^{*(k-1)} \otimes 1), 
     (\bar{R}_{w_{2}} * \dots * \bar{R}_{w_{e}}) \otimes \bar{R}_{\ell_{1}} \right) \\
  &= \left( \sum_{k = 1}^{m_{1}}(\omega_{|\ell_{1}|}',\omega_{|\ell_{1}|}^{k-1})^{-1} \right)
     \left( R_{\ell_{d}} * \dots * R_{\ell_{2}}, \bar{R}_{w_{2}} * \dots * \bar{R}_{w_{e}} \right) 
     (R_{\ell_{1}},\bar{R}_{\ell_{1}}), 
\end{align*}
where we used~\eqref{eq:twisted_shuffle_square} in the last equality. So by induction hypothesis, $(\wtd{R}_{\ell},\bar{R}_{w}) = 0$ 
unless $\ell = w$, and if $\ell = w$, then using $(\omega_{|\ell_{1}|}',\omega_{|\ell_{1}|})^{-1} = r_{|\ell_{1}|}^{-1}s_{|\ell_{1}|}$ 
(cf.~\eqref{eq:rs_gamma}) and the induction assumption, we obtain 
\begin{align*}
  &(\wtd{R}_{\ell},\bar{R}_{\ell}) \\
  &\ \ = \frac{[m_{1}]_{r_{|\ell_{1}|},s_{|\ell_{1}|}}}{r_{|\ell_{1}|}^{m_{1} - 1}}  
     \prod_{i = 2}^{h} \left([m_{i}]_{r_{|\ell_{i}|},s_{|\ell_{i}|}}! r_{|\ell_{i}|}^{-\frac{1}{2}m_{i}(m_{i} - 1)} 
            (R_{\ell_{i}},\bar{R}_{\ell_{i}})^{m_{i}} \right)  
    [m_{1} - 1]_{r_{|\ell_{1}|},s_{|\ell_{1}|}}! r_{|\ell_{1}|}^{-\frac{1}{2}(m_{1}-1)(m_{1}-2)} 
    (R_{\ell_{1}},\bar{R}_{\ell_{1}})^{m_{1}} \\
  &\ \ = \prod_{i = 1}^{h} \left( [m_{i}]_{r_{|\ell_{i}|},s_{|\ell_{i}|}}! r_{|\ell_{i}|}^{-\frac{1}{2}m_{i}(m_{i}-1)} 
     (R_{\ell_{i}},\bar{R}_{\ell_{i}})^{m_{i}} \right),
\end{align*}
where $\ell = \ell_{1}^{m_{1}}\ell_{2}^{m_{2}} \dots \ell_{h}^{m_{h}}$ is the canonical factorization of $\ell$ into 
dominant Lyndon words with $\ell_{1} > \dots > \ell_{h}$.

We now consider the case when $\ell_{1} > w_{1}$. Then 
\begin{align*}
  (\wtd{R}_{\ell},\bar{R}_{w}) 
  &= (R_{\ell_{d}} * \dots * R_{\ell_{1}}, \bar{R}_{w_{1}} * \dots * \bar{R}_{w_{e}}) 
   = (R_{\ell_{1}} \otimes (R_{\ell_{d}} * \dots * R_{\ell_{2}}), \Delta(\bar{R}_{w_{1}}) * \dots * \Delta(\bar{R}_{w_{e}})) \\
  &= \sum \hat{\vartheta}_{w_{1,1},w_{1,2},\ldots ,w_{e,1},w_{e,2}} 
     (R_{\ell_{1}}, \bar{R}_{w_{1,1}} * \dots * \bar{R}_{w_{e,1}}) 
     (R_{\ell_{d}} * \dots * R_{\ell_{2}},\bar{R}_{w_{1,2}} * \dots * \bar{R}_{w_{e,2}}),
\end{align*}
and as in~\eqref{eq:PBW_pairing_expansion} every term in the sum above satisfies (for all $1 \le t \le e$):
\[
  w_{t,1},w_{t,2} \in \cal{W}^{+}, \qquad w_{t,1} \le w_{t}, \qquad  
  w_{t} \le w_{t,2} \quad \mathrm{unless} \quad w_{t,2} = \emptyset.
\]  
Suppose that $(R_{\ell_{1}},\bar{R}_{w_{1,1}} * \dots * \bar{R}_{w_{e,1}}) \neq 0$ for some term in the above sum. 
Since $\ell \neq \emptyset$, $\ell_{1} \neq \emptyset$, we have $w_{t,1} \neq \emptyset$ for at least one integer $t$; 
let $k$ be the maximal such integer. Then, we have:  
\begin{align*}
  (R_{\ell_{1}},\bar{R}_{w_{1,1}} * \dots * \bar{R}_{w_{e,1}}) 
  &= (R_{\ell_{1}},\bar{R}_{w_{1,1}} * \dots * \bar{R}_{w_{k,1}}) \\ 
  &= \sum_{\ell_{1,1},\ell_{1,2} \in \cal{W}^{+}} \vartheta_{\ell_{1,1},\ell_{1,2}}^{\ell_{1}} 
     (R_{\ell_{1,2}}, \bar{R}_{w_{k,1}})(R_{\ell_{1,1}},\bar{R}_{w_{1,1}} * \dots * \bar{R}_{w_{k-1,1}}).
\end{align*}
Suppose that a particular term 
\[
  \vartheta_{\ell_{1,1},\ell_{1,2}}^{\ell_{1}} 
  (R_{\ell_{1,2}}, \bar{R}_{w_{k,1}})(R_{\ell_{1,1}},\bar{R}_{w_{1,1}} * \dots * \bar{R}_{w_{k-1,1}})
\]
is nonzero. We have $|w_{k,1}| \le |w_{k}| < |w|$ and $|\ell_{1,2}| \le |\ell_{1}| < |\ell|$, so by induction on 
the height, it follows that $\ell_{1,2} = w_{k,1}$ for this term. In particular, this shows that $\ell_{1,2} \neq 0$, 
so that $\ell_{1,2} \ge \ell_{1}$. But then we obtain
\[
  w_{1} < \ell_{1} \le \ell_{1,2} = w_{k,1} \le w_{k} \le w_{1},
\]
a contradiction. Hence, $(\wtd{R}_{\ell},\bar{R}_{w}) = 0$ when $\ell_{1} > w_{1}$, 
which exhausts all cases and proves the theorem.
\end{proof}

In the next Section, we shall explicitly evaluate the constants $\{(R_{\ell},\bar{R}_{\ell})\}_{\ell\in \cal{L}^+}$ 
for a specific ordering on $I$. Combining this with the theorem above, we shall derive our main Theorems~\ref{thm:main-1} 
and~\ref{thm:main-2} in the last Section.


\section{Root Vectors and their Pairings}\label{sec:root_vectors}

In this Section, we explicitly compute $R_\ell$ and the pairing $(R_{\ell},\bar{R}_{\ell})$ for each dominant 
Lyndon word $\ell\in \cal{L}^+$, with the ordering $1 < \dots < n$ on $I=\{1,\dots,n\}$. Similarly to~\cite[\S6]{CHW}, 
we treat classical types case~by~case. We shall use the notation for positive roots that was introduced 
in~\eqref{eq:root_notation_A}--\eqref{eq:root_notation_D}, and we also follow~\eqref{eq:parity-convention}.


\subsection{Type $A_{n}$}
\

For the order $1 < 2 < \dots < n$, the dominant Lyndon words in type $A_n$ are given by (cf.~\cite[\S8.1]{L}):

\begin{lemma}\label{lem:Lyndon_A} 
The set of dominant Lyndon words is 
\begin{equation*}
  \cal{L}^{+} = \big\{ [i\dots j] \,\big|\, 1 \le i \le j \le n \big\}.
\end{equation*}
\end{lemma}

We shall now explicitly evaluate the corresponding elements $R_\ell$:

\begin{prop}\label{prop:R_A} 
For $\ell = [i\ldots j]$ with $1 \le i \le j \le n$, we have 
\begin{equation*}
  R_{\ell} = (r-s)^{j-i}[i\ldots j].
\end{equation*}
\end{prop}

\begin{proof} 
We proceed by induction on $j - i$, with the case $j - i = 0$ being obvious. For $j - i > 0$, the costandard factorization 
$\ell = \ell_{1}\ell_{2}$ is explicitly given by $\ell_{1} = [i \dots (j-1)]$ and $\ell_{2} = [j]$. Thus, by the induction 
assumption, we have:
\begin{align*}
  R_{\ell} 
  &= R_{\ell_{1}} \circledast R_{\ell_{2}} = (r-s)^{j - 1 - i} [i\ldots j-1] \circledast [j] \\
  &= (r - s)^{j - 1 - i} 
     \left( [i\ldots j-1] *[j] - (\omega_{j}',\omega_{\gamma_{i,j-1}}) [j] * [i\ldots j-1] \right) \\
  &= (r-s)^{j - 1 - i} \left( (\omega_{\gamma_{i,j-1}}',\omega_{j})^{-1} - (\omega_{j}',\omega_{\gamma_{i,j-1}}) \right) 
     [i\ldots j] + (r - s)^{j - 1 -i} ([i\ldots (j-2)] \circledast [j])[j-1],
\end{align*}
which yields the result because $(\omega_{\gamma_{i,j-1}}',\omega_{j})^{-1} = r$, $(\omega_{j}',\omega_{\gamma_{i,j-1}}) = s$,
and $[i \dots (j-2)] \circledast [j] = 0$.
\end{proof}

Finally, let us derive the formula for the pairing of the above elements:

\begin{cor}\label{cor:pairing_A} 
For $\ell = [i\ldots j]$ with $1 \le i \le j \le n$, we have 
\begin{equation*}
  (R_{\ell},\bar{R}_{\ell}) = (r - s)^{j - i}.
\end{equation*}
\end{cor}

\begin{proof} 
We proceed by induction on $j - i$, with the case $j = i$ being clear. Evoking the costandard factorization of $\ell$ 
(with $\ell_{1} = [i \dots (j-1)]$, $\ell_{2} = [j]$) and Lemma~\ref{lem:symm_pairing_properties}, we obtain:
\begin{align*}
  (R_{\ell},\bar{R}_{\ell}) 
  &= (r - s)^{j-i} ([i \dots j], \bar{R}_{\ell_{1}} * [j] - r [j] * \bar{R}_{\ell_{1}}) 
   = (r - s)^{j-i} (\Delta([i \dots j]) , [j] \otimes \bar{R}_{\ell_{1}} - r \bar{R}_{\ell_{1}} \otimes [j]) \\
  &= (r - s)^{j-i} ([j]\otimes [i\dots (j-1)], [j]\otimes \bar{R}_{\ell_{1}} ) 
   = (r - s)(R_{\ell_{1}},\bar{R}_{\ell_{1}}) = (r - s)^{j - i},
\end{align*}
where the last two equalities follow from $R_{\ell_1}=(r-s)^{j-i-1}[i\dots (j-1)]$ and the induction hypothesis. 
\end{proof}


\subsection{Type $B_{n}$}
\

For the order $1 < 2 < \dots < n$, the dominant Lyndon words in type $B_n$ are given by (cf.~\cite[\S6.2]{CHW}):

\begin{lemma}\label{lem:Lyndon_B} 
The set of dominant Lyndon words is 
\[
  \cal{L}^{+} = 
  \big\{ [i\ldots j] \,\big|\, 1 \le i \le j \le n \big\} \cup 
  \big\{ [i\ldots n\, n\ldots j] \,\big|\, 1 \le i < j  \le n \big\}.
\]
\end{lemma}

We shall now explicitly evaluate the corresponding elements $R_\ell$:

\begin{prop}\label{prop:R_B} 
(1) For $\ell = [i\ldots j]$ with $1 \le i \le j \le n$, we have 
\begin{equation*}
  R_{\ell} = (r^{2} - s^{2})^{j - i}[i\ldots j].
\end{equation*}

\noindent
(2) For $\ell = [i\ldots n\, n\ldots j]$ with $1 \le i < j \le n$, we have 
\begin{equation*}
  R_{\ell} = (rs)^{2(j - n)} (r^{2} - s^{2})^{2n - i - j + 1} [i\ldots n\, n\ldots j].
\end{equation*}
\end{prop}

\begin{proof} 
Because $(\omega_{j}',\omega_{\gamma_{i,j-1}}) = s^{2}$ and $(\omega_{\gamma_{i,j-1}}',\omega_{j})^{-1} = r^{2}$ 
for all $1 \le i < j \le n$, the proof of part (1) is exactly the same as the proof of Proposition~\ref{prop:R_A}. 
For (2), we proceed by induction on $n - j$. If $n - j = 0$, then $\ell = [i\ldots n\, n]$, and its costandard 
factorization is $\ell = \ell_{1}\ell_{2}$ where $\ell_{1} = [i\ldots n]$ and $\ell_{2} = [n]$. Thus:
\begin{align*}
  R_{\ell} 
  &= R_{\ell_{1}} \circledast R_{\ell_{2}} 
   = (r^{2} - s^{2})^{n-i} \left( [i\ldots n] * [n] - (\omega_{n}',\omega_{\gamma_{in}}) [n] * [i\ldots n]\right ) \\
  &= (r^{2} - s^{2})^{n-i} \left( (\omega_{\gamma_{in}}',\omega_{n})^{-1} [i\ldots n\, n] + ([i\ldots (n-1)] * [n])[n] \right) \\
  &\quad - (r^{2} - s^{2})^{n-i} \left( (\omega_{n}',\omega_{\gamma_{in}}) [i\ldots n\, n] 
    + (\omega_{n}',\omega_{\gamma_{in}}) (\omega_{n}',\omega_{n})^{-1} ([n] * [i\ldots (n-1)])[n]\right ) \\
  &= (r^{2} - s^{2})^{n-i} \left( ([i\ldots (n-1)] * [n])[n] - s^{2}([n] * [i\ldots (n-1)])[n]\right ) \\
  &= (r^{2} - s^{2}) R_{\ell_{1}}[n] = (r^{2} - s^{2})^{n + 1 - i} [i\ldots n\, n],
\end{align*}
where we use the equalities $(\omega_{\gamma_{in}}',\omega_{n})^{-1} = (\omega_{n}',\omega_{\gamma_{in}}) = rs$ 
and $(\omega_{n}',\omega_{n})^{-1} = r^{-1}s$ in the fourth line.

Now suppose that $n - j > 0$ and the result holds for all larger $j$. Then, the costandard factorization of 
$\ell = [i\ldots n\, n\ldots j]$ is $\ell = \ell_{1}\ell_{2}$ with $\ell_{1} = [i\ldots n\, n \ldots (j+1)]$ 
and $\ell_{2} = [j]$. Thus, by the induction assumption:
\begin{align*}
  R_{\ell} 
  &= R_{\ell_{1}} \circledast R_{\ell_{2}} \\
  &= (rs)^{2(j + 1 - n)} (r^{2} - s^{2})^{2n - i - j} 
     ([i\ldots n\, n \ldots (j+1)] * [j] - (\omega_{j}',\omega_{\beta_{i,j+1}})[j] * [i\ldots n\, n \ldots (j + 1)]) \\
  &= (rs)^{2(j + 1 - n)} (r^{2} - s^{2})^{2n - i - j}
     \left( (\omega_{\beta_{i,j+1}}',\omega_{j})^{-1} [i\ldots n\, n \dots j] + ([i\ldots n\, n\ldots (j + 2)] * [j])[j+1] \right) \\ 
  & \ \ - (rs)^{2(j + 1 - n)} (r^{2} - s^{2})^{2n - i - j}
     \left( (\omega_{j}',\omega_{\beta_{i,j+1}}) [i\ldots n\, n\ldots j] - 
           (\omega_{j}',\omega_{\beta_{i,j+2}}) ([j] * [i\ldots n\, n\ldots (j+2)])[j + 1] \right) \\
  &= (rs)^{2(j - n)} (r^{2} - s^{2})^{2n - i -j + 1} [i\ldots n\, n\ldots j] \\
  &\quad + (rs)^{2(j + 1 - n)} (r^{2} - s^{2})^{2n - i - j}
   \left( [i\ldots n\, n\ldots (j + 2)] * [j] - [j] * [i\ldots n\, n\ldots (j + 2)] \right)[j + 1],
\end{align*}
where we used the equalities $(\omega_{\beta_{i,j+1}}',\omega_{j})^{-1} = s^{-2}$, $(\omega_{j}',\omega_{\beta_{i,j+1}}) = r^{-2}$, 
and $(\omega_{j}',\omega_{\beta_{i,j+2}}) = 1$ in the last line. Thus, it suffices to prove that 
\[
  [i\ldots n\, n\ldots (j + 2)] * [j] - [j] * [i\ldots n\, n \ldots (j+2)] = 0
\]
(for $j = n -1$, the above equality should be rather interpreted as $[i\ldots n] * [n-1] - [n-1] * [i\ldots n]=0$). 
Since $(\omega'_{\beta_{ik}},\omega_{j})^{-1} = 1$, $(\omega_{\gamma_{ik}}',\omega_{j})^{-1} = 1$, and 
$(\omega_{j}',\omega_{k})^{-1} = 1$ for $j + 2 \le k \le n$, we have 
\begin{multline*}
  [i\ldots n\, n\ldots (j+2)] * [j] - [j] *[i\ldots n\, n \ldots (j + 2)] = \\ 
  \left( [i\ldots (j + 1)] * [j] - [j] * [i\ldots (j + 1)] \right) [(j + 2) \ldots n\, n \ldots (j + 2)]
\end{multline*}
and thus it remains to verify $[i\ldots (j + 1)] *[j] - [j] *[i\ldots (j + 1)] = 0$ for $i < j \le n-1$. 
To this end, we have:
\begin{align*}
  & [i\ldots (j + 1)] * [j] - [j] * [i\ldots (j + 1)] \\
  &= (\omega_{\gamma_{i,j + 1}}',\omega_{j})^{-1} [i\ldots (j + 1) j] - [i\ldots (j+1) j] 
     + ([i\ldots j] * [j])[j+1] - (\omega_{j}',\omega_{j+1})^{-1} ([j] * [i\ldots j])[j + 1] \\ 
  &= ([i\ldots j] * [j] - r^{2} [j] * [i\ldots j])[j+1] \\
  &= ((\omega_{\gamma_{ij}}',\omega_{j})^{-1} [i\ldots j\, j] + ([i\ldots (j-1)] * [j])[j] 
      - r^{2} [i\ldots j\, j] - s^{2}([j] * [i\ldots (j-1)])[j])[j+1] \\
  &= (s^{2} - r^{2})[i\ldots j\, j (j + 1)] + \frac{1}{(r^{2} - s^{2})^{j - 1 - i}}R_{[i\ldots j]} [j(j+1)] \\
  &= (s^{2} - r^{2})[i\ldots j\, j (j + 1)] + (r^{2} - s^{2}) [i\ldots j\, j (j+1)] = 0.
\end{align*}
This completes the proof of (2).
\end{proof}

Finally, let us derive the formula for the pairing of the above elements:

\begin{cor}\label{cor:pairing_B} 
(1) For $\ell = [i\ldots j]$ with $1 \le i \le j \le n$, we have 
\begin{equation*}
  (R_{\ell},\bar{R}_{\ell}) = (r^{2} - s^{2})^{j - i}.
\end{equation*}

\noindent
(2) For $\ell = [i\ldots n\, n\ldots j]$ with $1 \le i < j \le n$, we have
\begin{equation*}
  (R_{\ell},\bar{R}_{\ell}) = (rs)^{2(j - n)} (r^{2} - s^{2})^{2n - i - j + 1}.
\end{equation*}
\end{cor}

\begin{proof} 
The proof of (1) is the same as the proof of Corollary~\ref{cor:pairing_A}. For (2), we proceed by induction on $n - j$. 
If $n - j = 0$, then $\ell = [i\ldots n\, n]$ and its costandard factorization is $\ell = \ell_{1}\ell_{2}$ with 
$\ell_{1} = [i\ldots n]$ and $\ell_{2} = [n]$. Therefore, we get:
\begin{align*}
  (R_{\ell},\bar{R}_{\ell}) 
  &= (r^{2} - s^{2})^{n+1-i} ([i\ldots n\, n], \bar{R}_{\ell_{1}} * [n] - rs [n] * \bar{R}_{\ell_{1}}) \\
  &= (r^{2} - s^{2})^{n+1-i} (\Delta([i\ldots n\, n]), [n] \otimes \bar{R}_{\ell_{1}} - rs \bar{R}_{\ell_{1}} \otimes [n]) \\
  &= (r^{2} - s^{2})^{n+1-i} ([n] \otimes [i\ldots n], [n] \otimes \bar{R}_{\ell_{1}}) \\
  &= (r^{2} - s^{2})([n],[n])(R_{\ell_{1}},\bar{R}_{\ell_{1}}) = (r^{2} - s^{2})^{n+1-i},
\end{align*}
where the last equality is a consequence of part (1).

If $n - j> 0$, then the costandard factorization is $\ell = \ell_{1}\ell_{2}$ with $\ell_{1} = [i\ldots n\, n \ldots (j+1)]$, 
$\ell_{2} = [j]$. Thus: 
\begin{align*}
  (R_{\ell},\bar{R}_{\ell}) 
  &= (rs)^{2(j - n)} (r^{2} - s^{2})^{2n - i - j + 1} 
     ([i\ldots n\, n\ldots j], \bar{R}_{\ell_{1}} *[j] - s^{-2} [j] * \bar{R}_{\ell_{1}}) \\
  &= (rs)^{2(j - n)} (r^{2} - s^{2})^{2n - i - j + 1} ([j] \otimes [i\ldots n\, n\ldots (j + 1)], [j] \otimes \bar{R}_{\ell_{1}}) \\
  &= (rs)^{-2} (r^{2} - s^{2}) ([j],[j]) (R_{\ell_{1}},\bar{R}_{\ell_{1}}) \\
  &= (rs)^{2(j - n)}(r^{2} - s^{2})^{2n - i - j + 1},
\end{align*}
where the last equality follows from the induction hypothesis.
\end{proof}


\subsection{Type $C_{n}$}
\

For the order $1 < 2 < \dots < n$, the dominant Lyndon words in type $C_n$ are given by (cf.~\cite[\S6.3]{CHW}):

\begin{lemma}\label{lem:Lyndon_C} 
The set of dominant Lyndon words is 
\[
  \cal{L}^{+} = 
  \big\{ [i\ldots j] \,\big|\, 1 \le i \le j \le n \big\} \cup 
  \big\{ [i\ldots n\ldots j] \,\big|\, 1 \le i < j < n \big\} \cup 
  \big\{ [i\ldots (n-1)\, i\ldots n] \,\big|\, 1 \le i < n \big\}.
\]
\end{lemma}

We shall now explicitly evaluate the corresponding elements $R_\ell$:

\begin{prop}\label{prop:R_C} 
(1) For $\ell = [i\ldots j]$ with $1 \le i \le j < n$ and $\ell = [n]$, we have 
\begin{equation*}
  R_{\ell} = (r - s)^{j - i} [i\ldots j].
\end{equation*}

\noindent
(2) For $\ell = [i\ldots n]$ with $1 \le i < n$, we have 
\begin{equation*}
  R_{\ell} = (r - s)^{n - 1 - i} (r^{2} - s^{2}) [i\ldots n].
\end{equation*}

\noindent 
(3) For $\ell = [i\ldots n \ldots j]$ with $1 \le i < j < n$, we have 
\begin{equation*}
  R_{\ell} = (rs)^{j-n} (r - s)^{2n - i - j - 1} (r^{2} - s^{2}) [i\ldots n \ldots j].
\end{equation*}

\noindent
(4) For $\ell = [i\ldots (n-1)\, i\ldots n]$ with $1 \le i < n$, we have 
\begin{equation*}
  R_{\ell} = (r - s)^{2n - 2i - 1} (r^{2} - s^{2}) r\, ([i\ldots (n-1)] * [i\ldots (n-1)])[n].
\end{equation*}
\end{prop}

\begin{proof} 
The proofs of parts (1) and (2) are similar to the proof of Proposition~\ref{prop:R_A}, while the proof of 
part (3) is similar to the proof of part (2) of Proposition~\ref{prop:R_B}. For part (4), we first note that 
$[i\ldots (n-1)] \in \cal{U}$ by part (1), and therefore $[i\ldots (n-1)] * [i\ldots (n-1)] \in \cal{U}$ as well. Moreover, 
it follows from Proposition~\ref{prop:image_description} that also $f = ([i\ldots (n-1)] * [i\ldots (n-1)])[n] \in \cal{U}$. 
Now, by Lemma~\ref{lem:shuffle_lw}, we have $\max(f) = \ell$. Then, since 
$\max(R_{w}) = w$ for all $w \in \cal{W}^{+}$ by Corollary~\ref{cor:max_Rw}, we may write 
\[
  f = \sum_{w \le \ell,\ w \in \cal{W}^{+}} \vartheta_{w} R_{w}
\]
for some $\vartheta_w\in \BC(r,s)$ with $\vartheta_\ell\ne 0$.
But $\ell$ is the smallest dominant word of its degree by Corollary~\ref{cor:smallest_word}, so we must have 
  $$R_{\ell} = \vartheta_{\ell}^{-1}f = \vartheta_{\ell}^{-1} ([i\ldots (n-1)] * [i\ldots (n-1)])[n].$$

Using this, we can now compute $R_{\ell}$. Since the costandard factorization is $\ell = \ell_{1}\ell_{2}$ 
with $\ell_{1} = [i\ldots (n-1)]$ and $\ell_{2} = [i\ldots n]$, we have: 
\begin{align*}
  R_{\ell} 
  &= R_{\ell_{1}} \circledast R_{\ell_{2}} 
   = (r - s)^{2n - 2i - 2} (r^{2} - s^{2}) \left( [i\ldots (n-1)] * [i\ldots n] - rs [i\ldots n] * [i\ldots (n-1)] \right) \\
  &= (r - s)^{2n - 2i - 2} (r^{2} - s^{2}) \left( ([i\ldots (n-2)] *[i\ldots n])[n-1] - r([i\ldots n] *[i\ldots (n-2)])[n-1] \right) \\
  &\quad + (r-s)^{2n - 2i - 2} (r^{2} - s^{2}) (r^{2} - rs) ([i\ldots (n-1)] * [i\ldots (n-1)])[n].
\end{align*}
Since we know that $R_{\ell}$ is a multiple of $([i\ldots (n-1)] * [i\ldots (n-1)])[n]$, it follows that 
\[
  ([i\ldots (n-2)] * [i\ldots n])[n-1] - r([i\ldots n] * [i\ldots (n-2)])[n-1] = 0,
\]  
and therefore 
\[
  R_{\ell} = (r-s)^{2n - 2i - 1} (r^{2} - s^{2}) r\, ([i\ldots (n-1)] * [i\ldots (n-1)])[n],
\]
which completes the proof.
\end{proof}

Finally, let us derive the formula for the pairing of the above elements:

\begin{cor}\label{cor:pairing_C} 
(1) For $\ell = [i\ldots j]$ with $1 \le i \le j < n$ and $\ell = [n]$, we have 
\begin{equation*}
  (R_{\ell},\bar{R}_{\ell}) = (r - s)^{j - i}.
\end{equation*}

\noindent
(2) For $\ell = [i\ldots n]$, we have 
\begin{equation*}
  (R_{\ell},\bar{R}_{\ell}) = (r - s)^{n - 1 - i} (r^{2} - s^{2}).
\end{equation*}

\noindent
(3) For $\ell = [i\ldots n\ldots j]$ with $1 \le i < j < n$, we have 
\begin{equation*}
  (R_{\ell},\bar{R}_{\ell}) = (r - s)^{2n - i - j - 1} (r^{2} - s^{2}) (rs)^{j - n}.
\end{equation*}

\noindent
(4) For $\ell = [i\ldots (n-1)\, i\ldots n]$ with $1 \le i < n$, we have 
\begin{equation*}
  (R_{\ell},\bar{R}_{\ell}) = (r - s)^{2n - 2i - 1}(r^{2} - s^{2})(r + s).
\end{equation*}
\end{cor}

\begin{proof} 
The proofs of parts (1)--(3) are similar to the arguments given in the preceding two Subsections, 
so we shall only present the details for part (4). First, we note that (cf.~\eqref{eq:dot-multiplication}) 
\begin{multline*}
  \Delta(([i\ldots (n-1)] *[i\ldots (n-1)])[n]) \\
  = (\Delta([i\ldots (n-1)]) * \Delta([i\ldots (n-1)])) \cdot ([n] \otimes 1) 
    + 1 \otimes ([i\ldots (n-1)] * [i\ldots (n-1)])[n].
\end{multline*}
Then, since the costandard factorization is $\ell = \ell_{1}\ell_{2}$ with $\ell_{1} = [i\ldots (n-1)]$ 
and $\ell_{2} = [i\ldots n]$, we have 
\begin{align*}
  (R_{\ell},\bar{R}_{\ell}) 
  &= (r-s)^{2n - 2i - 1} (r^{2} - s^{2}) r 
     \left( \Delta(([i\ldots (n-1)] * [i\ldots (n-1)])[n]), 
            \bar{R}_{\ell_{2}} \otimes \bar{R}_{\ell_{1}} - rs \bar{R}_{\ell_{1}} \otimes \bar{R}_{\ell_{2}} \right) \\
  &= (r-s)^{2n - 2i - 1} (r^{2} - s^{2}) r 
     \left( (\Delta([i\ldots (n-1)]) * \Delta([i\ldots (n-1)]))\cdot([n] \otimes 1), 
            \bar{R}_{\ell_{2}} \otimes \bar{R}_{\ell_{1}} \right) \\
  &= (r - s)^{2n - 2i - 1} (r^{2} - s^{2}) r 
     \left( (1 + r^{-1}s) ([i\ldots n] \otimes [i\ldots (n-1)]), \bar{R}_{\ell_{2}} \otimes \bar{R}_{\ell_{1}} \right) \\
  &= (r - s)^{2n - 2i - 1} (r^{2} - s^{2}) (r + s) \cdot 
     \frac{1}{(r - s)^{2n - 2i - 2} (r^{2} - s^{2})} (R_{\ell_{2}},\bar{R}_{\ell_{2}}) (R_{\ell_{1}},\bar{R}_{\ell_{1}}) \\
  &= (r - s)^{2n - 2i - 1} (r^{2} - s^{2}) (r + s),
\end{align*}
where the last equality follows from parts (1)--(2).
\end{proof}


\subsection{Type $D_{n}$}
\

For the order $1 < 2 < \dots < n$, the dominant Lyndon words in type $D_n$ are given by (cf.~\cite[\S6.4]{CHW}):

\begin{lemma}\label{lem:Lyndon_D} 
The set of dominant Lyndon words is 
\begin{equation*}
\begin{split}
  \cal{L}^{+} = 
  & \big\{ [i\ldots j] \,\big|\, 1 \le i \le j < n \big\} \cup 
    \big\{ [i\ldots (n-2)\, n] \,\big|\, 1 \le i \le n-2 \big\} \cup \\
  & \big\{ [i\ldots (n-2)\, n\, (n-1) \ldots j] \,\big|\, 1 \le i < j \le n-1 \big\}.
\end{split}
\end{equation*}
\end{lemma}

We shall now explicitly evaluate the corresponding elements $R_\ell$:

\begin{prop}\label{prop:R_D} 
(1) For $\ell = [i\ldots j]$ with $1 \le i \le j < n$, we have 
\begin{equation*}
  R_{\ell} = (r - s)^{j - i} [i\ldots j].
\end{equation*}

\noindent
(2) For $\ell = [i\ldots (n-2)\, n]$ with $1 \le i \le n-2$, we have 
\begin{equation*}
  R_{\ell} = (r - s)^{n - 1 - i} [i\ldots (n - 2)\, n].
\end{equation*}

\noindent
(3) For $\ell = [i\ldots (n-2)\, n\, (n-1) \ldots j]$ with $1 \le i < j \le n-1$, we have 
\begin{equation*}
    R_{\ell} = (rs)^{j + 1 - n} (r - s)^{2n - i - j - 1} 
    \Big( [i\ldots (n-1)\, n\, (n-2) \ldots j] + (rs)^{-1} [i\ldots (n-2)\, n\, (n-1)\ldots j] \Big).
\end{equation*}
\end{prop}

\begin{proof} 
The computations for parts (1)--(2) are completely analogous to those in the proof of Proposition~\ref{prop:R_A}, 
so we shall only provide the details for part (3). We proceed by induction on $n - j$. If $n - j = 1$, then the 
costandard factorization is $\ell = \ell_{1}\ell_{2}$ with $\ell_{1} = [i\ldots (n-2)\, n]$ and $\ell_{2} = [n-1]$. 
Therefore, we obtain: 
\begin{align*}
  R_{\ell} 
  &= (r - s)^{n - 1 - i} 
   \left( [i\ldots (n-2)\, n] * [n-1] - (\omega_{n-1}',\omega_{\beta_{in}}) [n-1] * [i\ldots (n-2)\, n] \right) \\
  &= (r - s)^{n - 1 - i} 
   \left( (\omega_{\beta_{in}}',\omega_{n-1})^{-1} - (\omega_{n-1}',\omega_{\beta_{in}}) \right) [i\ldots (n-2)\, n\, (n-1)] \\
  &\quad + (r - s)^{n - 1 - i}
   \left( ([i\ldots (n-2)] * [n-1])[n] - r^{-1}(\omega_{n-1}',\omega_{n})^{-1} ([n- 1] * [i\ldots (n-2)])[n] \right) \\
  &= (r - s)^{n - 1 - i} (s^{-1} - r^{-1}) [i\ldots (n-2)\, n\, (n-1)] \\
  &\quad + (r- s)^{n-1-i} \left( [i\ldots (n-2)] * [n-1] - s[n-1] * [i\ldots (n-2)] \right)[n] \\
  &= (r - s)^{n-i} (rs)^{-1} [i\ldots (n-2)\, n\, (n-1)] + (r - s) R_{[i\ldots (n-1)]}[n] \\
  &= (r - s)^{n-i} \big( [i\ldots (n-1)\, n] + (rs)^{-1} [i\ldots (n-2)\, n\, (n-1)] \big).
\end{align*}

Now, if $n - j > 1$, the costandard factorization is $\ell = \ell_{1}\ell_{2}$ with 
$\ell_{1} = [i\ldots (n-2)\, n\, \ldots (j+1)]$ and $\ell_{2} = [j]$. Thus, by induction hypothesis, we have:
\begin{align*}
  R_{\ell} 
  &= (rs)^{j + 2 - n} (r - s)^{2n - i - j - 2} [i\ldots (n-2)(n-1)\, n\, (n-2)\ldots (j + 1)] \circledast [j]  \\
  &\quad + (rs)^{j + 1 - n} (r - s)^{2n - i - j - 2} [i\ldots (n-2)\, n\, (n-1)(n-2)\ldots (j + 1)] \circledast [j] . 
\end{align*}
Note first that
\begin{align*}
  & [i\ldots (n-1)\, n\, (n-2)\ldots (j + 1)] \circledast [j] \\
  &= [i\ldots (n-1)\, n\, (n-2)\ldots (j + 1)] * [j] - r^{-1}[j] * [i\ldots (n-1)\, n\, (n-2)\, \ldots (j + 1)] \\
  &= (s^{-1} - r^{-1})[i\ldots (n-1)\, n\, (n-2)\ldots j] + ([i\ldots (n-1)\, n\, (n-2)\ldots (j + 2)] \circledast [j])[j+1].
\end{align*}
Moreover, as $(\omega_{\beta_{ik}}',\omega_{j})^{-1} = 1$, $(\omega_{k}',\omega_{j})^{-1} = 1$, 
and $(\omega_{j}',\omega_{k})^{-1} = 1$ for all $k \ge j + 2$, we have: 
\[
  [i\ldots (n-1)\, n\, (n-2)\ldots (j + 2)] \circledast [j] = 
  ([i\ldots (j + 1)] *[j] - [j] * [i\ldots (j + 1)])[(j + 2)\ldots (n-1)\, n\, (n-2)\ldots (j + 2)].
\]
Finally, computations analogous to those in the proof of Proposition~\ref{prop:R_B}(2) show that 
\[ 
  [i\ldots (j + 1)] *[j] - [j] * [i\ldots (j + 1)] = 0 \qquad \mathrm{for\ all} \quad i < j \le n-2,
\]  
so that 
\[
  [i\ldots (n-1)\, n\, (n-2)\ldots (j + 1)] \circledast [j] = (s^{-1} - r^{-1}) [i\ldots (n-1)\, n\, (n-2)\ldots j].
\]
Similarly,
\begin{align*}
  & [i\ldots (n-2)\, n\, (n-1)\ldots (j + 1)] \circledast [j] \\
  &= (s^{-1} - r^{-1}) [i\ldots (n-2)\, n\, (n-1)\ldots j] + ([i\ldots (n-2)\, n\, (n-1)\ldots (j + 2)] \circledast [j])[j+1],
\end{align*}
and as above, we have
\begin{multline*}
  [i\ldots (n-2)\, n\, (n-1)\ldots (j + 2)] \circledast [j] = \\
  ([i\ldots (j + 1)] *[j] - [j] * [i\ldots (j + 1)]) [(j + 2)\ldots (n-2)\, n\, (n-1)\ldots (j + 2)] = 0.
\end{multline*}
Thus, we obtain:
\begin{align*}
  R_{\ell} 
  &= (rs)^{j + 2 - n} (r - s)^{2n - i - j - 2} (s^{-1} - r^{-1}) [i\ldots (n-2)(n-1)\, n\, (n-2)\ldots j] \\
  &\quad + (rs)^{j + 1 - n} (r - s)^{2n - i - j - 2} (s^{-1} - r^{-1}) [i\ldots (n-2)\, n\, (n-1)(n-2)\ldots j] \\ 
  &= (rs)^{j + 1 - n} (r - s)^{2n - i - j - 1} 
   \big( [i\ldots (n-1)\, n\, (n-2)\ldots j] + (rs)^{-1} [i\ldots (n-2)\, n\, (n-1)\ldots j] \big),
\end{align*}
which completes the proof.
\end{proof}

Finally, let us derive the formula for the pairing of the above elements:

\begin{cor}\label{cor:pairing_D} 
(1) For $\ell = [i\ldots j]$ with $1 \le i \le j < n$, we have 
\begin{equation*}
  (R_{\ell},\bar{R}_{\ell}) = (r - s)^{j - i}.
\end{equation*}

\noindent
(2) For $\ell = [i\ldots (n-2)\, n]$ with $1 \le i \le n- 2$, we have 
\begin{equation*}
  (R_{\ell},\bar{R}_{\ell}) = (r - s)^{n - 1 - i}.
\end{equation*}

\noindent
(3) For $\ell = [i\ldots (n-2)\, n\, (n-1)\ldots j]$ with $1 \le i < j \le n-1$, we have 
\begin{equation*}
  (R_{\ell},\bar{R}_{\ell}) = (r - s)^{2n - i - j -1} (rs)^{j - n}.
\end{equation*}
\end{cor}

\begin{proof} 
The proofs of parts (1)--(2) are similar to the previous computations, so we shall omit the details. For part (3), 
we proceed by induction on $n - j$. If $n - j = 1$, then the costandard factorization is $\ell = \ell_{1}\ell_{2}$ 
with $\ell_{1} = [i\ldots (n-2)\, n]$ and $\ell_{2} = [n-1]$, so that
\begin{align*}
  (R_{\ell},\bar{R}_{\ell}) 
  &= (\Delta(R_{\ell}),[n-1] \otimes \bar{R}_{\ell_{1}} - s^{-1} \bar{R}_{\ell_{1}} \otimes [n-1]) \\
  &= (r - s)^{n-i} (\Delta([i\ldots (n-1)\, n] + (rs)^{-1} [i\ldots (n-2)\, n\, (n-1)]), 
     [n-1] \otimes \bar{R}_{\ell_{1}} - s^{-1} \bar{R}_{\ell_{1}} \otimes [n-1]) \\
  &= (r - s)^{n - i} (rs)^{-1} ([n-1] \otimes [i\ldots (n-2)\, n],[n-1] \otimes \bar{R}_{\ell_{1}}) \\
  &= (r - s) (rs)^{-1} (R_{\ell_{1}},\bar{R}_{\ell_{1}}) = (r-s)^{n-i} (rs)^{-1}.
\end{align*}
If $n - j > 1$, then the costandard factorization is $\ell = \ell_{1}\ell_{2}$ with 
$\ell_{1} = [i\ldots (n-2)\, n\, \ldots (j + 1)]$ and $\ell_{2} = [j]$, so that we have 
\begin{align*}
  (R_{\ell},\bar{R}_{\ell}) 
  &= (\Delta(R_{\ell}),[j] \otimes \bar{R}_{\ell_{1}} - s^{-1} \bar{R}_{\ell_{1}} \otimes [j]) \\
  &= (rs)^{j + 1 - n} (r - s)^{2n - i - j - 1} 
     (\Delta([i\ldots (n-1)\, n\, (n-2)\ldots j]),[j] \otimes \bar{R}_{\ell_{1}} - s^{-1} \bar{R}_{\ell_{1}} \otimes [j]) \\
  &\quad + (rs)^{j + 1 - n} (r - s)^{2n - i - j - 1}
   ((rs)^{-1} \Delta([i\ldots (n-2)\, n\, (n-1)\ldots j]),[j] \otimes \bar{R}_{\ell_{1}} - s^{-1} \bar{R}_{\ell_{1}} \otimes [j]) \\
  &=(rs)^{j + 1 - n} (r - s)^{2n - i - j - 1} ([i\ldots (n-1)\, n\, (n-2)\ldots (j+1)],\bar{R}_{\ell_{1}}) \\
  &\quad + (rs)^{j + 1 - n} (r - s)^{2n - i - j - 1}
   ((rs)^{-1}[i\ldots (n-2)\, n\, (n-1) \ldots (j+1)],\bar{R}_{\ell_{1}}) \\
  &= (rs)^{-1} (r - s) (R_{\ell_{1}},\bar{R}_{\ell_{1}}) 
   = (rs)^{j - n} (r - s)^{2n - i - j - 1},
\end{align*}  
where the last equality follows from the induction hypothesis.
\end{proof}


\section{Orthogonal PBW Bases for $U_{r,s}(\fg)$}\label{sec:main}

In this Section, we transfer (using Theorem~\ref{thm:pairing_comparison}) the orthogonal bases constructed 
in Section~\ref{sec:orthogonal_bases} to $U_{r,s}(\fg)$, which, along with the computations from 
Section~\ref{sec:root_vectors}, proves our main Theorems~\ref{thm:main-1} and~\ref{thm:main-2}.

To state our results, we first need to introduce the corresponding notion of \textbf{quantum root vectors} of $U_{r,s}(\fg)$. 
To this end, we recall the notation $\ell \colon \Phi^{+} \to \cal{L}^{+}$ for the inverse of the bijection from Theorem~\ref{thm:LR-bijection}. 
Then, for $\gamma \in \Phi^{+}$, we define 
\[
  e_{\gamma} = \ol{\Psi^{-1}\left( \bar{R}_{\ell(\gamma)} \right)}=\Psi^{-1}\left( R_{\ell(\gamma)} \right) \qquad \text{and} \qquad 
  f_{\gamma} = \varphi \left( \ol{\Psi^{-1}\left( R_{\ell(\gamma)} \right)} \right),
\]
where the $\BC$-algebra bar-involution $\, \bar{} \,$ and the $\BC(r,s)$-algebra anti-automorphism $\varphi$ were 
introduced in Proposition~\ref{prop:extra_structures}. Explicitly, if $\alpha,\beta \in \Phi^{+}$ are such that 
$\ell(\gamma) = \ell(\alpha)\ell(\beta)$ is the costandard factorization of $\ell(\gamma)$,~then: 
\begin{equation}\label{eq:q-root-vectors}
  e_{\gamma} = e_{\alpha}e_{\beta} - (\omega_{\beta}',\omega_{\alpha}) e_{\beta}e_{\alpha} \qquad \text{and} \qquad 
  f_{\gamma} = f_{\beta}f_{\alpha} - (\omega_{\alpha}',\omega_{\beta})^{-1} f_{\alpha}f_{\beta}.
\end{equation}

We also note that the lexicographical ordering on $\cal{L}^{+}$ induces, via the bijection $\ell$, 
a total ordering on $\Phi^{+}$:
\begin{equation}\label{eq:lyndon_order}
  \alpha < \beta \quad \Longleftrightarrow \quad \ell(\alpha) < \ell(\beta)  \ \ \mathrm{ lexicographically}.
\end{equation}
For the ordering $1 < 2 < \dots < n$, one can easily verify 
(using Lemmas~\ref{lem:Lyndon_A},~\ref{lem:Lyndon_B},~\ref{lem:Lyndon_C},~\ref{lem:Lyndon_D}) that 
the corresponding orderings on $\Phi^{+}$ are as follows:
\begin{itemize}

\item
\emph{Type $A_{n}$}
\begin{equation*}
\begin{split}
  \alpha_{1} < \alpha_{1} + \alpha_{2} < \dots < \alpha_{1} + \dots + \alpha_{n} < \alpha_{2} <
  \dots < \alpha_{n-1} < \alpha_{n-1} + \alpha_{n} < \alpha_{n}.
\end{split}
\end{equation*}

\item
\emph{Type $B_{n}$}
\begin{equation*}
\begin{split}
  & \alpha_{1} < \alpha_{1} + \alpha_{2} < \dots < \alpha_{1} + \dots + \alpha_{n} \\
  & \qquad < \alpha_{1} + \dots + \alpha_{n-1} + 2\alpha_{n} < \dots <
    \alpha_{1} + 2\alpha_{2} + \dots + 2\alpha_{n} \\
  &\quad < \alpha_{2} < \dots < \alpha_{n-1} < \alpha_{n-1} + \alpha_{n} < \alpha_{n-1} + 2\alpha_{n} < \alpha_{n}.
\end{split}
\end{equation*}

\item
\emph{Type $C_{n}$}
\begin{equation*}
\begin{split}
  & \alpha_{1} < \alpha_{1} + \alpha_{2} < \dots < \alpha_{1} + \dots + \alpha_{n-1} <
    2\alpha_{1} + \dots + 2\alpha_{n-1} + \alpha_{n} < \alpha_{1} + \dots + \alpha_{n} \\
  & \qquad < \alpha_{1} + \dots + \alpha_{n-2} + 2\alpha_{n-1} + \alpha_{n} < \dots <
    \alpha_{1} + 2\alpha_{2} + \dots + 2\alpha_{n-1} + \alpha_{n} \\
  &\quad < \alpha_{2} < \dots < \alpha_{n-1} < 2\alpha_{n-1} + \alpha_{n} < \alpha_{n-1} + \alpha_{n} < \alpha_{n}.
\end{split}
\end{equation*}

\item
\emph{Type $D_{n}$}
\begin{equation*}
\begin{split}
  & \alpha_{1} < \alpha_{1} + \alpha_{2} < \dots < \alpha_{1} + \dots + \alpha_{n-2}+\alpha_{n-1} <
    \alpha_{1} +\dots +  \alpha_{n-2} + \alpha_{n} < \alpha_{1} + \dots + \alpha_{n} \\
  & \qquad < \alpha_{1} + \dots + \alpha_{n-3} + 2\alpha_{n-2} + \alpha_{n-1} + \alpha_{n} < \dots <
    \alpha_{1} + 2\alpha_{2} + \dots + 2\alpha_{n-2} + \alpha_{n-1} + \alpha_{n} \\
  &\quad < \alpha_{2} < \dots < \alpha_{n-2}<\alpha_{n-2}+\alpha_{n-1}<\alpha_{n-2}+\alpha_{n} < 
   \alpha_{n-2} + \alpha_{n-1} + \alpha_{n} <\alpha_{n-1}<\alpha_{n}.
\end{split}
\end{equation*}
\end{itemize}

We may now state our first main result, which corresponds to parts (a) and (b) of~\cite[Theorem 5.12]{MT}:

\begin{theorem}\label{thm:main-1}
(1) The ordered products
\begin{equation}\label{eq:ordered}
  \left\{ \overset{\longleftarrow}{\underset{\gamma \in \Phi^{+}}{\prod}} e_{\gamma}^{m_{\gamma}}\, \Big|\, m_{\gamma} \ge 0 \right\} 
  \quad \text{and}\quad 
  \left\{ \overset{\longleftarrow}{\underset{\gamma \in \Phi^{+}}{\prod}} f_{\gamma}^{m_{\gamma}}\, \Big|\, m_{\gamma} \ge 0 \right\}
\end{equation}
are bases for $U_{r,s}^{+}(\fg)$ and $U_{r,s}^{-}(\fg)$, respectively. Here and below, the arrow $\leftarrow$ 
over the product signs refers to the total order~\eqref{eq:lyndon_order} on $\Phi^+$, thus ordering the positive roots 
in decreasing order.

\medskip
\noindent
(2) The Hopf pairing~\eqref{eq:Hopf-parity} is orthogonal with respect to these bases. More explicitly, we have:
\begin{equation}\label{eq:orthogonal}
    \left( \overset{\longleftarrow}{\underset{\gamma \in \Phi^{+}}{\prod}} f_{\gamma}^{n_{\gamma}}, 
           \overset{\longleftarrow}{\underset{\gamma \in \Phi^{+}}{\prod}} e_{\gamma}^{m_{\gamma}} \right)_{H} = \,
    \prod_{\gamma\in \Phi^+} \Big( \delta_{n_\gamma,m_\gamma} [m_{\gamma}]_{r_{\gamma},s_{\gamma}}! 
        s_{\gamma}^{-\frac{1}{2}m_{\gamma}(m_{\gamma} - 1)} (f_{\gamma},e_\gamma)^{m_{\gamma}}_{H} \Big).
\end{equation}
\end{theorem}

\begin{proof} 
For part (1), we first recall from Proposition~\ref{prop:lyndon_basis} that 
\begin{equation*}
  \big\{ R_{w} \,\big|\, w \in \cal{W}^{+} \big\} = 
  \big\{ R_{\ell_{1}} * \dots * R_{\ell_{k}} \,\big|\, 
         k \in \BZ_{\geq 0},\, \ell_{1}, \dots, \ell_{k} \in \cal{L}^{+},\, \ell_{1} \ge \dots \ge \ell_{k} \big\}
\end{equation*}
is a basis for $\cal{U}$. Since $x \mapsto \bar{x}$ is a $\BC$-algebra automorphism, we find that 
\begin{equation}\label{eq:bar_R_basis}
  \big \{\bar{R}_{w} \,|\, w \in \cal{W}^{+} \big\} = 
  \big\{ \bar{R}_{\ell_{1}} * \dots * \bar{R}_{\ell_{k}} \,\big|\, 
         k \in \BZ_{\geq 0},\, \ell_{1}, \dots, \ell_{k} \in \cal{L}^{+},\, \ell_{1} \ge \dots \ge \ell_{k} \big\}.
\end{equation}
is also a basis for $\cal{U}$. Then, evoking~\eqref{eq:R-basis}, the set
\begin{equation}\label{eq:wtd_R_basis}
  \big\{ \wtd{R}_{w} \,\big|\, w \in \cal{W}^{+} \big\} = 
  \big\{ R_{\ell_{k}} * \dots * R_{\ell_{1}} \,\big|\, 
         k \in \BZ_{\geq 0},\, \ell_{1}, \dots, \ell_{k} \in \cal{L}^{+},\, \ell_{1} \ge \dots \ge \ell_{k} \big\}
\end{equation}
is yet another basis for $\cal{U}$ because, by Theorem~\ref{thm:dual_bases}, it is orthogonal to 
$\{\bar{R}_{w} \,|\, w \in \cal{W}^{+}\}$ with respect to the non-degenerate pairing $(\cdot,\cdot)$ on $\cal{U}$.

Thus, applying $\Psi^{-1}$ followed by the bar involution $x \mapsto \bar{x}$ on $U_{r,s}^{+}(\fg)$ 
to~\eqref{eq:bar_R_basis}, we obtain the basis 
\begin{multline*}
  \big\{ e_{|\ell_{1}|} \dots e_{|\ell_{k}|} \,\big|\, 
         k \in \BZ_{\geq 0},\, \ell_{1}, \dots, \ell_{k} \in \cal{L}^{+},\, \ell_{1} \ge \dots \ge \ell_{k} \big\} = \\
  \big\{ e_{\gamma_{1}}\ldots e_{\gamma_{k}} \,\big|\, 
         k \in \BZ_{\geq 0},\, \gamma_{1},\ldots ,\gamma_{k} \in \Phi^{+},\, \gamma_{1} \ge \dots \ge \gamma_{k} \big\}
\end{multline*}
for $U_{r,s}^{+}(\fg)$, where we use the bijection $\ell$ and the ordering it induces on $\Phi^{+}$ via~\eqref{eq:lyndon_order}.
Similarly, as $\varphi \colon U_{r,s}^{+} \to U_{r,s}^{-}$ is an anti-isomorphism, applying 
$\varphi \circ (x \mapsto \bar{x}) \circ \Psi^{-1}$ to~\eqref{eq:wtd_R_basis} yields a basis of $U_{r,s}^{-}(\fg)$:
\begin{multline*}
  \big\{ f_{|\ell_{1}|} \dots f_{|\ell_{k}|} \,\big|\, 
         k \in \BZ_{\geq 0},\, \ell_{1}, \dots, \ell_{k} \in \cal{L}^{+},\, \ell_{1} \ge \dots \ge \ell_{k} \big\} = \\
  \big\{ f_{\gamma_{1}}\ldots f_{\gamma_{k}} \,\big|\, 
         k \in \BZ_{\geq 0},\, \gamma_{1},\ldots ,\gamma_{k} \in \Phi^{+},\, \gamma_{1} \ge \dots \ge \gamma_{k} \big\}.
\end{multline*}
This completes the proof of part (1).

To prove part (2), we shall use Theorem~\ref{thm:pairing_comparison}. Given sequences 
$(n_{\gamma})_{\gamma \in \Phi^{+}}, (m_{\gamma})_{\gamma \in \Phi^{+}}\in (\BZ_{\geq 0})^{\Phi^+}$, consider 
dominant words $\ell = \overset{\longleftarrow}{\underset{\gamma \in \Phi^{+}}{\prod}} \ell(\gamma)^{n_{\gamma}}$ and   
$w = \overset{\longleftarrow}{\underset{\gamma \in \Phi^{+}}{\prod}} \ell(\gamma)^{m_{\gamma}}$. 
Furthermore, for any $\mu = \sum_{i = 1}^{n}c_{i} \alpha_{i} \in \Phi^{+}$, we set 
\[
  C_{\mu} = \prod_{i = 1}^{n} \frac{1}{(s_{i} - r_{i})^{c_{i}}}.
\]
We may assume that $|\ell| = |w|$, because otherwise the claim is obvious. Then, since 
$\ol{\varphi(f_{\gamma})} = \Psi^{-1}(R_{\ell(\gamma)})$ and $\bar{e}_{\gamma} = \Psi^{-1}(\bar{R}_{\ell(\gamma)})$, 
Theorem~\ref{thm:pairing_comparison} implies that
\begin{equation}\label{eq:pairing_reduction}
  \left( \overset{\longleftarrow}{\underset{\gamma \in \Phi^{+}}{\prod}} f_{\gamma}^{n_{\gamma}}, 
         \overset{\longleftarrow}{\underset{\gamma \in \Phi^{+}}{\prod}} e_{\gamma}^{m_{\gamma}} \right)_{H} = 
  C_{|\ell|} \ol{ \left( 
    \overset{\longrightarrow}{\underset{\gamma \in \Phi^{+}}{\prod}} \Psi^{-1}\left (R_{\ell(\gamma)}^{n_{\gamma}}\right),
    \overset{\longleftarrow}{\underset{\gamma \in \Phi^{+}}{\prod}} \Psi^{-1}\left (\bar{R}_{\ell(\gamma)}^{m_{\gamma}} \right) 
    \right)} = 
  C_{|\ell|} \ol{\left( \wtd{R}_{\ell},\bar{R}_{w} \right)}.
\end{equation}
The rightmost term in~\eqref{eq:pairing_reduction} is zero unless $\ell = w$, i.e.\ $n_{\gamma} = m_{\gamma}$ for 
all $\gamma \in \Phi^{+}$, due to the first part of Theorem~\ref{thm:dual_bases}. Moreover, if $\ell=w$, then 
according to the second part of Theorem~\ref{thm:dual_bases}, we have
\[
  (\wtd{R}_{\ell},\bar{R}_{w}) = 
  \prod_{\gamma \in \Phi^{+}} \left( [m_{\gamma}]_{r_{\gamma},s_{\gamma}}! r_{\gamma}^{-\frac{1}{2}m_{\gamma}(m_{\gamma} - 1)}
    (R_{\ell(\gamma)},\bar{R}_{\ell(\gamma)})^{m_{\gamma}} \right),
\]
and therefore 
\begin{align*}
  C_{|\ell|} \ol{\left( \wtd{R}_{\ell},\bar{R}_{w} \right)} 
  &= \prod_{\gamma \in \Phi^{+}} \left( [m_{\gamma}]_{r_{\gamma},s_{\gamma}}! s_{\gamma}^{-\frac{1}{2}m_{\gamma}(m_{\gamma} - 1)} 
   C_{\gamma}^{m_{\gamma}} \ol{(R_{\ell(\gamma)},\bar{R}_{\ell(\gamma)})}^{m_{\gamma}} \right) \\  
  &= \prod_{\gamma \in \Phi^{+}} \left( [m_{\gamma}]_{r_{\gamma},s_{\gamma}}! s_{\gamma}^{-\frac{1}{2}m_{\gamma}(m_{\gamma} - 1)}
   (f_{\gamma},e_{\gamma})^{m_{\gamma}}_{H} \right),
\end{align*}
where the last equality follows from 
\begin{equation}\label{eq:aux-pairing-matching}
  (f_{\gamma},e_{\gamma})_H=C_{\gamma} \ol{(R_{\ell(\gamma)},\bar{R}_{\ell(\gamma)})}, 
\end{equation}
due to~\eqref{eq:pairing_reduction}, a corollary of Theorem~\ref{thm:pairing_comparison}. 
This completes the proof.
\end{proof}

By above theorem, the Hopf pairing $(\cdot,\cdot)_H$ is completely determined by nonzero constants 
$\{(f_\gamma,e_\gamma)_H\}_{\gamma\in \Phi^+}$. Our second main result is the explicit evaluation of these constants, 
recovering~\cite[(5.21)--(5.24)]{MT}. 
We shall use the notation for the positive roots that was introduced in~\eqref{eq:root_notation_A}--\eqref{eq:root_notation_D}.

\begin{theorem}\label{thm:main-2} 
(1) In type $A_n$ (that is, $\fg = \frak{sl}_{n+1}$), we have 
\[
  (f_{\gamma_{ij}},e_{\gamma_{ij}})_{H} = \frac{1}{s - r} \qquad \text{for}\quad 1 \le i \le j \le n.
\]

\noindent
(2) In type $B_n$ (that is, $\fg = \frak{so}_{2n+1}$), we have
\begin{align*}
  (f_{\gamma_{ij}},e_{\gamma_{ij}})_{H} &= \frac{1}{s^{2} - r^{2}} \qquad \text{for}\quad 1 \le i \le j < n, \\
  (f_{\gamma_{in}},e_{\gamma_{in}})_{H} &= \frac{1}{s - r} \qquad \text{for}\quad 1 \le i \le n, \\
  (f_{\beta_{ij}},e_{\beta_{ij}})_{H} &= \frac{[2]_{r,s}^{2}(rs)^{2(j - n)}}{s^{2} - r^{2}} \qquad \text{for}\quad 1 \le i < j \le n.
\end{align*}

\noindent
(3) In type $C_n$ (that is, $\fg = \frak{sp}_{2n}$), we have 
\begin{align*}
  (f_{\gamma_{ij}},e_{\gamma_{ij}})_{H} &= \frac{1}{s - r} 
    \qquad \text{for}\quad 1 \le i \le j \le n\quad \text{with}\quad (i,j) \neq (n,n), \\
  (f_{\gamma_{nn}},e_{\gamma_{nn}})_{H} &= \frac{1}{s^{2} - r^{2}}, \\
  (f_{\beta_{ii}},e_{\beta_{ii}})_{H} &= \frac{[2]_{r,s}^{2}}{s^{2} - r^{2}} \qquad \text{for}\quad 1 \le i < n, \\
  (f_{\beta_{ij}},e_{\beta_{ij}})_{H} &= \frac{(rs)^{j - n}}{s - r} \qquad \text{for}\quad 1 \le i < j < n.
\end{align*}

\noindent
(4) In type $D_n$ (that is, $\fg = \frak{so}_{2n}$), we have 
\begin{align*}
  (f_{\gamma_{ij}},e_{\gamma_{ij}})_{H} &= \frac{1}{s -r} \qquad \text{for} \qquad 1 \le i \le j < n, \\
  (f_{\beta_{ij}},e_{\beta_{ij}})_{H} &= \frac{(rs)^{j - n}}{s - r} \qquad \text{for}\quad 1 \le i < j \le n.
\end{align*}
\end{theorem}

\begin{proof}
(1) As $\gamma_{ij} = \alpha_i + \dots +\alpha_j$ for $1\le i\le j \le n$, combining Corollary~\ref{cor:pairing_A} 
and~\eqref{eq:aux-pairing-matching}, we obtain:  
\[
  (f_{\gamma_{ij}},e_{\gamma_{ij}})_{H} = 
  \frac{1}{(s - r)^{j - i +1}}\ol{\left( R_{\ell(\gamma_{ij})},\bar{R}_{\ell(\gamma_{ij})} \right)} = 
  \frac{1}{(s - r)^{j-i+1}} (s - r)^{j - i} = \frac{1}{s - r}.
\]

\noindent
(2) The computation for the roots $\gamma_{ij}$ with $1 \le i \le j < n$ is the same as the one above. 
For the roots $\gamma_{in}=\alpha_i + \dots + \alpha_n$, combining Corollary~\ref{cor:pairing_B}(1) and~\eqref{eq:aux-pairing-matching}, 
we get:
\[
  (f_{\gamma_{in}},e_{\gamma_{in}})_{H} = 
  \frac{1}{(s^{2} - r^{2})^{n -i}}\cdot \frac{1}{s - r}\cdot 
    \ol{\left( R_{\ell(\gamma_{in})},\bar{R}_{\ell(\gamma_{in})} \right)} = 
  \frac{1}{s - r}.
\]
For the roots $\beta_{ij}=\alpha_i + \dots + \alpha_{j-1} + 2\alpha_j + \dots + 2\alpha_n$, combining 
Corollary~\ref{cor:pairing_B}(2) and~\eqref{eq:aux-pairing-matching}, we obtain:
\[
  (f_{\beta_{ij}},e_{\beta_{ij}})_{H} = 
  \frac{1}{(s^{2} - r^{2})^{2n - i -j}}\cdot \frac{1}{(s - r)^{2}} \cdot 
    \ol{\left( R_{\ell(\beta_{ij})},\bar{R}_{\ell(\beta_{ij})} \right)} = 
  \frac{s^{2} - r^{2}}{(s - r)^{2}}(rs)^{2(j - n)} = \frac{[2]_{r,s}^{2}(rs)^{2(j - n)}}{s^{2} - r^{2}}.
\]

\noindent
(3) We shall only carry out the verification for the roots $\beta_{ii}$, since the other formulas are proved as the ones above. 
For the roots $\beta_{ii} = 2\alpha_i + \dots + 2\alpha_{n-1} + \alpha_n$, combining Corollary~\ref{cor:pairing_C}(4) 
and~\eqref{eq:aux-pairing-matching}, we get:
\[
  (f_{\beta_{ii}},e_{\beta_{ii}})_{H} = 
  \frac{1}{(s - r)^{2n - 2i}} \cdot \frac{1}{s^{2} - r^{2}} \cdot 
    \ol{\left( R_{\ell(\beta_{ii})},\bar{R}_{\ell(\beta_{ii})} \right)} = 
  \frac{1}{s - r}(r + s) = \frac{(r + s)^{2}}{s^{2} - r^{2}} = \frac{[2]_{r,s}^{2}}{s^{2} - r^{2}}.
\]

\noindent
(4) The computation for the roots $\gamma_{ij}$ with $1 \le i \le j < n$ is the same as the one in (1). 
For the roots $\beta_{in} = \alpha_{i} + \dots + \alpha_{n-2} + \alpha_{n}$, combining Corollary~\ref{cor:pairing_D}(2) 
and~\eqref{eq:aux-pairing-matching} gives us
\[
  (f_{\beta_{in}},e_{\beta_{in}})_{H} = 
  \frac{1}{(s - r)^{n - i}} \cdot \ol{\left (R_{\ell(\beta_{in})},\bar{R}_{\ell(\beta_{in})}\right )} = 
  \frac{1}{s - r}.
\]
For the roots $\beta_{i,n-1} = \alpha_{i} + \dots + \alpha_{n}$ with $1 \le i < n-1$, as well as the roots 
$\beta_{ij} = \alpha_{i} + \dots + \alpha_{j-1} + 2\alpha_{j} + \dots + 2\alpha_{n-2} + \alpha_{n-1} + \alpha_{n}$ 
with $1 \le i < j < n-1$, combining Corollary~\ref{cor:pairing_D}(3) and~\eqref{eq:aux-pairing-matching} yields:
\[
  (f_{\beta_{ij}},e_{\beta_{ij}})_{H} = 
  \frac{1}{(s - r)^{2n - i - j}} \cdot \ol{\left (R_{\ell(\beta_{ij})},\bar{R}_{\ell(\beta_{ij})}\right )} = 
  \frac{(rs)^{j - n}}{s - r}.
\]

This completes the proof.
\end{proof}


\appendix

\section{General Pairing Formulas for Root Vectors}\label{sec:app_pairing_formulas}

In this Appendix, we derive a formula for the pairing $(R_{\ell},\bar{R}_{\ell})$ for certain types of dominant Lyndon words $\ell \in \cal{L}^{+}$ that is valid for any ordering of $I = \{1,\ldots ,n\}$. The types of dominant Lyndon words $\ell$ that we shall mainly concern ourselves with in this Appendix are those whose \underline{first letter occurs exactly once}. We note that this is equivalent to saying that every left factor of $\ell$ is also Lyndon, so in particular, the costandard factorization of $\ell$ has the form $\ell = \ell_{1}i$, where $i \in I$ is a single letter.

To prove the main result of this Appendix, we shall need a few lemmas on root systems. As in the rest of the paper, $\Phi$ is an irreducible reduced root system with an ordered set of simple roots $\Pi = \{\alpha_{1},\ldots ,\alpha_{n}\}$. 

First, we remind the reader of some basic facts about root systems that we shall use frequently.

\begin{lemma}\label{lem:basic_facts} 
(1) For all $\alpha,\beta \in \Pi$, $\alpha \neq \beta$, we have $(\alpha,\beta) \le 0$.

\medskip
\noindent
(2) If $\alpha,\beta \in \Phi$ and $\alpha \neq \pm \beta$, then $\alpha + \beta \in \Phi$ whenever $(\alpha,\beta) < 0$, and $\alpha - \beta \in \Phi$ whenever $(\alpha,\beta) > 0$.
\end{lemma}

\begin{lemma}\label{lem:diamond} 
If $\gamma \in \Phi$, and for distinct $\alpha,\beta \in \Pi$, we have $\gamma + \alpha,\gamma + \beta \in \Phi$, then $\gamma + \alpha + \beta \in \Phi \cup \{0\}$. Likewise, if $\gamma - \alpha,\gamma - \beta \in \Phi$, then $\gamma - \alpha - \beta \in \Phi \cup \{0\}$.
\end{lemma}

\begin{proof} 
If $\gamma + \alpha$ is proportional to $\beta$, then we must have $\gamma + \alpha = -\beta$, so that $\gamma + \alpha + \beta = 0$. Similarly, if $\gamma + \beta$ is proportional to $\alpha$, we have $\gamma + \beta + \alpha = 0$. Thus, we may assume that neither of those two cases hold. Then if $\gamma + \alpha + \beta \notin \Phi$, we must have $(\gamma + \alpha,\beta) \ge 0$ and $(\gamma + \beta,\alpha) \ge 0$. This implies that $(\gamma,\beta) + (\gamma,\alpha) \ge -2(\alpha,\beta)$. But since $\gamma + \beta - (\gamma + \alpha) = \beta - \alpha \notin \Phi$, we must have $(\gamma + \beta,\gamma + \alpha) \le 0$. Combining this with the previous inequality yields
\[
  0 \ge (\gamma,\gamma) + (\gamma,\beta) + (\gamma,\alpha) + (\alpha,\beta) \ge (\gamma,\gamma) - (\alpha,\beta).
\]
But $(\alpha,\beta) \le 0$ and $(\gamma,\gamma) > 0$. Thus, we have a contradiction. 

The second claim follows by applying the first one to $-\gamma$.
\end{proof}

\begin{lemma}\label{lem:cube_completion} 
Let $\gamma \in \Phi^{+}$, and suppose that for some $i \neq j$, $\beta = \gamma - \alpha_{i} \in \Phi^{+}$ and $\beta' = \gamma - \alpha_{j} \in \Phi^{+}$. Suppose further that for some $m \le \hgt(\gamma) - 3$, there is a sequence $\alpha_{k_{1}},\ldots ,\alpha_{k_{m}} \in \Pi$ such that $\beta - \alpha_{k_{1}} - \ldots - \alpha_{k_{t}} \in \Phi^{+}$ and $\beta' - \alpha_{k_{1}} - \ldots - \alpha_{k_{t}} \in \Phi^{+}$ whenever $0 \le t \le m$. Then $\gamma - \alpha_{k_{1}} - \ldots - \alpha_{k_{t}} \in \Phi^{+}$ whenever $0 \le t \le m$.
\end{lemma}

\begin{proof} 
We proceed by induction on $t$. If $t = 0$, then the assertion is obvious. For the induction step, suppose that $\gamma_{p} = \gamma - \alpha_{k_{1}} - \ldots - \alpha_{k_{p}} \in \Phi^{+}$ for all $p < t$. Let $\beta_{p} = \beta - \alpha_{k_{1}} - \ldots - \alpha_{k_{p}}$ and $\beta_{p}' = \beta - \alpha_{k_{1}} - \ldots - \alpha_{k_{p}}$; by assumption, both of these are positive roots. Moreover, $\beta_{t-1} = \gamma_{t-1} - \alpha_{i}$ and $\beta_{t-1}' = \gamma_{t - 1} - \alpha_{j}$. Thus, Lemma~\ref{lem:diamond} implies that $\gamma_{t - 1}' = \gamma_{t - 1} - \alpha_{i} - \alpha_{j} \in \Phi^{+}$ as well. Now, because $i \neq j$, we have $k_{t} \neq i$ or $k_{t} \neq j$. If $k_{t} \neq i$, then since $\beta_{t - 1}' - \alpha_{i} = \gamma_{t - 1}' \in \Phi^{+}$ and $\beta_{t - 1}' - \alpha_{k_{t}} = \beta_{t}' \in \Phi^{+}$, we conclude from Lemma~\ref{lem:diamond} (and the assumption that $m \le \hgt(\gamma) - 3$) that $\gamma_{t}' = \gamma_{t - 1} - \alpha_{i} - \alpha_{j} - \alpha_{k_{t}} \in \Phi^{+}$. Likewise, if $k_{t} \neq j$, then since $\beta_{t- 1} - \alpha_{j} = \gamma_{t-1}' \in \Phi^{+}$ and $\beta_{t - 1} - \alpha_{k_{t}} = \beta_{t} \in \Phi^{+}$, we again conclude that $\gamma_{t}' = \gamma_{t - 1} - \alpha_{i} - \alpha_{j} - \alpha_{k_{t}} \in \Phi^{+}$. Finally, since $\gamma_{t}' + \alpha_{i} = \beta_{t}' \in \Phi^{+}$ and $\gamma_{t}' + \alpha_{j} = \beta_{t} \in \Phi^{+}$, we have $\gamma_{t}' + \alpha_{i} + \alpha_{j} = \gamma_{t} \in \Phi^{+}$ by Lemma~\ref{lem:diamond}, which completes the proof.
\end{proof}

\begin{lemma}\label{lem:cover_number_bound} 
Let $\gamma \in \Phi^{+}$ be a positive root. Then there are at most three distinct simple roots $\alpha \in \Pi$ such that $\gamma - \alpha \in \Phi^{+}$.
\end{lemma}

\begin{proof} 
Suppose otherwise, and let $\alpha_{i_{1}},\alpha_{i_{2}},\alpha_{i_{3}},\alpha_{i_{4}} \in \Pi$ be distinct simple roots such that $\gamma - \alpha_{i_{k}} \in \Phi^{+}$ for each $k = 1,2,3,4$. Then by Lemma~\ref{lem:diamond}, $\gamma - \alpha_{i_{k}} - \alpha_{i_{l}} \in \Phi^{+}$ whenever $k \neq l$. But the element $\gamma' = 2\gamma - \alpha_{i_{1}} - \alpha_{i_{2}} - \alpha_{i_{3}} - \alpha_{i_{4}}$ is nonzero, so we have
\[
  0 < (\gamma',\gamma') = ((\gamma - \alpha_{i_{1}} - \alpha_{i_{2}}) + (\gamma - \alpha_{i_{3}} - \alpha_{i_{4}}), (\gamma - \alpha_{i_{1}} - \alpha_{i_{3}}) + (\gamma - \alpha_{i_{2}} - \alpha_{i_{4}})).
\]
Then at least one of the values
\begin{align*}
  &(\gamma - \alpha_{i_{1}} - \alpha_{i_{2}},\gamma - \alpha_{i_{1}} - \alpha_{i_{3}}), & &(\gamma - \alpha_{i_{1}} - \alpha_{i_{2}},\gamma - \alpha_{i_{2}} - \alpha_{i_{4}}), \\
  &(\gamma - \alpha_{i_{3}} - \alpha_{i_{4}},\gamma - \alpha_{i_{1}} - \alpha_{i_{3}}), & &(\gamma - \alpha_{i_{3}} - \alpha_{i_{4}},\gamma - \alpha_{i_{2}} - \alpha_{i_{4}}), 
\end{align*}
must be positive; we can assume without loss of generality that $(\gamma - \alpha_{i_{1}} - \alpha_{i_{2}},\gamma - \alpha_{i_{1}} - \alpha_{i_{3}}) > 0$. But $(\gamma - \alpha_{i_{1}} - \alpha_{i_{2}}) - (\gamma - \alpha_{i_{1}} - \alpha_{i_{3}}) = \alpha_{i_{3}} - \alpha_{i_{2}} \notin \Phi$, which contradicts Lemma~\ref{lem:basic_facts}.
\end{proof}

We are now ready to prove technical lemmas that we shall need for the proofs of Theorems~\ref{thm:root_vector_support} and~\ref{thm:pairing_formula}. In the proofs below, we will make frequent use of Leclerc's algorithm; the reader should refer to Proposition~\ref{prop:leclerc_algorithm} for a reminder of its statement.

\begin{lemma}\label{lem:word_description} 
Suppose that $\ell = [i_{1}\ldots i_{d}]$ is a dominant Lyndon word such that $i_{1}$ occurs exactly once. Suppose that $j \neq i_{1},i_{d}$ and $|\ell| - \alpha_{j} \in \Phi^{+}$. Then if $k$ is the smallest positive integer such that $i_{d - k} = j$, we have $\ell(|\ell| - \alpha_{j}) = [i_{1}\ldots i_{d - k - 1}i_{d - k + 1}\ldots i_{d}]$.
\end{lemma}

\begin{proof} 
We first consider the case that $k = 1$. Let $\ell' = \ell(|\ell| - \alpha_{j})$, and suppose that the Lemma does not hold. Since $j \neq i_{1}$, the first letter of $\ell'$ occurs exactly once, and hence the costandard factorization of $\ell'$ is $\ell' = \ell_{1}'h$ for some $h \in I$. Because every left factor of $\ell$ is a dominant Lyndon word, we have $\ell(|\ell| - \alpha_{i_{d}} - \alpha_{i_{d-1}}) = [i_{1}\ldots i_{d - 2}]$, and therefore by Leclerc's algorithm, $[i_{1}\ldots i_{d - 2}i_{d}] < \ell_{1}'h$, which forces $[i_{1}\ldots i_{d - 2}] < \ell_{1}'$. But using Leclerc's algorithm again, we find that $\ell_{1}'hi_{d-1} < [i_{1}\ldots i_{d}]$, which combined with the previous inequality yields 
\[
  \ell_{1}'hi_{d-1} < [i_{1}\ldots i_{d}] < \ell_{1}',
\]
a contradiction.

Now, we proceed by induction on the length of $\ell$. If $\ell$ has length $3$ (every word satisfying the above assumptions has length at least $3$), then we have $\ell = [i_{1}i_{2}i_{3}]$ with $j = i_{2}$, so this case follows from the first part of the proof.

For the induction step, suppose that the length of $\ell$ is $d > 3$, and that the Lemma holds for all $\ell$ of smaller length. We can also assume that $k \ge 2$. For $1 \le t \le d$, let $\ell_{t} = [i_{1}\ldots i_{t}]$. Note that by the choice of $k$, applying Lemma~\ref{lem:diamond} to $|\ell_{t+1}|$ shows that $|\ell_{t}| - \alpha_{j} \in \Phi^{+}$ whenever $d - k < t < d$ (by a descending induction in $t$). For each $d - k < t \leq d$, set $\ell_{t}' = \ell(|\ell_{t}| - \alpha_{j})$. Then the word $\ell_{d - 1}$ satisfies the assumptions of the Lemma with $i_{d - 1}$ in place of $i_{d}$, so by induction, we have $\ell_{d - 1}' = [i_{1}\ldots i_{d - k -1}i_{d - k + 1}\ldots i_{d - 1}]$.

Now, let $\ell_{d}' = [j_{1}\ldots j_{d - 1}]$. Because $i_{1}$ still occurs exactly once in $\ell_{d}'$, and it is the smallest letter of $\ell_{d}'$, it follows that every left factor of $\ell_{d}'$ is also dominant Lyndon. Thus, if $j_{d - 1} = i_{d}$, then we must have $[j_{1}\ldots j_{d - 2}] = \ell_{d - 1}'$, so in this case we are done. Therefore we can assume that $j_{d - 1} \neq i_{d}$. But then we can apply the induction hypothesis to $\ell_{d}'$ with $i_{d}$ in place of $j$ and $j_{d - 1}$ in place of $i_{d}$. If $e$ is the smallest integer such that $j_{d - e} = i_{d}$, then since $\ell_{d-1}' = \ell(|\ell_{d}'| - \alpha_{i_{d}})$, the induction hypothesis yields
\[
  [i_{1}\ldots i_{d - k - 1}i_{d - k + 1}\ldots i_{d - 1}] = \ell_{d-1}' = [j_{1}\ldots j_{d - e - 1}j_{d - e + 1}\ldots j_{d - 1}].
\]
We now have three cases to consider: $k < e$, $k > e$, and $k = e$. Before proceeding, we introduce the notation $w_{t} = [j_{1}\ldots j_{t}]$ for $1 \le t \le d-1$ (so in particular, $w_{d - 1} = \ell_{d}'$), which we shall use in each of the three cases below.

Suppose first that $k < e$. Then $j_{p} = i_{p}$ for $1 \le p \le d - e - 1$ and $d - k + 1 \le p \le d - 1$, and $j_{p + 1} = i_{p}$ for $d - e \le p \le d - k - 1$. This implies that
\[
  \ell_{d}' = w_{d - 1} = [i_{1}\ldots i_{d - e -1}i_{d}i_{d - e}i_{d - e + 1}\ldots i_{d - k - 1}i_{d - k  +1}\ldots i_{d-1}].
\]
Now, every left factor of $\ell_{d}'$ is dominant Lyndon, so in particular, $[i_{1}\ldots i_{d - e - 1}i_{d}i_{d - e}]$ is a dominant Lyndon word. On the other hand, we also know that $[i_{1}\ldots i_{d-e}]$ is dominant Lyndon, and therefore Leclerc's algorithm implies that $[i_{1}\ldots i_{d -e}i_{d}] < [i_{1}\ldots i_{d - e - 1}i_{d}i_{d - e}]$, i.e. $i_{d - e} < i_{d}$. However, this implies that
\[
  w_{d-1}j > [i_{1}\ldots i_{d}] = \ell,
\]
contradicting Leclerc's algorithm.

Next, suppose that $k > e$. Then $j_{p} = i_{p}$ for $1 \le p \le d - k - 1$ and $d - e + 1
 \le p \le d- 1$, and $j_{p} = i_{p + 1}$ for $d - k \le p \le d - e - 1$. Thus, it follows that in this case 
\[
  \ell_{d}' = w_{d-1} = [i_{1}\ldots i_{d - k -1}i_{d - k + 1}\ldots i_{d - e}i_{d}i_{d - e + 1}\ldots i_{d - 1}].
\]
Because every left factor of $\ell_{d}'$ is Lyndon, $[i_{1}\ldots i_{d - k - 1}i_{d - k + 1}\ldots i_{d - e}i_{d}i_{d - e + 1}]$ is a dominant Lyndon word. But so is $[i_{1}\ldots i_{d - k - 1}i_{d - k + 1}\ldots i_{d - e}i_{d - e + 1}]$ (it is a left factor of $\ell_{d - 1}'$), so it follows from Leclerc's algorithm that 
\[
  [i_{1}\ldots i_{d - k - 1}i_{d - k + 1} \ldots i_{d - e}i_{d- e + 1}i_{d}] < [i_{1}\ldots i_{d- k - 1}i_{d - k + 1}\ldots  i_{d - e}i_{d}i_{d - e + 1}],
\] 
and hence $i_{d} > i_{d - e + 1}$.

Now, note that $\alpha_{i_{d - 1}},\alpha_{i_{d - 2}},\ldots ,\alpha_{i_{d -e + 1}}$ is a sequence of simple roots such that $|\ell_{d}'| - \alpha_{i_{d - 1}} -\ldots - \alpha_{i_{d - t}} \in \Phi^{+}$ and $|\ell_{d-1}| - \alpha_{i_{d - 1}} - \ldots - \alpha_{i_{d - t}} \in \Phi^{+}$ whenever $1 \le t \le e - 1$. Because $d - 1 > e$, we also have $e - 1 \le \hgt(|\ell|) - 3 = d-3$, and therefore we are in a position to apply Lemma~\ref{lem:cube_completion}, which tells us that for $1 \le t \le e - 1$, $|\ell| - \alpha_{i_{d-1}} - \ldots - \alpha_{i_{d - t}} \in \Phi^{+}$. Let $v_{d - t} = \ell(|\ell| - \alpha_{i_{d - 1}} - \ldots - \alpha_{i_{d- t}})$ for $1 \le t \le e-1$. Now, since $i_{1}$ still occurs exactly once in each $v_{d - t}$, we know that right factor of its costandard factorization is a single letter. Suppose that the last letter of $v_{d - e + 1}$ is $h \neq j$. Then since the dominant Lyndon word $[i_{1}\ldots i_{d - k - 1}i_{d - k + 1}\ldots i_{d - e}i_{d}]$ has degree $|v_{d - e + 1}| - \alpha_{j}$, the induction hypothesis implies $h = i_{d}$. Thus, the last letter of $v_{d - e + 1}$ must be either $i_{d}$ or $j$, so it follows that $v_{d - e + 1}$ is either $[i_{1}\ldots i_{d - k - 1}i_{d - k + 1}\ldots i_{d - e}i_{d}j]$ or $[i_{1}\ldots i_{d - e}i_{d}]$. Now, note that by Leclerc's algorithm, we have $\ell_{d}'j < \ell$, and therefore $i_{d - k + 1} < j$. This implies that $[i_{1}\ldots i_{d - k - 1}i_{d - k + 1}\ldots i_{d - e}i_{d}j] < [i_{1}\ldots i_{d - e}i_{d}]$, so by Leclerc's algorithm, $v_{d - e +1} = [i_{1}\ldots i_{d - e}i_{d}]$. Let us now determine $v_{d - e + 2}$. It follows from Lemma~\ref{lem:cover_number_bound} that $v_{d - e + 2}$ must be one of the following three words: $v_{d - e + 1}i_{d- e + 1} = [i_{1}\ldots i_{d - k - 1}ji_{d - k + 1}\ldots i_{d-e}i_{d}i_{d - e + 1}]$, $w_{d - e + 1}j = [i_{1}\ldots i_{d - k - 1}i_{d -k + 1}\ldots i_{d - e}i_{d}i_{d - e+ 1}j]$, or $\ell_{d - e + 1}i_{d} = [i_{1}\ldots i_{d - k - 1}ji_{d - k + 1}\ldots  i_{d-e}i_{d - e + 1}i_{d}]$. But then the inequalities $i_{d - k + 1} < j$ and $i_{d - e + 1} < i_{d}$ imply that
\[
  w_{d - e + 1}j < \ell_{d - e + 1}i_{d} < v_{d-e + 1}i_{d - e + 1},
\] 
so by Leclerc's algorithm, $v_{d - e + 2} = v_{d-e + 1}i_{d - e + 1} = [i_{1}\ldots i_{d - e}i_{d}i_{d - e + 1}]$. Continuing alike, we conclude that $v_{d - t} = [i_{1}\ldots i_{d - e}i_{d}i_{d - e + 1}\ldots i_{d - t - 1}]$ for $1 \le t \le e - 2$, so in particular, $v_{d - 1} = [i_{1}\ldots i_{d - e}i_{d}i_{d - e+ 1}\ldots i_{d - 2}]$. But then since $i_{d} > i_{d - e + 1}$, we have 
\[
  v_{d - 1}i_{d- 1} = [i_{1}\ldots i_{d - e}i_{d}i_{d - e + 1}\ldots i_{d - 1}] > [i_{1}\ldots i_{d}] = \ell,
\]
which violates Leclerc's algorithm.

Finally, we consider the case $k = e$. Then we have $j_{p} = i_{p}$ for all $p \neq d - e$, and therefore 
\[
  \ell_{d}' = w_{d-1} = [i_{1}\ldots i_{d - k-1}i_{d}i_{d - k + 1}\ldots i_{d-1}].
\]
Then $[i_{1}\ldots i_{d - k - 1}i_{d}i_{d - k + 1}]$ is Lyndon. But so is $[i_{1}\ldots i_{d - k - 1}i_{d - k + 1}]$, so as in the previous cases we conclude that $i_{d} > i_{d - k + 1}$.

Now, as in the $k > e$ case, we can use Lemma~\ref{lem:cube_completion} to conclude that $|\ell| - \alpha_{i_{d - 1}} - \ldots - \alpha_{i_{d- t}} \in \Phi^{+}$ for $1 \le t \le k - 1$, and as before we set $v_{d - t} = \ell(|\ell| - \alpha_{i_{d - 1}} - \ldots - \alpha_{i_{d - t}})$. Again, we know that the right factor in the costandard factorization of each $v_{d-t}$ is a single letter, and repeating the argument used in the previous case shows that the last letter of $v_{d - k + 1}$ is either $i_{d}$ or $j$. Therefore $v_{d-k+1}$ is equal to either $[i_{1}\ldots i_{d - k - 1}i_{d}j]$ or $[i_{1}\ldots i_{d - k - 1}ji_{d}]$. Since we must have $\ell_{d}'j < \ell$ by Leclerc's algorithm, it follows that $i_{d} < j$, and hence $v_{d - k + 1} = [i_{1}\ldots i_{d - k - 1}ji_{d}]$. Now we determine $v_{d -k + 2}$. Using Lemma~\ref{lem:cover_number_bound}, we see that there are three possibilities: $v_{d - k + 1}i_{d - k + 1} = [i_{1}\ldots i_{d - k - 1}ji_{d}i_{d - k + 1}]$, $w_{d - k + 1}j = [i_{1}\ldots i_{d -k - 1}i_{d}i_{d - k + 1}j]$, or $\ell_{d - k + 1}i_{d} = [i_{1}\ldots i_{d - k - 1}ji_{d - k + 1}i_{d}]$. Then using the inequalities $i_{d} < j$ and $i_{d} > i_{d - k + 1}$, we get 
\[
  w_{d - k + 1}j < \ell_{d - k + 1}i_{d} < v_{d - k + 1}i_{d - k +1},
\] 
so by Leclerc's algorithm, $v_{d - k + 2} = [i_{1}\ldots i_{d - k - 1}ji_{d}i_{d - k + 1}]$. As in the previous case, we can continue in this manner to obtain $v_{d - t} = [i_{1}\ldots i_{d - k - 1}ji_{d}i_{d - k + 1}\ldots i_{d - t -1}]$ for $1 \le t \le k - 2$. In particular, $v_{d - 1} = [i_{1}\ldots i_{d - k - 1}ji_{d}i_{d- k + 1}\ldots i_{d - 2}]$. But then since $i_{d} > i_{d - k + 1}$, we have 
\[
  v_{d - 1}i_{d-1} > [i_{1}\ldots i_{d}] = \ell,
\]
which violates Leclerc's algorithm.

We got a contradiction in all three cases, so the last letter of $\ell_{d}'$ must be $i_{d}$ and the proof is complete.
\end{proof}

For the rest of this Appendix, we shall only need the following slightly weaker corollary to Lemma~\ref{lem:word_description}:

\begin{cor}\label{cor:costandard_factorizations} 
Suppose that $\ell = [i_{1}\ldots i_{d}]$ is a dominant Lyndon word such that $i_{1}$ occurs exactly once. Suppose that $j \neq i_{1},i_{d}$, and $|\ell| - \alpha_{j} \in \Phi^{+}$. Let $\ell' = \ell(|\ell| - \alpha_{j})$ and $\ell'' = \ell(|\ell| - \alpha_{j} - \alpha_{i_{d}})$. Then the costandard factorization of $\ell'$ is $\ell' = \ell''i_{d}$. 
\end{cor}

\begin{proof} 
Since $i_{1}$ occurs once in $\ell'$, the above statement is equivalent to saying that the last letter of $\ell'$ is $i_{d}$. This is a direct consequence of Lemma~\ref{lem:word_description}.
\end{proof}

For the remainder of this Appendix, we will occasionally need to use the notion of the support of an element of $\cal{F}$. Given $x \in \cal{F}$ and its unique expression $x = \sum_{w \in \cal{W}}c_{w}w$ in terms of the basis $\cal{W}$ for $\cal{F}$, we define the \textbf{support} of $x$ to be the set
\begin{equation*}
  \supp(x) = \{w \in \cal{W} \,|\, c_{w} \neq 0\}.
\end{equation*}
Below, we will also make frequent use of the notation $\bb{C}(r,s)^{*} = \bb{C}(r,s) \setminus \{0\}$.

\begin{lemma}\label{lem:last_letter} 
Let $\ell = [i_{1}\ldots i_{d}]$ be a dominant Lyndon word such that $i_{1}$ occurs exactly once. Then for any $j \in I \setminus \{i_{1}\}$ such that $|\ell| - \alpha_{j} \in \Phi$, we have $\epsilon_{j}'(R_{\ell}) \in \bb{C}(r,s)^{*}R_{\ell'}$, where $\ell' = \ell(|\ell| - \alpha_{j})$. Furthermore, if $|\ell| - \alpha_{j} \notin \Phi$, then $\epsilon_{j}'(R_{\ell}) = 0$.
\end{lemma}

\begin{proof} 
We proceed by induction on the length of $\ell$. If $\ell = i_{1}$ has length $1$, then for all $j \neq i_{1}$, $\epsilon_{j}'(R_{\ell}) = 0$. If $\ell$ has length $2$, then $|\ell| - \alpha_{j} \notin \Phi$ if and only if $j \neq i_{1},i_{2}$, in which case we clearly have $\epsilon_{j}'(R_{\ell}) = 0$. If $j = i_{2}$, it is easy to verify that $\epsilon_{i_{2}}'(R_{\ell}) \in \bb{C}(r,s)^{*}i_{1} = \bb{C}(r,s)^{*}R_{i_{1}}$. 

Now suppose that $\ell$ has length at least $3$, and the Lemma holds for any dominant Lyndon word of smaller length that satisfies the assumptions. We know that the costandard factorization of $\ell$ is $\ell = \ell_{1}i_{d}$, where $\ell_{1} = [i_{1}\ldots i_{d - 1}]$. Suppose first that $j \neq i_{d},i_{1}$ and that $|\ell| - \alpha_{j} \in \Phi$. By Lemma~\ref{lem:diamond}, we also have $|\ell_{1}| - \alpha_{j} \in \Phi$, and therefore by Corollary~\ref{cor:costandard_factorizations}, the costandard factorization of $\ell' = \ell(|\ell| - \alpha_{j})$ is $\ell' = \ell''i_{d}$, where $\ell'' = \ell(|\ell_{1}| - \alpha_{j})$. This means that $R_{\ell'} = R_{\ell''} * i_{d} - (\omega_{i_{d}}',\omega_{|\ell''|})i_{d} * R_{\ell''}$. On the other hand, the induction hypothesis implies that $\epsilon_{j}'(R_{\ell_{1}}) = cR_{\ell''}$ for some $c \in \bb{C}(r,s)^{*}$, so we have 
\begin{align*}
  \epsilon_{j}'(R_{\ell}) &= \epsilon_{j}'(R_{\ell_{1}} * i_{d} - (\omega_{i_{d}}',\omega_{|\ell_{1}|}) i_{d} * R_{\ell_{1}}) = 
  \epsilon_{j}'(R_{\ell_{1}}) * i_{d} - (\omega_{i_{d}}',\omega_{|\ell_{1}|})(\omega_{i_{d}}',\omega_{j})^{-1}i_{d} * \epsilon_{j}'(R_{\ell_{1}})\\
  &= c (R_{\ell''} * i_{d} - (\omega_{i_{d}}',\omega_{|\ell''|})i_{d} * R_{\ell''}) = cR_{\ell'}.
\end{align*}
Now suppose that $j = i_{d}$. Since $\max(R_{\ell}) = \ell$ by Lemma~\ref{lem:root_vector_max}, we know that $\epsilon_{i_{d}}'(R_{\ell})$ is nonzero and contains the dominant Lyndon word $\ell_{1} = [i_{1}\ldots i_{d - 1}]$ in its support. Furthermore, if $w$ is any other word in the support of $\epsilon_{i_{d}}'(R_{\ell})$, then we know that $wi_{d} < \ell$, and hence we must have $w < \ell_{1}$. This shows that $\max(\epsilon_{i_{d}}'(R_{\ell})) = \ell_{1}$. However, if $\epsilon_{i_{d}}'(R_{\ell}) \notin \bb{C}(r,s)^{*}R_{\ell_{1}}$, then we can write 
\[
  \epsilon_{i_{d}}'(R_{\ell}) = \sum_{w \in \cal{W}^{+}}c_{w}R_{w},
\]
where each $w$ in the sum above has degree $|\ell_{1}|$, and $c_{w} \neq 0$ for some $w \in \cal{W}^{+} \setminus \{\ell_{1}\}$. But by Corollary~\ref{cor:max_Rw}, we have $\max(R_{w}) = w$, and therefore $\ell_{1} = \max(\epsilon_{i_{d}}'(R_{\ell})) \ge w$. This is a contradiction, because by Corollary~\ref{cor:smallest_word}, $\ell_{1}$ is the smallest dominant word of its degree. Therefore we must have $\epsilon_{i_{d}}'(R_{\ell}) \in \bb{C}(r,s)^{*}R_{\ell_{1}}$.

Finally, suppose that $|\ell| - \alpha_{j} \notin \Phi$. If $\epsilon_{j}'(R_{\ell}) \neq 0$, then we can write 
\[
  \epsilon_{j}'(R_{\ell}) = \sum_{w \in \cal{W}^{+}}c_{w}\wtd{R}_{w},
\]
where $|w| = |\ell| - \alpha_{j}$ whenever $c_{w} \neq 0$. Then by Theorem~\ref{thm:dual_bases},  
\[
  (\epsilon_{j}'(R_{\ell}),\bar{R}_{w}) = c_{w}(\wtd{R}_{w},\bar{R}_{w})
\]
for all $w \in \cal{W}^{+}$. On the other hand, combining Lemma~\ref{lem:symm-der-adjoint} and $\Psi \circ \partial_{i}' = \epsilon_{i}' \circ \Psi$ 
from the proof of Proposition~\ref{prop:psi_equals_upsilon}, we obtain $(x, y * j) = (\epsilon_{j}'(x),y)$ for all $x,y \in \cal{U}$. Hence, we have  
\[
  (\epsilon_{j}'(R_{\ell}),\bar{R}_{w}) =(R_{\ell}, \bar{R}_{w} * j).
\]
Now, upon transitioning to the basis $\{\epsilon_v\}$ via Proposition~\ref{prop:triangularity}, we get (as $\bar{\epsilon}_v=\epsilon_v$ for any $v$)
\[
  \bar{R}_{w} * j = \left (\epsilon_{w} + \sum_{v \in \cal{W}^{+}}^{v > w}\bar{\chi}_{v,w}\epsilon_{v}\right ) * j = 
  \epsilon_{wj} + \sum_{v \in \cal{W}^{+}}^{v > w}\bar{\chi}_{v,w}\epsilon_{vj},
\]
for some $\chi_{v,w} \in \bb{C}(r,s)$. Since $v > w$ clearly implies that $vj > wj$, it follows from Proposition~\ref{prop:triangularity} again that transitioning back to the Lyndon basis yields
\[
  \bar{R}_{w} * j = \sum_{u \in \cal{W}^{+}}^{u \ge wj}c_{u,wj}\bar{R}_{u},
\]  
for some $c_{u,wj} \in \bb{C}(r,s)$. However, if $c_{u,wj} \neq 0$, then we must have $|u| = |\ell|$, and since $\ell$ is dominant Lyndon, Corollary~\ref{cor:smallest_word} implies that $\ell \le u$. Suppose that $\ell = u$ for some $u \in \cal{W}^{+}$ such that $u \ge wj$. Then $\ell \ge wj$. But we know that $i_{1}$ occurs first in $\ell$, and it is also the smallest letter of both $\ell$ and $wj$, so this inequality implies that $w$ starts with $i_{1}$. But $w$ is not Lyndon, so we can write $w = w_{1}w_{2}\ldots w_{t}$ for some $t \ge 2$, where $w_{1} \ge w_{2} \ge \ldots \ge w_{t}$ and each $w_{k}$ is dominant Lyndon (see~\eqref{eq:canonical_dominant_fact}). Then each $w_{k}$ must start with $i_{1}$, which contradicts the fact that $i_{1}$ occurs only once in $w$. Therefore we have $\ell < u$ for all $u$ in the sum above with $c_{u,wj} \neq 0$. But this implies that 
\[
  c_{w}(\wtd{R}_{w},\bar{R}_{w}) = (\epsilon_{j}'(R_{\ell}),\bar{R}_{w}) = (R_{\ell},\bar{R}_{w} * j) = 0,
\]
for all $w \in \cal{W}^{+}$, which is a contradiction. Therefore $\epsilon_{j}'(R_{\ell}) = 0$. 
\end{proof}

\begin{theorem}\label{thm:root_vector_support} 
Let $\ell$ be a dominant Lyndon word such that the first letter occurs exactly once. Then 
\begin{equation}\label{eq:root_vector_support}
  \supp(R_{\ell}) = \Big \{w = [j_{1}\ldots j_{d}]\, \Big |\, |w| = |\ell|,\ \text{and}\  j_{1} \le j_{k}\ \text{and}\   
  \alpha_{j_{1}} + \ldots + \alpha_{j_{k}} \in \Phi^{+}\ \text{for all}\  1 \le k \le d \Big\}.
\end{equation}
\end{theorem}

\begin{proof}  
Denote the set on the right-hand side of~\eqref{eq:root_vector_support} by $\cal{A}_{\ell}$. Let us first show that $\supp(R_{\ell}) \subseteq \cal{A}_{\ell}$, which we shall do by induction on the length of $\ell$. If $\ell$ has length $1$, then the claim is obvious. Now suppose that $\ell = [i_{1}\ldots i_{d}]$ has length $d > 1$, and let $\ell_{1} = [i_{1}\ldots i_{d - 1}]$, so that $R_{\ell} = R_{\ell_{1}} * i_{d} - (\omega_{i_{d}}',\omega_{|\ell_{1}|})i_{d} * R_{\ell_{1}}$. We know from Lemma~\ref{lem:first_letter} that each $w \in \supp(R_{\ell})$ begins with $i_{1}$, and therefore $j_{1} \le j_{k}$ for all $k$ if $w = [j_{1}\ldots j_{d}] \in \supp(R_{\ell})$. Now suppose that $c_{w}w$ is a term in $R_{\ell}$, where $c_{w} \neq 0$ and $w = [j_{1}\ldots j_{d}]$. Then $c_{w}[j_{1}\ldots j_{d-1}]$ is a nonzero term in $\epsilon_{j_{d}}'(R_{\ell})$, so we must have $\alpha_{j_{1}} + \ldots + \alpha_{j_{d - 1}} \in \Phi^{+}$, because otherwise we would have a contradiction to the second part of Lemma~\ref{lem:last_letter}. Then by the first part of Lemma~\ref{lem:last_letter}, we have $\epsilon_{j_{d}}'(R_{\ell}) = cR_{\ell'}$ where $\ell' = \ell(|\ell| - \alpha_{j_{d}})$ and $c \in \bb{C}(r,s)^{*}$. Then $w' = [j_{1}\ldots j_{d - 1}] \in \supp(R_{\ell'})$, so by induction we have $\alpha_{j_{1}} + \ldots + \alpha_{j_{k}} \in \Phi^{+}$ whenever $1 \le k \le d - 1$. This completes the proof that $\supp(R_{\ell}) \subseteq \cal{A}_{\ell}$. 

To prove the other inclusion, we again proceed by induction on the length of $\ell$, with the base case being obvious. Let $w = [j_{1}\ldots j_{d}] \in \cal{A}_{\ell}$. Then $|\ell| - \alpha_{j_{d}} \in \Phi^{+}$, so by Lemma~\ref{lem:last_letter}, $\epsilon_{j_{d}}'(R_{\ell}) = cR_{\ell'}$ for some $c \in \bb{C}(r,s)^{*}$ where $\ell' = \ell(|\ell| - \alpha_{j_{d}})$. By induction, $\supp(R_{\ell'}) = \cal{A}_{\ell'}$, so in particular, $[j_{1}\ldots j_{d-1}] \in \supp(R_{\ell'}) = \supp(\epsilon_{j_{d}}'(R_{\ell}))$. Hence $[j_{1}\ldots j_{d}] \in \supp(R_{\ell})$.
\end{proof}

For the next theorem, we define the number
\begin{equation*}
  p_{\alpha,\beta} = \max \big\{k \ge 0 \,|\, \alpha - k\beta \in \Phi\big\}
\end{equation*}
associated to any pair $\alpha,\beta \in \Phi^{+}$. Note that if $\Phi$ is simply-laced and $\alpha+\beta \in \Phi^{+}$, then $p_{\alpha,\beta} = 0$.

\begin{theorem}\label{thm:pairing_formula} 
Let $\ell$ be a dominant Lyndon word such that the first letter of $\ell$ occurs exactly once, and let $\ell = \ell_{1}i$ be its costandard factorization. Then if $p_{|\ell_{1}|,\alpha_{i}} = 0$, we have
\[
  (R_{\ell},\bar{R}_{\ell}) = 
  \left( (\omega_{|\ell_{1}|}',\omega_{i})^{-1} - (\omega_{i}',\omega_{|\ell_{1}|}) \right) (R_{\ell_{1}},\bar{R}_{\ell_{1}}).
\]
\end{theorem}

\begin{proof} 
By the definition of $\bar{R}_{\ell}$, we have
\[
  \left (R_{\ell},\bar{R}_{\ell}\right ) = \left (R_{\ell},\bar{R}_{\ell_{1}} * i\right ) - (\omega_{|\ell_{1}|}',\omega_{i})^{-1} \left (R_{\ell},i * \bar{R}_{\ell_{1}}\right ).
\]
For the second term, note that, if we write $R_{\ell} = \sum c_{w}w$, the definition of $\Delta$ implies that 
\[
  \Delta(R_{\ell}) = \sum_{\substack{w_{1},w_{2} \in \cal{W},\\ w = w_{1}w_{2}}} c_{w}w_{2} \otimes w_{1}.
\]
By Lemma~\ref{lem:first_letter}, the first letter of $\ell$ must also be the first letter of $w_{1}$ whenever $w = w_{1}w_{2}$, $c_{w} \neq 0$, $w_1\ne \emptyset$. Since $i$ cannot be equal to the first letter of $\ell$, we find that 
\[
  (R_{\ell},i * \bar{R}_{\ell_{1}}) = (\Delta(R_{\ell}), \bar{R}_{\ell_{1}} \otimes i) = 0.
\]
Thus, 
\begin{equation}\label{eq:pairing_reduction-2}
  (R_{\ell},\bar{R}_{\ell}) = (R_{\ell},\bar{R}_{\ell_{1}} * i) = (\epsilon_{i}'(R_{\ell}),\bar{R}_{\ell_{1}}).
\end{equation}
Since $p_{|\ell_{1}|,\alpha_{i}} = 0$, Lemma~\ref{lem:last_letter} implies that $\epsilon_{i}'(R_{\ell_{1}}) = 0$. Therefore 
\[
  \epsilon_{i}'(R_{\ell}) = \epsilon_{i}'(R_{\ell_{1}} * i - (\omega_{i}',\omega_{|\ell_{1}|}) i * R_{\ell_{1}}) = 
  (\omega_{|\ell_{1}|}',\omega_{i})^{-1}R_{\ell_{1}} - (\omega_{i}',\omega_{|\ell_{1}|})R_{\ell_{1}}.
\]
Combining this with~\eqref{eq:pairing_reduction-2} completes the proof.
\end{proof}

Finally, let us describe how Theorem~\ref{thm:pairing_formula} translates to the Hopf pairing $(\cdot,\cdot)_{H}$ on $U_{r,s}(\fg)$.

\begin{cor}\label{cor:hopf_pairing_formula} 
Let $\gamma \in \Phi^{+}$ be a positive root such that the first letter of the dominant Lyndon word $\ell(\gamma)$ occurs exactly once. Let $\alpha,\beta \in \Phi^{+}$ be such that $\ell(\gamma) = \ell(\alpha)\ell(\beta)$ is the costandard factorization of $\ell(\gamma)$. Then if $p_{\alpha,\beta} = 0$, we have
\[
  (f_{\gamma},e_{\gamma})_{H} = \left( (\omega_{\beta}',\omega_{\alpha}) - (\omega_{\alpha}',\omega_{\beta})^{-1} \right) 
  (f_{\alpha},e_{\alpha})_{H}(f_{\beta},e_{\beta})_{H}.
\]
\end{cor}

\begin{proof} 
Note first that the conditions on $\ell(\gamma)$ imply that $\beta = \alpha_{i} \in \Pi$ for some $i$, and therefore $(f_{\beta},e_{\beta})_{H} = \frac{1}{s_{i} - r_{i}}$. Suppose that $\alpha = \sum_{j = 1}^{n}c_{j}\alpha_{j}$. Then, as in the proof of Theorem~\ref{thm:main-2}, combining~\eqref{eq:aux-pairing-matching} with Theorem~\ref{thm:pairing_formula} yields
\begin{align*}
  (f_{\gamma},e_{\gamma})_{H} &= \left( \prod_{j = 1}^{n}\frac{1}{(s_{j} - r_{j})^{c_{j}}} \right)\cdot 
    \frac{1}{s_{i} - r_{i}}\ol{\left (R_{\ell(\gamma)},\bar{R}_{\ell(\gamma)}\right )} \\
  &= \left( \prod_{j = 1}^{n}\frac{1}{(s_{j} - r_{j})^{c_{j}}} \right)\cdot 
    \frac{1}{s_{i} - r_{i}}\ol{\left( (\omega_{\alpha}',\omega_{\beta})^{-1} - (\omega_{\beta}',\omega_{\alpha}) \right)} 
    \ol{(R_{\ell(\alpha)},\bar{R}_{\ell(\alpha)})} \\
  &= \left( (\omega_{\beta}',\omega_{\alpha}) - (\omega_{\alpha}',\omega_{\beta})^{-1} \right) 
    (f_{\alpha},e_{\alpha})_{H}(f_{\beta},e_{\beta})_{H},
\end{align*}
as desired.
\end{proof}


\end{document}